\documentclass[preprint,12pt]{elsarticle}

\usepackage{amssymb}
\usepackage{amsmath}

\usepackage{amsfonts}
\usepackage{amsthm}
\usepackage{graphicx}
\usepackage{bm}
\usepackage{dsfont}

\journal{}

\numberwithin{equation}{section}
\newtheorem{theorem}{Theorem}[section]
\newtheorem{lemma}{Lemma}[section]
\newtheorem{corollary}{Corollary}[section]

\newtheorem{remark}{Remark}[section]
\usepackage{todonotes}
\let\vec\bm

\makeatletter
\newcommand{\mat}[1]{{\mathpalette\mat@{#1}}}
\newcommand{\mat@}[2]{%
  \begingroup
  \sbox\z@{$\m@th#1\underline{#2}$}%
  \dimen@=\dp\z@ \advance\dimen@ -2\mat@dimen{#1}%
  \dp\z@=\dimen@
  \sbox\z@{$\m@th\underline{\box\z@}$}%
  \box\z@
  \endgroup
}
\newcommand\mat@dimen[1]{%
  \fontdimen8
  \ifx#1\displaystyle\textfont\else
  \ifx#1\textstyle\textfont\else
  \ifx#1\scriptstyle\scriptfont\else
  \scriptscriptfont\fi\fi\fi 3
}
\makeatother
\let \vec \bm

\usepackage{verbatim}

\usepackage{caption}
\usepackage{subcaption}
\usepackage{fullpage} 
\usepackage{adjustbox}

\usepackage{hyperref}
\usepackage{cleveref}
\usepackage{mathtools}

\crefname{equation}{}{}
\Crefname{equation}{Equation}{Equations}
\crefname{lemma}{Lemma}{Lemmas}
\crefname{theorem}{Theorem}{Theorems}
\crefname{corollary}{Corollary}{Corollaries}
\crefname{figure}{Figure}{Figures}
\Crefname{figure}{Figure}{Figures}
\crefname{table}{Table}{Tables}
\Crefname{table}{Table}{Tables}
\crefname{remark}{Remark}{Remarks}

\usepackage{stmaryrd} 
\usepackage{pgf,tikz}
\usetikzlibrary{spy,matrix, calc, arrows,shapes,decorations}
\usepackage{tkz-euclide}
\usetikzlibrary{positioning,arrows,calc,decorations.markings,math,arrows.meta}
\usepackage{pgfplots} 
\usepgfplotslibrary{groupplots}
\usepackage{csvsimple} 
\usepackage{siunitx} 
\usepackage{tabularx}
\usepackage{booktabs}
\usepackage{multirow}
\usepackage[normalem]{ulem}
\useunder{\uline}{\ul}{}

\pgfplotscreateplotcyclelist{unfmixed1}{%
 teal,every mark/.append style={solid,fill=teal!80!black},mark=*,mark size=2.5pt\\
 red!70!white,every mark/.append style={solid,fill=red!80!black},mark=pentagon*,mark size=2.5pt\\
cyan!60!black,every mark/.append style={solid,fill=cyan!80!black},mark=diamond*,mark size=2.5pt\\
}

\pgfplotscreateplotcyclelist{unfmixed2}{%
teal,every mark/.append style={solid,fill=teal!80!black},mark=*,mark size=2.5pt,line width=1.15pt\\
orange,every mark/.append style={solid,fill=orange!80!black},mark=square*,mark size=2.5pt,line width=1.15pt\\
teal,every mark/.append style={solid,fill=teal!80!black},mark=*,mark size=2.5pt,dashed,line width=1.15pt\\
orange,every mark/.append style={solid,fill=orange!80!black},mark=square*,mark size=2.5pt,dashed,line width=1.15pt\\
}

\pgfplotscreateplotcyclelist{unfmixed3}{%
 orange,every mark/.append style={solid,fill=orange!80!black},mark=square*,mark size=2.5pt\\
 teal,every mark/.append style={solid,fill=teal!80!black},mark=*,mark size=2.5pt\\
 cyan!60!black,every mark/.append style={solid,fill=cyan!80!black},mark=diamond*,mark size=2.5pt\\
 red!70!white,every mark/.append style={solid,fill=red!80!black},mark=pentagon*,mark size=2.5pt\\
 green,every mark/.append style={solid,fill=green!80!black},mark=triangle*,mark size=2.5pt\\
}

\pgfplotscreateplotcyclelist{colors1}{%
teal,dashed,every mark/.append style={solid,fill=teal!80!black},mark=*,mark size=2.5pt\\
 red!70!white,dashed,every mark/.append style={solid,fill=red!80!black},mark=pentagon*,mark size=2.5pt\\
 green,dashed,every mark/.append style={solid,fill=green!80!black},mark=triangle*,mark size=2.5pt\\
cyan!60!black,dashed,every mark/.append style={solid,fill=cyan!80!black},mark=diamond*,mark size=2.5pt\\
 teal,every mark/.append style={solid,fill=teal!80!black},mark=*,mark size=2.5pt\\
 red!70!white,every mark/.append style={solid,fill=red!80!black},mark=pentagon*,mark size=2.5pt\\
 green,every mark/.append style={solid,fill=green!80!black},mark=triangle*,mark size=2.5pt\\
cyan!60!black,every mark/.append style={solid,fill=cyan!80!black},mark=diamond*,mark size=2.5pt\\
}

\pgfplotscreateplotcyclelist{colors2}{%
black,every mark/.append style={solid,fill=teal!80!black},mark=*,mark size=2.5pt\\
violet,every mark/.append style={solid,fill=violet!80!black},mark=square*,mark size=2.5pt\\
teal,every mark/.append style={solid,fill=teal!80!black},mark=*,mark size=2.5pt\\
red!70!white,every mark/.append style={solid,fill=red!80!black},mark=pentagon*,mark size=2.5pt\\
green,every mark/.append style={solid,fill=green!80!black},mark=triangle*,mark size=2.5pt\\
cyan!60!black,every mark/.append style={solid,fill=cyan!80!black},mark=diamond*,mark size=2.5pt\\
orange,every mark/.append style={solid,fill=orange!80!black},mark=square*,mark size=2.5pt\\
}

\pgfplotscreateplotcyclelist{colors3}{%
teal,every mark/.append style={solid,fill=teal!80!black},mark=*,mark size=2.5pt\\
orange,every mark/.append style={solid,fill=orange!80!black},mark=square*,mark size=2.5pt\\
cyan!60!black,every mark/.append style={solid,fill=cyan!80!black},mark=pentagon*,mark size=2.5pt\\
}

\pgfplotscreateplotcyclelist{paulcolors}{%
violet,every mark/.append style={solid,fill=violet},mark=square*,very thick,mark size=3pt\\%
teal,every mark/.append style={solid,fill=teal},mark=*,very thick,mark size=3pt\\%
orange,every mark/.append style={solid,fill=orange},mark=diamond*,very thick,mark size=3pt\\%
violet,densely dashed,every mark/.append style={solid,fill=violet},mark=square*,very thick,mark size=3pt\\%
teal,densely dashed,every mark/.append style={solid,fill=teal},mark=*,very thick,mark size=3pt\\%
orange,densely dashed,every mark/.append style={solid,fill=orange},mark=diamond*,very thick,mark size=3pt\\%
violet,dash dot,every mark/.append style={solid,fill=violet},mark=square*,very thick,mark size=3pt\\%
teal,dash dot,every mark/.append style={solid,fill=teal},mark=*,very thick,mark size=3pt\\%
orange,dash dot,every mark/.append style={solid,fill=orange},mark=diamond*,very thick,mark size=3pt\\%
}
\pgfplotscreateplotcyclelist{paulcolors2}{%
teal,every mark/.append style={solid,fill=teal},mark=*,very thick,mark size=3pt\\%
orange,every mark/.append style={solid,fill=orange},mark=diamond*,very thick,mark size=3pt\\%
teal,densely dashed,every mark/.append style={solid,fill=teal},mark=*,very thick,mark size=3pt\\%
orange,densely dashed,every mark/.append style={solid,fill=orange},mark=diamond*,very thick,mark size=3pt\\%
teal,dash dot,every mark/.append style={solid,fill=teal},mark=*,very thick,mark size=3pt\\%
orange,dash dot,every mark/.append style={solid,fill=orange},mark=diamond*,very thick,mark size=3pt\\%
}
\pgfplotscreateplotcyclelist{paulcolors1}{%
teal,every mark/.append style={solid,fill=teal},mark=*,very thick,mark size=3pt\\%
}
\pgfplotscreateplotcyclelist{paulcolorsfill}{%
teal,fill=teal,fill opacity=0.5,thick\\
orange,fill=orange,fill opacity=0.5,thick\\
violet,fill=violet,fill opacity=0.5,thick\\
}

\pgfplotscreateplotcyclelist{barcolors}{%
teal,fill=teal,fill opacity=0.5,thick\\
teal,fill=teal,fill opacity=0.5,thick,postaction={pattern=crosshatch dots,pattern color=teal}\\
orange,fill=orange,fill opacity=0.5,thick\\
orange,fill=orange,fill opacity=0.5,thick,postaction={pattern=crosshatch dots, pattern color=orange}\\
}

\pgfplotscreateplotcyclelist{barcolors2}{%
teal,fill=teal,fill opacity=0.5,thick\\
orange,fill=orange,fill opacity=0.5,thick\\
teal,fill=teal,fill opacity=0.5,thick,postaction={pattern=crosshatch dots,pattern color=teal}\\
orange,fill=orange,fill opacity=0.5,thick,postaction={pattern=crosshatch dots, pattern color=orange}\\
}

\pgfplotscreateplotcyclelist{barcolors3}{%
teal,fill=teal,fill opacity=0.7,thick,postaction={pattern=crosshatch dots,pattern color=teal}\\
orange,fill=orange,fill opacity=0.7,thick,postaction={pattern=crosshatch dots, pattern color=orange}\\
}

\pgfplotsset{
    discard if not/.style 2 args={
        x filter/.append code={
            \edef\tempa{\thisrow{#1}}
            \edef\tempb{#2}
            \ifx\tempa\tempb
            \else
                
            \fi
        }
    },
    my legend style/.style={
            legend entries={
                blue plot,
                red plot,
                green plot
            },
            legend style={
                at={([yshift=2pt]0,1)},
                anchor=south west,
            },
            legend columns=4,
    }
}
\pgfplotsset{compat=1.16}

\pgfplotsset{tick label style={font=\small},label style={font=\small},legend style={font=\small},}
\pgfplotsset{ width=.49\linewidth}

\usepackage{listings}
\usepackage[plain]{algorithm}

\definecolor{maroon}{cmyk}{0, 0.87, 0.68, 0.32}
\definecolor{halfgray}{gray}{0.55}
\definecolor{ipython_frame}{RGB}{207, 207, 207}
\definecolor{ipython_bg}{RGB}{247, 247, 247}
\definecolor{ipython_red}{RGB}{186, 33, 33}
\definecolor{ipython_green}{RGB}{0, 128, 0}
\definecolor{ipython_cyan}{RGB}{64, 128, 128}
\definecolor{ipython_purple}{RGB}{170, 34, 255}

\lstdefinelanguage{iPython}{
    morekeywords=[2]{abs,all,any,basestring,bin,bool,bytearray,callable,chr,classmethod,cmp,compile,complex,delattr,dict,dir,divmod,enumerate,eval,execfile,file,filter,float,format,frozenset,getattr,globals,hasattr,hash,help,hex,id,input,int,isinstance,issubclass,iter,len,list,locals,long,map,max,memoryview,min,next,object,oct,open,ord,pow,property,range,raw_input,reduce,reload,repr,reversed,round,set,setattr,slice,sorted,staticmethod,str,sum,super,tuple,type,unichr,unicode,vars,xrange,zip,apply,buffer,coerce,intern},%
    sensitive=true,%
    morecomment=[l]\#,%
    morestring=[b]',%
    morestring=[b]",%
    morestring=[s]{'''}{'''},
    morestring=[s]{"""}{"""},
    morestring=[s]{r'}{'},
    morestring=[s]{r"}{"},%
    morestring=[s]{r'''}{'''},%
    morestring=[s]{r"""}{"""},%
    morestring=[s]{u'}{'},
    morestring=[s]{u"}{"},%
    morestring=[s]{u'''}{'''},%
    morestring=[s]{u"""}{"""},%
    literate=
    {á}{{\'a}}1 {é}{{\'e}}1 {í}{{\'i}}1 {ó}{{\'o}}1 {ú}{{\'u}}1
    {Á}{{\'A}}1 {É}{{\'E}}1 {Í}{{\'I}}1 {Ó}{{\'O}}1 {Ú}{{\'U}}1
    {à}{{\`a}}1 {è}{{\`e}}1 {ì}{{\`i}}1 {ò}{{\`o}}1 {ù}{{\`u}}1
    {À}{{\`A}}1 {È}{{\'E}}1 {Ì}{{\`I}}1 {Ò}{{\`O}}1 {Ù}{{\`U}}1
    {ä}{{\"a}}1 {ë}{{\"e}}1 {ï}{{\"i}}1 {ö}{{\"o}}1 {ü}{{\"u}}1
    {Ä}{{\"A}}1 {Ë}{{\"E}}1 {Ï}{{\"I}}1 {Ö}{{\"O}}1 {Ü}{{\"U}}1
    {â}{{\^a}}1 {ê}{{\^e}}1 {î}{{\^i}}1 {ô}{{\^o}}1 {û}{{\^u}}1
    {Â}{{\^A}}1 {Ê}{{\^E}}1 {Î}{{\^I}}1 {Ô}{{\^O}}1 {Û}{{\^U}}1
    {œ}{{\oe}}1 {Œ}{{\OE}}1 {æ}{{\ae}}1 {Æ}{{\AE}}1 {ß}{{\ss}}1
    {ç}{{\c c}}1 {Ç}{{\c C}}1 {ø}{{\o}}1 {å}{{\r a}}1 {Å}{{\r A}}1
    {€}{{\EUR}}1 {£}{{\pounds}}1
    {^}{{{\color{ipython_purple}\^{}}}}1
    {=}{{{\color{ipython_purple}=}}}1
    {+}{{{\color{ipython_purple}+}}}1
    {*}{{{\color{ipython_purple}$^\ast$}}}1
    {/}{{{\color{ipython_purple}/}}}1
    {+=}{{{+=}}}1
    {-=}{{{-=}}}1
    {*=}{{{$^\ast$=}}}1
    {/=}{{{/=}}}1,
    literate=
    *{-}{{{\color{ipython_purple}-}}}1
     {?}{{{\color{ipython_purple}?}}}1,
    identifierstyle=\color{black}\ttfamily,
    commentstyle=\color{ipython_cyan}\ttfamily,
    stringstyle=\color{ipython_red}\ttfamily,
    keepspaces=true,
    showspaces=false,
    showstringspaces=false,
    rulecolor=\color{ipython_frame},
    framexleftmargin=0mm,
    numbers=left,
    numberstyle=\tiny\color{halfgray},
    numbersep=1mm,
    xleftmargin=1mm,
    basicstyle=\scriptsize,
    keywordstyle=\color{ipython_green}\ttfamily,
}

\lstdefinestyle{trefftzy}{
    language=iPython,
    emptylines=1,
    breaklines=true,
    basicstyle=\footnotesize\ttfamily\color{black},    
    moredelim=**[is][\color{teal}]{<}{>},
    moredelim=**[is][\color{purple}]{'}{'},
}

\renewcommand{\d}[1]{\,\mathrm{d}{#1}}

\let\div\undefined
\DeclareMathOperator{\div}{div}

\DeclareMathOperator{\tr}{tr}
\DeclareMathOperator{\sym}{sym}
\DeclareMathOperator{\anti}{skw}
\DeclareMathOperator{\curl}{curl}
\DeclareMathOperator{\Ra}{Ra}
\DeclareMathOperator{\symgrad}{\mat{\varepsilon}}

\begin{document}

\begin{frontmatter}



\title{Achieving Material Robustness via Symmetric Stress Finite Element 
Discretizations \tnoteref{t1}}

\tnotetext[t1]{This project has received funding through the UKRI Digital 
Research Infrastructure Programme through the Science and Technology Facilities 
Council's Computational Science Centre for Research
Communities (CoSeC). CP was supported by an appointment to the NRC Research
Associateship Program at the U.S. Naval Research Laboratory, 
administered by the Fellowships Office of the National Academies of Sciences, 
Engineering, and Medicine. 
UZ gratefully acknowledges the Erwin Schrödinger International Institute for 
Mathematics and Physics (ESI) for supporting the stay in Vienna where this 
paper was completed.
Distribution Statement A.  Approved for public release: distribution is unlimited.}


\author[Oxford]{Pablo Brubeck}
\author[NRL]{Charles Parker}
\author[Oxford]{Umberto Zerbinati} 
\affiliation[Oxford]{organization={University of Oxford},
            addressline={Mathematics Institute, Woodstock Road},
            city={Oxford},
            postcode={OX2 6GG},
            state={Oxfordshire},
            country={United Kingdom}}
\affiliation[NRL]{organization={U.S. Naval Research Laboratory},
            addressline={4555 Overlook Ave. S.W.},
            city={Washington},
            postcode={20375},
            state={DC},
            country={USA}}

\begin{abstract}
	When discretizing symmetric stress tensors in variational problems
	arising in continuum mechanics, one has to choose 
	how to enforce the symmetry of the stress tensor: (i) strongly
	by requiring the discrete tensors to be pointwise symmetric or 
	(ii) weakly by introducing a Lagrange multiplier. For 
	$H(\div)$-conforming finite element discretizations
	of Hellinger--Reissner elasticity and velocity--stress formulations
	of incompressible flow, where symmetry of the Cauchy stress tensor
	is tied to the conservation of angular momentum, we show that 
	this choice may substantially impact the accuracy of the numerical scheme.
	Through a series of benchmark problems featuring anisotropic constitutive 
	laws inspired by fiber reinforced material, liquid crystal polymer networks, 
	and polar fluids, we show that schemes enforcing symmetry weakly can yield 
	arbitrarily poor stress approximations 
	-- even for zero-stress configurations. However,
	schemes enforcing symmetry strongly deliver accurate stress approximations 
	independently of the constitutive law, a property we term 
	\emph{material robustness}.
	We present a unifying theory that rigorously explains this behavior.
\end{abstract}


\end{frontmatter}

\section{Introduction}
\label{sec:intro}

Variational formulations appearing in many applications involve 
spaces of symmetric tensors. Perhaps the most well-known is the 
Hellinger-Reissner formulation of elasticity\footnote{See \cite{Eugster2022} 
for an English translation of Hellinger's work \cite{Hellinger13}} 
\cite{Hellinger13,Reissner50,Reissner53}, 
where the Cauchy and second Piola-Kirchhoff stress tensors
are symmetric, corresponding to the preservation of angular momentum
\footnote{Conservation of angular momentum is not a direct consequence of the
symmetry of the Cauchy stress tensor alone as we will see in greater detail
later. Yet, for the specific constitutive relations of the
Hellinger--Reissner formulation, the symmetry of the Cauchy stress tensor
indeed corresponds to the conservation of angular momentum.}.
Other examples include the linearized strain field and the 
(right) Cauchy-Green tensors in intrinsic formulations of elasticity 
\cite{Ciarlet09}, the bending-moments tensor in a mixed formulation of 
Reissner-Mindlin plates \cite{Sky23}, the elasticity tensor in linear 
Cosserat elasticity \cite{Dziubek25,Neff09}, and 
the viscous stress tensor in incompressible flow \cite{Gopalakrishnan20}
and in multicomponent convection-diffusion \cite{Aznaran25}, to name a few. 

When designing discretizations for these problems, one must 
decide whether to enforce symmetry of the tensor strongly (pointwise) or 
weakly (e.g. with a Lagrange multiplier). For $H(\div)$-conforming 
finite element discretizations of symmetric tensors, 
schemes with weak symmetry  
are often preferred over those with strong symmetry in the literature, largely
owing to the perceived complexity of symmetric 
$H(\div)$-conforming tensor elements. Indeed,
many strongly symmetric tensor elements possess supersmoothness 
at mesh vertices and/or edges 
\cite{AdamsCockburn05,ArnoldAwanouWinther08,ArnoldWinther02,%
HuZhang15d2,HuZhang15d3} or 
are macroelements 
\cite{ChenHuang25,Gong23,Gopalakrishnan25,JohnsonMercier78,WatwoodHartz68}
or both \cite{Christiansen24}. 
In contrast, schemes with weak symmetry 
\cite{AmaraThomas79,ArnoldBrezziDouglas84,ArnoldFalkWinther07,%
BoffiBrezziFortin08,CockburnGopalakrishnanGuzman10,FarhloulFortin97,%
GopalakrishnanGuzman12,Stenberg88}
involve more familiar finite element 
spaces that are readily available in software packages. 
Generally, weak enforcement of 
symmetry is not equivalent to strong enforcement unless the finite 
element spaces are carefully chosen \cite{GopalakrishnanGuzman12}, but
quasioptimal error estimates for the Hellinger-Reissner problem
are available for most of the above schemes;
see also \cite{Lederer23,Lederer24} for a unified analysis of schemes with 
strong \cite{Lederer23} and weak symmetry \cite{Lederer24}.

We revisit the importance of strongly versus weakly enforcing the symmetry 
of the (Cauchy) stress tensor in the context of Hellinger-Reissner formulations 
of linear elasticity and stress-velocity formulations of incompressible flow. 
In these contexts, symmetry of the stress tensor
is intimately related to the preservation of angular momentum. Preserving 
physical constraints in discretizations can be crucial in applications. 
For example, in incompressible flow, discretizations that are 
pointwise divergence-free lead to so-called \textit{pressure-robust}
error estimates, while schemes that are only weakly divergence-free 
can produce arbitrarily poor velocity approximations to problems 
with a smooth flow (including no flow) --- see 
\cite{JohnLinkeMerdonNeilanRebholz17} for a review and 
\cref{sec:pressure-robustness} for an example. However, the importance
of the symmetry constraint in our setting appears to be unaddressed in the 
literature.

We develop benchmark problems with materials with a variety of  
constitutive laws, 
including anisotropic settings inspired by liquid crystal polymer networks 
and polar fluids, that demonstrate that schemes with weakly imposed symmetry
can produce arbitrarily poor stress approximations, even for zero stress 
configurations. However, most schemes with strongly imposed symmetry 
do not suffer from this phenomenon, regardless of the constitutive law, 
a property we call \textit{material robustness}. 
Fortunately, symmetric tensor elements
and other ``exotic'' finite elements have become increasingly available 
in open source software packages such as NGSolve 
\cite{SchoberNetgenl97,SchoberlNGSolve14,Sky23} and 
Firedrake \cite{Aznaran22,BrubeckKirby25,FiredrakeUserManual},
which both implement 
the 2D Hu-Zhang element \cite{HuZhang15d2},
while the 2D Arnold-Winther element \cite{ArnoldWinther02}
and the 2D and 3D Johnson-Mercier-K\v{r}\'{i}\v{z}ek element 
\cite{JohnsonMercier78,Krizek82,WatwoodHartz68} are also 
available in Firedrake (see \cite{BrubeckKirby25}
for a list of other available exotic elements).
We additionally develop an abstract
theory to unify the construction and behavior of these examples.
These insights advocate for the 
broader adoption of material robust methods
in computational continuum mechanics, especially in applications sensitive to 
rotational invariants.

The remainder of the paper is organized as follows.
In \cref{sec:general-setup}, we introduce a general setup for 
Hellinger-Reissner problems that we consider. 
We then present a series of numerical examples in \cref{sec:examples} to 
demonstrate the discrepancies of enforcing symmetry 
strongly versus weakly that motivate our notion of material robustness. 
Abstract analysis for general saddle point problems
is developed in \cref{sec:theory} to explain the behavior 
observed in the numerical examples. Finally, we present 
in \cref{sec:conclusions} a transient example to show how the observations 
in \cref{sec:examples} for static problems carry over to the time-dependent 
setting.

\section{General problem setup}
\label{sec:general-setup}

Consider the balance law for a generic continuum in 
Eulerian coordinates 
\cite[Result 5.5, Result 7, \& Exercise 5.7]{GonzalezStuart08} in an open 
domain $\Omega \subset \mathbb{R}^3$:
\begin{subequations}
	\label{eq:general_balance_law}
	\begin{align}
		\label{eq:balance}
		\partial_t \rho + \nabla \cdot (\rho \vec{u}) &= 0,\\
		\label{eq:linear_momentum}
		\rho \left(\partial_t \vec{u}  + \vec{u} \cdot \nabla \vec{u}\right) 
			- \nabla \cdot \mat{\sigma} &= \rho \vec{f},\\
		\label{eq:angular_momentum}
		\rho \left(\partial_t \vec{\eta} + \vec{u} \cdot \nabla \vec{\eta}\right) 
			- \nabla \cdot \mat{\zeta} &= \vec{\xi} + \rho \vec{\tau},
	\end{align}	
\end{subequations}
where $\rho$ is the density, $\vec{u}$ is the velocity, $\vec{\eta}$ is the 
angular momentum, $\mat{\sigma}$ is the 
Cauchy stress tensor, $\mat{\zeta}$ is the couple stress tensor 
\footnote{The couple stress tensor is the analog of the Cauchy stress 
tensor but with respect to body torque rather than force \cite{GonzalezStuart08}, 
in the same way that the angular momentum and the body torque are analogous to 
the velocity and the body force respectively.}, $\vec{f}$ is 
the body force, $\vec{\tau}$ is an external body torque and $\vec{\xi}$ is the 
vectorified form of the antisymmetric part of the stress tensor, i.e.
\begin{equation}
	\vec{\xi} = \begin{pmatrix}
		\sigma_{32} - \sigma_{23}, 
		& \sigma_{13} - \sigma_{31}, 
		& \sigma_{21} - \sigma_{12}
	\end{pmatrix}^{\top}.
\end{equation}
We remark that \eqref{eq:balance} expresses the conservation of mass, 
\eqref{eq:linear_momentum} the conservation of the total linear momentum, 
and \eqref{eq:angular_momentum} the conservation of the angular momentum.
To close the system, 
the balance laws \cref{eq:general_balance_law} are augmented with initial
and boundary conditions and constitutive laws relating the 
stress tensors $\mat{\sigma}$, $\mat{\zeta}$ to the kinematic variables
$\rho$, $\vec{u}$, and $\vec{\eta}$.

It is common to say that the angular momentum is conserved if and only if 
the antisymmetric part of the Cauchy stress tensor vanishes 
($\vec{\xi} \equiv \vec{0}$). 
In absence of body torques and couple stresses 
($\mat{\zeta} \equiv \mat{0}$ and $\vec{\tau} \equiv \vec{0}$), this is indeed 
true, as the angular momentum balance law \cref{eq:angular_momentum} reduces to
\begin{equation}
	\label{eq:angular_momentum_simplified}
	\rho \left(\partial_t \vec{\eta} + \vec{u} \cdot \nabla \vec{\eta}\right) 
	= \vec{\xi},
\end{equation}
and so angular momentum is conserved 
(i.e.~the material time derivative of $\vec{\eta}$ vanishes) if and only if 
$\vec{\xi} \equiv \vec{0}$ (i.e.~$\mat{\sigma} = \mat{\sigma}^{\top}$) --- 
see \cite[Exercise 5.7]{GonzalezStuart08}. In the presence of
couple stresses ($\mat{\zeta} \neq 0$) and/or external body torques 
($\vec{\tau} \neq \vec{0}$), the Cauchy stress tensor may be 
symmetric even though angular momentum is not conserved, or 
the angular momentum may be conserved even though the Cauchy stress tensor
is not symmetric. We proceed below in
a stationary linearized setting where body torques and couple stresses are 
absent, so the symmetry of the Cauchy stress tensor does indeed 
correspond to the conservation of angular momentum.

\subsection{PDE formulation of the linearized static problem}

Let $\Omega \subset \mathbb{R}^d$, $d \in \{2,3\}$, be a Lipschitz
polyhedral domain.
In the three examples below, we assume that the system is in equilibrium and
neglect the body torques and couple stresses. 
For flow problems, we linearize about an equilibrium velocity 
$\vec{u}_0$ with $|\vec{u}_0| \ll 1$ in $\Omega$,
and assume that $\rho \equiv 1$. In 
this case, the conservation of mass \cref{eq:balance}
and linear momentum \cref{eq:linear_momentum} read
\begin{align}
	\label{eq:mass-conservation}
	\nabla \cdot \vec{u} = -\nabla \cdot \vec{u}_0 =: g_{\div}
	\quad \text{and} \quad \nabla \cdot \mat{\sigma} = \vec{f},
\end{align}
where we abuse notation and use $\vec{u}$ to also denote the perturbation 
velocity.
To close \cref{eq:general_balance_law}, we require a constitutive 
law relating the Cauchy stress $\mat{\sigma}$ to 
the symmetric part of the gradient of the velocity 
$\symgrad(\vec{u}) := (\nabla \vec{u} + (\nabla \vec{u})^{\top})/2$. 
Linearizing a generic constitutive law of the form 
\begin{align*}
	\mat{\sigma} = \mat{G}(\mat{\varepsilon}(\vec{u}), q, \tilde{\mat{H}}), 
\end{align*}
where $q$ is the spherical response and $\tilde{\mat{H}}$ is a symmetric tensor
associated with some external field, possibly depending nonlinearly on 
$\vec{u}$ and $q$, often leads to a relation of the form
\begin{align}
	\label{eq:constitutive_law_fluid}
	\mat{\sigma} = 2\mu \mat{\varepsilon}(\vec{u}) - p \mat{I} + \tilde{\mat{F}},
\end{align}
where $\mu > 0$ is a parameter (often the viscosity), 
$p$ is a Lagrange multiplier enforcing the divergence constraint in 
\cref{eq:mass-conservation}, and 
$\tilde{\mat{F}}$ is a linearization of $\tilde{\mat{H}}$.
The physical interpretation of $\tilde{\mat{F}}$ varies depending on 
the continuum we are modelling and will be discussed for each example in 
section \Cref{sec:examples}.

Due to the presence of the term $\tilde{\mat{F}}$ and the nonvanishing bulk 
viscosity, $p$ is neither the mechanical nor the thermodynamic pressure 
\cite{Rajagopal15}. For this reason, it appears compelling from an 
engineering point of view to eliminate such Lagrange multiplier
from our formulation. We take the trace of \eqref{eq:constitutive_law_fluid} to 
obtain
\begin{equation}
	p = \frac{1}{d}\tr\left( 2\mu \symgrad(\vec{u}) 
		+ \tilde{\mat{F}} - \mat{\sigma} \right).
\end{equation}
Substituting this relation back into \cref{eq:constitutive_law_fluid}
gives 
\begin{align*}
\mat{\sigma}^D = 2 \mu \symgrad(\vec{u})^D 
	+ \tilde{\mat{F}}^D 
\implies 
\symgrad(\vec{u})^D = \frac{1}{2\mu} \left( \mat{\sigma}^D 
	- \tilde{\mat{F}}^D \right),
\end{align*}
where $\mat{\sigma}^D := \mat{\sigma} - d^{-1} \tr(\sigma) \mat{\mathbb{I}}$ is 
the deviatoric part of $\mat{\sigma}$.
On noting that
$\tr \symgrad(\vec{u}) = \nabla \cdot \vec{u} = g_{\mathrm{div}}$,
we have
\begin{align}
	\label{eq:eq:constitutive_law_fluid_strain_stress}
\symgrad(\vec{u}) = \frac{1}{2\mu} \mat{\sigma}^D 
	- \left( \frac{1}{2\mu}\tilde{\mat{F}}^D
		- \frac{1}{d} g_{\div} \mat{\mathbb{I}}  \right) 
	=: \frac{1}{2\mu} \mat{\sigma}^D - \mat{F}.
\end{align}

We also consider solid materials linearized about a zero displacement state.
Similar manipulations \cite[Chapter 7.1]{GonzalezStuart08} give rise to 
constitutive relations of the form
\begin{align}
	\label{eq:linear_constituitive_law_intro}
	\symgrad(\vec{u}) = \frac{1}{2\mu} \mat{\sigma}^D 
		+ \frac{1}{d(2\mu + \lambda)} (\tr \mat{\sigma}) \mat{\mathbb{I}} 
		- \mat{F},
\end{align}
where $\vec{u}$ is now the displacement and $\mu, \lambda > 0$ are material
parameters. Since
\cref{eq:eq:constitutive_law_fluid_strain_stress} is precisely of the form 
\cref{eq:linear_constituitive_law_intro} with $\lambda = \infty$,
\cref{eq:general_balance_law}
augmented with the constitutive law \cref{eq:linear_constituitive_law_intro}
becomes the following for both fluids and solids:
\begin{subequations}
	\label{eq:linear_laws_strong}
	\begin{alignat}{2}
		\label{eq:linear_constitutive_law}
		\frac{1}{2\mu} \mat{\sigma}^D + \frac{1}{d(2\mu + d \lambda)} 
			(\tr \mat{\sigma}) \mat{\mathbb{I}} -  \symgrad(\vec{u})
		&= \mat{F}
		\qquad & &\text{in } \Omega, \\
		\label{eq:linear_momentum_balance}
		\nabla \cdot \mat{\sigma} &= \vec{f} \qquad & &\text{in } \Omega,\\
		\label{eq:linear_u_dirichlet}
		\vec{u} &= \vec{g} \qquad & &\text{on } \partial \Omega,
	\end{alignat}	
\end{subequations}
where $\vec{f}$ encodes the external forces and $\vec{g}$ is the 
prescribed value of $\vec{u}$ on $\partial \Omega$. For simplicity,
we assume that $\mat{F}$ is independent of $\vec{u}$ so that 
\cref{eq:linear_laws_strong} is a linear system of partial 
differential equations.
If $\lambda = \infty$, we additionally require the compatibility condition 
$\int_{\partial \Omega} \vec{g} \cdot \vec{n} \d{s} 
= -\int_{\Omega} \tr \mat{F} \d{x}$, which can be seen by 
taking the trace of \cref{eq:linear_constitutive_law}, integrating over
$\Omega$, and applying the divergence theorem.
The formulation \cref{eq:linear_laws_strong} for incompressible flow
($\lambda = \infty$) is sometimes called the stress-velocity formulation 
and, at least in the numerical analysis community, seems to date back to 
\cite{CaiLeeWang04} for the design of least-squares finite element methods.


\begin{remark}
Constitutive laws for linear isotropic material with 
an external anisotropy may also be expressed as
\begin{align}
	\label{eq:linear_elasticity_constituitive_typical}
	\mat{\sigma} &= 2\mu \symgrad(\vec{u}) + \lambda 
		\mathrm{tr}(\symgrad(\vec{u})) \mat{\mathbb{I}} 
		+ \tilde{\mat{F}}.
\end{align}
Taking the trace of \cref{eq:linear_elasticity_constituitive_typical} gives
\begin{align*}
	\tr \mat{\sigma} = (2\mu + d\lambda) \tr \symgrad(\vec{u})
	+ \tr \mat{\tilde{F}}
	\implies
	\symgrad(\vec{u}) = 
	\frac{1}{2\mu}(\mat{\sigma} - \mat{\tilde{F}}) - 
	\frac{\lambda}{2\mu(2\mu + d\lambda)} 
	\tr (\mat{\sigma} - \mat{\tilde{F}}) \mat{\mathbb{I}}. 
\end{align*}
After a little more algebraic manipulation, we arrive at 
\cref{eq:linear_constituitive_law_intro} with 
\begin{align}
	\label{eq:strain-func-stress-anisotropy-from-stress-func-strain}
	\mat{F} = \frac{1}{2\mu} \mat{\tilde{F}}^D 
		+ \frac{1}{d(2\mu + d\lambda)} \tr (\tilde{\mat{F}}) \mat{\mathbb{I}}.
\end{align}
\end{remark}

\subsection{Variational formulation with strong symmetry}
\label{sec:general-setup-strong-sym}

We consider a variational formulation of \cref{eq:linear_laws_strong},
also known as the Hellinger-Reissner formulation 
\cite{Hellinger13,Reissner50,Reissner53},
where $\vec{u}$ and $\mat{\sigma}$ are independent variables. 
We use standard 
notation for Sobolev spaces of scalar-valued functions. 
For vector- or matrix-valued spaces, we include the codomain in the space; e.g.,
$L^2(\Omega; \mathbb{R}^d)$ is the space of square-integrable
vector-valued functions. Other codomains include 
$\mathbb{R}^{d\times d}_{\sym}$ the set of symmetric matrices,
and $\mathbb{R}^{d\times d}_{\anti}$ the set of skew-symmetric matrices.
We also define matrix-valued spaces with square-integrable divergence 
\begin{alignat*}{2}
	H(\div; \Omega, \circ)
		&:= \{ \mat{\tau} \in L^2(\Omega, \circ) : 
				\nabla \cdot \mat{\tau} \in L^2(\Omega, \mathbb{R}^d) \},
		\qquad & &\circ \in \{ \mathbb{R}^{d \times d}, 
			\mathbb{R}_{\sym}^{d \times d} \}, \\
	H_{*}(\div; \Omega, \circ) &:= \left\{ 
		\mat{\tau} \in H(\div; \Omega, \circ) : 
			\int_{\Omega} \tr \mat{\tau} \d{x} = 0 \right\}, 
		\qquad & &\circ \in \{ \mathbb{R}^{d \times d},
			\mathbb{R}_{\sym}^{d \times d} \},
\end{alignat*}
where the divergence of matrix is taken row-wise. 

The Hellinger-Reissner formulation of \cref{eq:linear_laws_strong} then reads
as follows: Find $\mat{\sigma} \in \Sigma^{\sym}$ and $\vec{u} \in V$ such that
\begin{subequations}
	\label{eq:linear_laws_weak_strong_sym}
	\begin{alignat}{2}
		\label{eq:linear_laws_weak_strong_sym_1}
		a(\mat{\sigma}, \mat{\tau}) + b(\mat{\tau}, \vec{u}) 
			&= \langle \mat{\tau} \vec{n}, \vec{g} \rangle_{\partial \Omega}
				+ (\mat{F}, \mat{\tau})_{L^2(\Omega)} 
		\qquad & &\forall \mat{\tau} \in \Sigma^{\sym}, \\
		\label{eq:linear_laws_weak_strong_sym_2}
		b(\mat{\sigma}, \vec{v}) &= (\vec{f}, \vec{v})_{L^2(\Omega)} \qquad & 
		&\forall \vec{v} \in V,
	\end{alignat}	
\end{subequations}
where the spaces are chosen as 
\begin{align}
	\label{eq:sigma-sym-v-def}
	\Sigma^{\sym} := \begin{cases}
		H(\div; \Omega, \mathbb{R}^{d\times d}_{\sym}) 
			& \text{if } \lambda < \infty, \\
		H_*(\div; \Omega, \mathbb{R}^{d\times d}_{\sym}) 
			& \text{if } \lambda = \infty,
	\end{cases}
	\quad \text{and} \quad
	V := L^2(\Omega; \mathbb{R}^d),
\end{align}
and the bilinear forms are given by
\begin{subequations}
	\label{eq:bilinear_forms}
	\begin{align}
	\label{eq:a-bilinear}
	a(\mat{\sigma}, \mat{\tau}) &:= 
		\frac{1}{2\mu} (\mat{\sigma}^D, \mat{\tau}^D)_{L^2(\Omega)} 
		+ \frac{1}{d(2\mu + d \lambda)} 
			(\tr \mat{\sigma}, \tr \mat{\tau})_{L^2(\Omega)},\\
	\label{eq:b-bilinear}
	b(\mat{\sigma}, \vec{v}) &:= 
		(\nabla \cdot \mat{\sigma}, \vec{v})_{L^2(\Omega)}.
\end{align}	
\end{subequations}
The angle brackets $\langle \cdot, \cdot \rangle_{\partial \Omega}$ 
in \cref{eq:linear_laws_weak_strong_sym_1}
denote the duality pairing between $H^{-1/2}(\partial \Omega; \mathbb{R}^d)$
and $H^{1/2}(\partial \Omega; \mathbb{R}^d)$. Since the normal trace
$\mat{\tau} \mapsto \tau \vec{n}|_{\partial \Omega}$ is a bounded operator
from $H(\div; \Omega, \mathbb{R}^{d \times d})$
to $H^{-1/2}(\partial \Omega; \mathbb{R}^d)$, all terms in
\cref{eq:linear_laws_weak_strong_sym} are well-defined
provided that
$\vec{f} \in L^2(\Omega; \mathbb{R}^d)$ and 
$\vec{g} \in H^{1/2}(\partial \Omega; \mathbb{R}^d)$, which
we shall assume for the remainder of the manuscript. 
We review the well-posedness of \cref{eq:linear_laws_weak_strong_sym} in 
\cref{sec:well-posedness-mixed} and note here that 
\cref{eq:linear_laws_weak_strong_sym} admits a unique solution.

A standard conforming Galerkin approximation of 
\cref{eq:linear_laws_weak_strong_sym} reads as follows: 
Find $\mat{\sigma}_h \in \Sigma^{\sym}_h$ and 
$\vec{u}_h \in V_h$ such that
\begin{subequations}
	\label{eq:linear_laws_weak_strong_sym_fem}
	\begin{alignat}{2}
		\label{eq:linear_laws_weak_strong_sym_fem_1}
		a(\mat{\sigma}_h, \mat{\tau}_h) + b(\mat{\tau}_h, \vec{u}_h) 
			&= \langle \mat{\tau}_h \vec{n}, \vec{g} \rangle_{\partial \Omega}
				+ (\mat{F}, \mat{\tau}_h)_{L^2(\Omega)} 
		\qquad & &\forall \mat{\tau}_h \in \Sigma_h^{\sym}, \\
		\label{eq:linear_laws_weak_strong_sym_fem_2}
		b(\mat{\sigma}_h, \vec{v}_h) &= (\vec{f}, \vec{v}_h)_{L^2(\Omega)} 
			\qquad & &\forall \vec{v}_h \in V_h,
	\end{alignat}	
\end{subequations}
where $\Sigma^{\sym}_h \subset \Sigma^{\sym}$ and $V_h \subset V$
are finite dimensional spaces.
\noindent As mentioned in \cref{sec:intro}, conforming finite elements for 
$\Sigma^{\sym}$ have been recently adopted in open source software.
One challenge is constructing a finite
element space $\Sigma^{\sym}_h$ that is both 
$H(\div; \Omega, \mathbb{R}^{d \times d})$-conforming and symmetric.
Another challenge is choosing a space $V_h$ such that
both the discrete kernel coercivity condition
\begin{align}
	\label{eq:a-elliptic-kernel-strong-sym-fem}
	a(\mat{\sigma}_h, \mat{\sigma}_h) 
		\geq \alpha_h^{\sym} \|\mat{\sigma}_h\|_{\div}^2 
		\qquad \forall \mat{\sigma}_h \in 
		\{ \mat{\tau}_h \in \Sigma^{\sym}_h : b(\mat{\tau}_h, \vec{v}_h) = 0 
			\ \forall \vec{v}_h \in V_h  \}
\end{align}
holds for some $\alpha_h^{\sym} > 0$ and the discrete inf-sup condition holds
\begin{align}
	\label{eq:b-infsup-strong-sym-fem}
	\beta_h^{\sym} := \inf_{\vec{v}_h \in V_h} 
		\sup_{\mat{\tau}_h \in \Sigma_h^{\sym}}
		\frac{b(\mat{\tau}_h, \vec{v}_h)}{ \| \mat{\tau}_h \|_{\div} 
			\|\vec{v}_h\| }
	 > 0.
\end{align}
Here, we use $\|\cdot\|_{\div}$ to denote the $H(\div)$ norm
and $\|\cdot\|$ to denote the $L^2$ norm. 
Nevertheless, more elements satisfying these conditions are
appearing in the literature; we refer to the recent manuscripts 
\cite{ChenHuang25,Gopalakrishnan25} and references therein for a
review. The same arguments in the beginning of this section show that
any scheme of the form of \cref{eq:linear_laws_weak_strong_sym_fem} will 
conserve the angular momentum since 
$\mat{\sigma}_h \in \Sigma_{h}^{\mathrm{sym}}$ is pointwise symmetric.

In the numerical examples below, we consider the following
two schemes on a shape regular, conforming simplicial mesh $\mathcal{T}_h$:
\begin{enumerate}
	\item the Johnson-Mercier-K\v{r}\'{i}\v{z}ek (JMK) scheme 
	(see \cite{JohnsonMercier78,WatwoodHartz68} for $d=2$ 
	and \cite{Krizek82} for $d=3$),
	which takes $\Sigma_h^{\sym}$ to be 
	$H(\div; \Omega, \mathbb{R}_{\sym}^{d \times d})$-conforming
	piecewise linear matrices
	on a barycentric refinement of $\mathcal{T}_h$ and 
	$V_h$ to be the space of piecewise constants
	on the barycentrically-refined mesh. See also \cite{Gopalakrishnan25}
	for an extension to any dimension and for reduced elements.

   \item The Hu-Zhang ($\mathrm{HZ}_k$) scheme (see \cite{HuZhang15d2} 
	for $d=2$ and \cite{HuZhang15d3} for $d=3$). For $d=2$, $\Sigma_h^{\sym}$
	is the space of $H(\div; \Omega, \mathbb{R}_{\sym}^{d \times d})$-conforming
	piecewise polynomial matrices of degree $k \geq 3$ that are also continuous
	at element vertices and 
	$V_h = \mathcal{P}_{k-1}(\mathcal{T}_h; \mathbb{R}^2)$, where
	$\mathcal{P}_{k-1}(\mathcal{T}_h; \circ)$ is the space of
	discontinuous piecewise polynomials of degree $k-1$ on $\mathcal{T}_h$
   	over the space $\circ$.
\end{enumerate}

\noindent The finite element spaces in the JMK scheme ($d=2,3$) and the 
$\mathrm{HZ}_k$ schemes ($d=2$) are implemented in the finite element software 
package Firedrake \cite{BrubeckKirby25,FiredrakeUserManual}. We also note 
that both schemes satisfy $\div \Sigma^{\sym}_h = V_h$, a common 
feature for strongly symmetric schemes.


\subsection{Variational formulation with weak symmetry}
\label{sec:general-setup-weak-sym}

A more common technique in the literature is to enforce
the symmetry constraint weakly. More precisely, we remove the
symmetry constraint in $\Sigma^{\sym}$ and instead enforce symmetry with
a Lagrange multiplier: Find $\mat{\sigma} \in \Sigma$, $\vec{u} \in V$,
and $\mat{\omega} \in \Xi$ such that
\begin{subequations}
	\label{eq:linear_laws_weak_weak_sym}
	\begin{alignat}{2}
		\label{eq:linear_laws_weak_weak_sym_1}
		a(\mat{\sigma}, \mat{\tau}) + b(\mat{\tau}, \vec{u}) 
			+ c(\mat{\tau}, \mat{\omega})  
		&= \langle \mat{\tau} \vec{n}, \vec{g} \rangle_{\partial \Omega} 
			+ (\mat{F}, \mat{\tau})_{L^2(\Omega)} \qquad 
			& &\forall \mat{\tau} \in \Sigma, \\
		\label{eq:linear_laws_weak_weak_sym_2}
		b(\mat{\sigma}, \vec{v}) &= (\vec{f}, \vec{v})_{L^2(\Omega)} 
			\qquad & &\forall \vec{v} \in V, \\
		\label{eq:linear_laws_weak_weak_sym_3}
		c(\mat{\sigma}, \mat{\xi}) &= 0 
			\qquad & &\forall \mat{\xi} \in 
			\Xi,
	\end{alignat}	
\end{subequations}
where $V$ is defined in \cref{eq:sigma-sym-v-def},
\begin{align}
	\label{eq:sigma-xi-def}
	\Sigma := \begin{cases}
		H(\div; \Omega, \mathbb{R}^{d\times d}) 
			& \text{if } \lambda < \infty, \\
		H_*(\div; \Omega, \mathbb{R}^{d\times d}) 
			& \text{if } \lambda = \infty,
	\end{cases}
	\quad 
	\Xi := L^2(\Omega, \mathbb{R}_{\anti}^{d \times d}),
	\quad \text{and} \quad
	c(\mat{\tau}, \mat{\xi}) := (\mat{\tau}, \mat{\xi})_{L^2(\Omega)}.
\end{align}
Problem \cref{eq:linear_laws_weak_weak_sym} is also well-posed --- 
see \cref{sec:well-posedness-mixed} for a review. A direct calculation shows
that \cref{eq:linear_laws_weak_strong_sym} and 
\cref{eq:linear_laws_weak_weak_sym} have the same solution.

\begin{remark}
	\label{rem:xi-lag-mult-antisym}
	Note that \cref{eq:linear_laws_weak_weak_sym_3} shows that 
	$\mat{\sigma}$ is pointwise symmetric, and so 
	taking $\tau \in C^{\infty}_0(\Omega; \mathbb{R}^{d \times d}_{\anti})$
	in \cref{eq:linear_laws_weak_weak_sym_1} gives
	\begin{align*}
		0 = a(\mat{\sigma}, \mat{\tau}) + b(\mat{\tau}, \vec{u}) 
			+ c(\mat{\tau}, \mat{\omega}) 
			= (\div \mat{\tau}, \vec{u})_{L^2(\Omega)}
				+ (\mat{\tau}, \mat{\omega})_{L^2(\Omega)}
			= -\langle \anti(\nabla \vec{u}), \mat{\tau} \rangle_{\Omega}
				+ (\mat{\tau}, \mat{\omega})_{L^2(\Omega)},
	\end{align*}
	where we used that $\mat{F}$ is symmetric, 
	$\langle \cdot, \cdot \rangle_{\Omega}$
	denotes the action of the distribution, and $\anti$ denotes the 
	skew-symmetric part of a tensor. Thus, the Lagrange multiplier 
	$\mat{\omega} = \anti(\nabla \vec{u})$, and
	we will refer to $\mat{\omega}$ as the rotation tensor, the tensorization of 
	the ``vorticity'' $\nabla \times \vec{u}$.
\end{remark}

A standard Galerkin formulation of \cref{eq:linear_laws_weak_weak_sym}
reads as follows:
Find $\mat{\sigma}_h \in \Sigma_h$, $\vec{u}_h \in V_h$,
and $\mat{\omega}_h \in \Xi_h$ such that
\begin{subequations}
	\label{eq:linear_laws_weak_weak_sym_fem}
	\begin{alignat}{2}
		\label{eq:linear_laws_weak_weak_sym_fem_1}
		a(\mat{\sigma}_h, \mat{\tau}_h) + b(\mat{\tau}_h, \vec{u}_h) 
			+ c(\mat{\tau}_h, \mat{\omega}_h)  
		&= \langle \mat{\tau}_h \vec{n}, \vec{g} \rangle_{\partial \Omega} 
			+ (\mat{F}, \mat{\tau}_h)_{L^2(\Omega)} \qquad 
			& &\forall \mat{\tau}_h \in \Sigma_h, \\
		\label{eq:linear_laws_weak_weak_sym_fem_2}
		b(\mat{\sigma}_h, \vec{v}_h) &= (\vec{f}, \vec{v}_h)_{L^2(\Omega)} 
			\qquad & &\forall \vec{v}_h \in V_h, \\
		\label{eq:linear_laws_weak_weak_sym_fem_3}
		c(\mat{\sigma}_h, \mat{\xi}_h) &= 0 
			\qquad & &\forall \mat{\xi}_h \in \Xi_h,
	\end{alignat}	
\end{subequations}
where $\Sigma_h \subset \Sigma$, $V_h \subset V$, and $\Xi_h \subset \Xi$
are finite dimensional spaces. $H(\div; \mathbb{R}^{d \times d})$-conforming 
spaces $\Sigma_h$ can be constructed by taking each row of a matrix
in $\Sigma_h$ to be a (vector-valued) 
$\mathrm{H}(\div; \Omega, \mathbb{R}^d)$-conforming 
element --- some examples are below.
Note that $V_h$ and $\Xi_h$ do not have any continuity requirement, so their
construction is also simple. Here, the challenge lies in choosing spaces so that
both the discrete kernel coercivity condition 
\begin{align}
	\label{eq:a-elliptic-kernel-weak-sym-fem}
	a(\mat{\sigma}_h, \mat{\sigma}_h) \geq \alpha_h \|\mat{\sigma}_h\|_{\div}^2 
		\qquad \forall \mat{\sigma}_h \in 
		\{ \mat{\tau}_h \in \Sigma_h : 
			b(\mat{\tau}_h, \vec{v}_h) + c(\mat{\tau}_h, \mat{\xi}_h) = 0 
			\ \forall \vec{v}_h \in V_h, \ \forall \mat{\xi}_h \in \Xi_h  \}
\end{align}
holds for some $\alpha_h > 0$ and the discrete inf-sup condition holds
\begin{align}
	\label{eq:b-infsup-weak-sym-fem}
	\beta_h := \inf_{\substack{\vec{v}_h \in V_h \\ \mat{\xi}_h \in \Xi_h}} 
		\sup_{\mat{\tau}_h \in \Sigma_h}
		\frac{b(\mat{\tau}_h, \vec{v}_h) + c(\mat{\tau}_h, \mat{\xi}_h)}{ 
			\| \mat{\tau}_h \|_{\div} (\|\vec{v}_h\| + \|\mat{\xi}_h\|) }
	 > 0.
\end{align}
Typically, one chooses the spaces so that $\div \Sigma_h = V_h$, and
\cref{eq:a-elliptic-kernel-weak-sym-fem} follows from the corresponding
result on the continuous level. Thus, inf-sup compatibility 
\cref{eq:b-infsup-weak-sym-fem} is the main difficulty. We refer to
Chapter 9 of \cite{BoffiBrezziFortin13} for a review of many available 
choices. Generally, $\mat{\sigma}_h \in \Sigma_h$ given by 
\cref{eq:linear_laws_weak_weak_sym_fem} will not be pointwise symmetric; 
thus, the arguments in the beginning of this section show that 
angular momentum is \textit{not} conserved. However, we do have 
weak conservation in the sense of 
\cref{eq:linear_laws_weak_weak_sym_fem_3}.

In the examples below, we consider the following two schemes defined
on a shape-regular, conforming mesh $\mathcal{T}_h$:

\begin{enumerate}
	\item The PEERS scheme \cite{ArnoldBrezziDouglas84} defined 
	for $d=2$ as follows. 
	For $\ell \geq 1$, let $\mathcal{RT}_{\ell}(\mathcal{T}_h)$ 
	be the standard Raviart-Thomas space \cite{RaviartThomas77},
	$\mathcal{CG}_{\ell}(\mathcal{T}_h)$ the space of continuous piecewise 
	polynomials of degree $\ell$, and $\mathcal{B}_{\ell}(\mathcal{T}_h)$
	the subspace of $\mathcal{CG}_{\ell}(\mathcal{T}_h)$ that vanishes on
	all mesh edges. Then, $\Sigma_h$, $V_h$, and $\Xi_h$ are chosen
	as follows:
	\begin{subequations}
		\label{eq:peers}
		\begin{align}
		\label{eq:peers-stress}
		\Sigma_h &= (\mathcal{RT}_1(\mathcal{T}_h) 
			\oplus \curl \mathcal{B}_{3}(\mathcal{T}_h))^2, \\
		\label{eq:peers-displacement}
		V_h &= \mathcal{P}_{0}(\mathcal{T}_h; \mathbb{R}^2), \\
		\label{eq:peers-anti}
		\Xi_h &= \left\{ \begin{bmatrix}
			0 & \phi \\
			-\phi & 0
			\end{bmatrix} : \phi \in  \mathcal{CG}_{1}(\mathcal{T}_h) \right\},
	\end{align}
	\end{subequations}
	where $\curl$ is a $\pi/2$ counterclockwise rotation of the gradient and
	the notation in \cref{eq:peers-stress} means that the rows of matrices in
	of $\Sigma_h$ belong to 
	$\mathcal{RT}_1(\mathcal{T}_h) \oplus \curl \mathcal{B}_3(\mathcal{T}_h)$. 
	The extension of PEERS to $d=3$ can be found in 
	\cite[Example 9.4.2]{BoffiBrezziFortin13} and to higher-order 
	in \cite[Example 9.4.1]{BoffiBrezziFortin13}.

\item The Arnold-Falk-Winther ($\mathrm{AFW}_k$) scheme 
	\cite{ArnoldFalkWinther07} defined for 
	$d \in \{2, 3\}$ and $k \geq 1$ as follows. For $\ell \geq 1$, let 
	$\mathcal{BDM}_{\ell}(\mathcal{T}_h)$ denote the standard
	Brezzi-Douglas-Marini \cite{BrezziDouglasMarini85} space of order $\ell$.
	Then, $\Sigma_h$, $V_h$, and $\Xi_h$ are chosen
	as follows:
	\begin{align}
		\label{eq:afw}
		\Sigma_h &= \mathcal{BDM}_k(\mathcal{T}_h)^d, 
		\quad
		V_h = \mathcal{P}_{k-1}(\mathcal{T}_h; \mathbb{R}^d),
		\quad \text{and} \quad 
		\Xi_h = \mathcal{P}_{k-1}(\mathcal{T}_h; 
			\mathbb{R}^{d \times d}_{\anti}).
	\end{align}
	At lowest order ($k=1$), this scheme is similar to the 
	hybridized scheme in \cite{AmaraThomas79}.
\end{enumerate}

\begin{remark}
	For all the schemes above, we also incorporate the constraint
	$\int_{\Omega} \tr \mat{\sigma} \d{x} = 0$ into the space 
	$\Sigma_h^{\sym}$ or $\Sigma_h$ if $\lambda = \infty$. However, this is
	only for theoretical purposes. In the implementation, this additional
	constraint is instead applied as a post-processing procedure. 
	See \cref{sec:linear-solvers} for additional discussion.
\end{remark}

\begin{remark}
	Suppose we start with a constitutive law of the form
	\cref{eq:linear_elasticity_constituitive_typical} with $\lambda < \infty$
	and perform the change of variables 
	$\tilde{\mat{\sigma}} = \mat{\sigma}  - \tilde{\mat{F}}$. Then,
	$\tilde{\mat{\sigma}}$ and $\vec{u}$ satisfy the same
	law \cref{eq:linear_constitutive_law} with $\mat{F} \equiv 0$.
	Thus, in this setting, one could always assume without loss 
	of generality that $\mat{F} \equiv 0$. Similarly, if $g_{\div} \equiv 0$,
	one could also eliminate $\mat{F}$ when $\lambda = \infty$.
	However, we do not pursue this simplification as we consider
	$g_{\div} \neq 0$ and want to highlight the potential behavior in the 
	nonlinear setting where such a change of variable may not be possible.
\end{remark}

\begin{remark}
	That we only consider ``enclosed flow'' or ``pure displacement'' 
	boundary conditions $(\vec{u} = \vec{g}$ on $\partial \Omega$)
	is neither essential to the behavior of the methods 
	nor the analysis presented below. If we instead have 
	$\vec{u} = \vec{g}_1$ on $\Gamma_D$ and 
	$\mat{\sigma} \vec{n} = \vec{g}_2$ on $\Gamma_T$ for some 
	partition of the boundary $\partial \Omega = \Gamma_D \cup \Gamma_T$,
	then the condition $\mat{\sigma} \vec{n} = \vec{g}_2$ becomes 
	an essential boundary condition on the stress spaces in 
	\cref{eq:linear_laws_weak_strong_sym_fem,eq:linear_laws_weak_weak_sym_fem}.
 	However, the Firedrake library, which we use for the numerical examples 
	below, does not support the strong imposition of essential boundary 
	conditions for finite elements with supersmoothness, which includes the 
	stress space in the Hu-Zhang scheme.
\end{remark}

\section{Examples}
\label{sec:examples}

We now consider three different anisotropy terms $\mat{F}$ in 
\cref{eq:linear_laws_weak_strong_sym_1} inspired by the constitutive laws
of various materials. In each example, we construct a manufactured solution
so that $\vec{f} \equiv \vec{0}$, $\mat{\sigma} \equiv \mat{0}$, and
$\vec{u} \neq \vec{0}$ but known analytically.
We fix $\mu = 10^{-4}$ and choose $\lambda$ depending on the 
example. 
However, we note that the behavior of the methods highlighted below is 
insensitive to the choice of $\mu$.

We set $\Omega = (0, 1)^d$ and generate an unstructured mesh using 
ngsPETSc \cite{ngsPETSc} with \texttt{maxh} as specified in each
example. These meshes avoid a superconvergence
phenomenon we observed for the PEERS scheme on a structured mesh.
We then measure the convergence of the JMK and or $\mathrm{HZ}_3$ scheme 
described at
the end of \cref{sec:general-setup-strong-sym} for the strongly symmetric 
formulation \cref{eq:linear_laws_weak_strong_sym_fem}
and the PEERS (only if $d=2$) and 
$\mathrm{AFW}_k$ schemes, $k \geq 1$, described at the end of 
\cref{sec:general-setup-weak-sym}, for the weakly symmetric formulation
\cref{eq:linear_laws_weak_weak_sym_fem}. 
We compute the stress errors $\|\mat{\sigma} - \mat{\sigma}_h\|_{\div}$ 
and displacement errors $\|\vec{u} - \vec{u}_h\|$ on a sequence of meshes
uniformly refined by bisecting every edge.
For the weakly symmetric schemes, we also compute the error in the 
rotation tensor $\|\mat{\omega} - \mat{\omega}_h\|$, where we recall 
from \cref{rem:xi-lag-mult-antisym} that
$\mat{\omega} = \anti(\nabla \vec{u})$.

All the examples are implemented using the Firedrake 
\cite{FiredrakeUserManual} library. The code to reproduce the results
can be found at \cite{BrubeckZerbinatiParker26}. 
A discussion of the linear solvers appears in \cref{sec:linear-solvers}.

\subsection{Example 1: Linear isotropic solid in 2D} 
\label{sec:lin_el}

We first consider the case of a linear isotropic solid with no external
anisotropy ($\mat{F} \equiv \mat{0}$). In this case, the only displacement
fields that are stress-free are rigid body motions. Here, we are not 
interested in the nearly incompressible regime, so we set $\lambda = 1$
and $\texttt{maxh}=1/8$.
See \cref{sec:well-posedness-mixed} for some remarks on the stability and
well-posedness of the Hellinger-Reissner formulation in the limit as $\lambda
\to \infty$.

We choose the exact displacement to be the following rigid
body motion 
\begin{equation}
	\label{eq:lin_el_exact}
\vec{u}(x, y) = \delta \begin{bmatrix} -y \\ x \end{bmatrix}
\implies \mat{\sigma} \equiv \mat{0}, \quad  
\vec{g} = \vec{u}|_{\partial \Omega}, \quad \text{and} \quad
\mat{\omega} = \delta \begin{bmatrix}
		0 & 1 \\
		-1 & 0
	\end{bmatrix},
\end{equation}
where $\delta \in \{ 10, 10^3, 10^5\}$ is a scaling parameter that we vary.
Note that the exact rotation tensor satisfies $\mat{\omega} \in \Xi_h$ for both 
the PEERS and the $\mathrm{AFW}_k$ schemes for all $k \geq 1$.

The numerical results are displayed in 
\cref{fig:le2d-errors}. Note that
all three methods produce about the same errors and the displacement
is converging linearly. As we will see 
in \cref{sec:structure-preservation-examples}, 
$\mat{\sigma}_h \equiv \mat{0}$ in exact
arithmetic for each of the schemes. 
Here, the growth in the error as $\delta$ increases arises
from the accumulation of roundoff errors that are proportional to $\delta$.
The same phenomenon is observed for 
discretizations of incompressible flow --- see \cref{sec:pressure-robustness}.

\definecolor{alpha1}{RGB}{100, 143, 255}
\definecolor{alpha2}{RGB}{120, 94, 240}
\definecolor{alpha3}{RGB}{220, 38, 127}
\definecolor{alpha4}{RGB}{254, 97, 0}
\definecolor{alpha5}{RGB}{255, 176, 0}

\pgfplotscreateplotcyclelist{myplots}{
	{mark=*, alpha1, mark options={fill=alpha1}},
	{mark=triangle*, alpha2, mark options={fill=alpha2}},
	{mark=square*, alpha3, mark options={fill=alpha3}},
	{mark=diamond*, alpha4, mark options={fill=alpha4}},
	{mark=pentagon*, alpha5, mark options={fill=alpha5}}
}

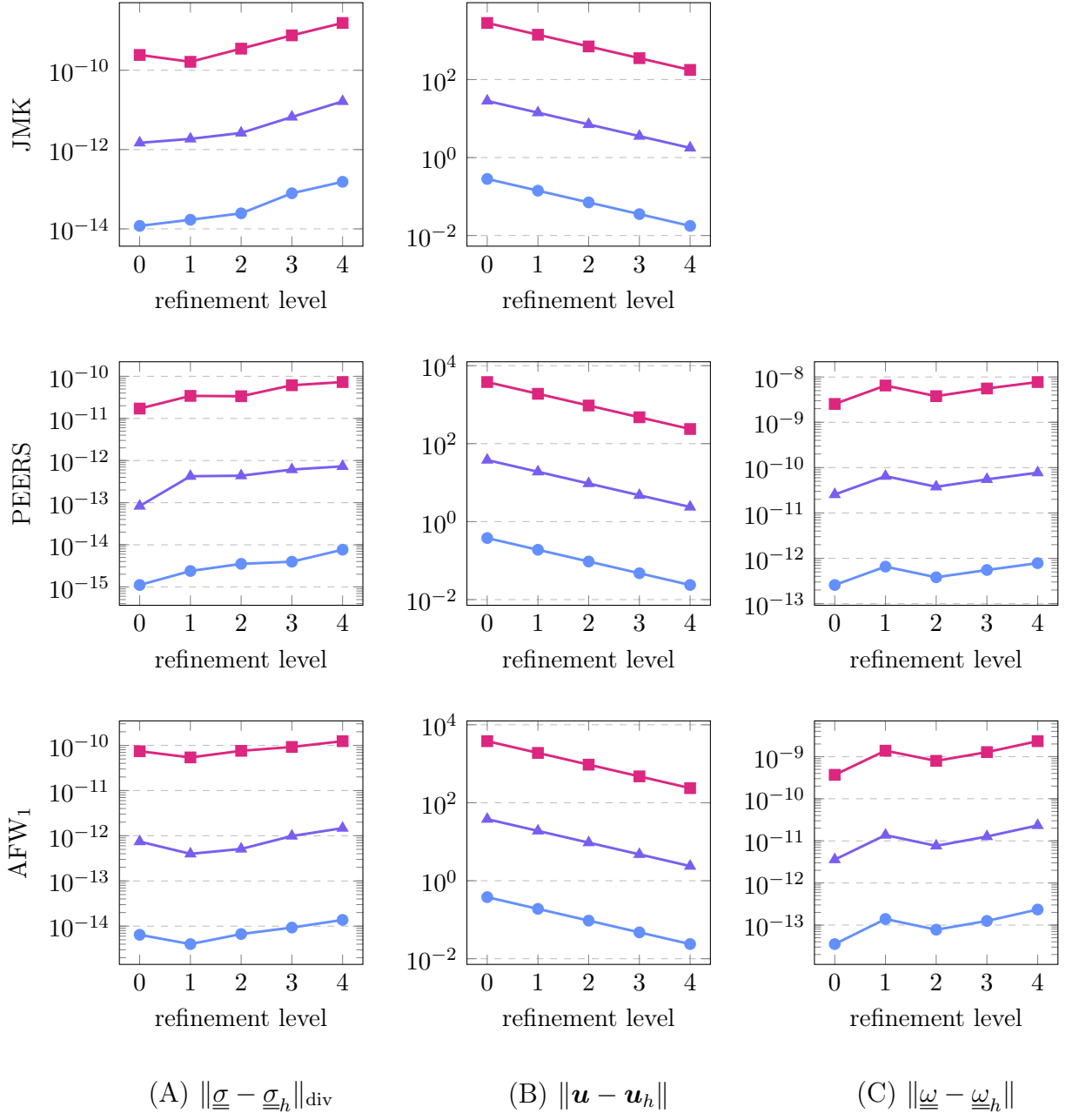
\begin{figure}[htbp]
	\centering
	\begin{tikzpicture}
		\centering
		\begin{groupplot}[%
			group style={group size=3 by 3, 
				         horizontal sep=45pt, 
						 vertical sep=50pt},
			width=0.32\linewidth,
			height=0.32\linewidth,
			domain=2:5, 
			xtick={0,1,2,3,4}, 
			ymode=log, 
			xlabel={refinement level},
			ymajorgrids=true,
			grid style=dashed,
			every axis plot/.append style={line width=1.1pt},
			cycle list name=myplots,
		]
		\nextgroupplot[ylabel={JMK}]
		\addplot+[discard if not={Bnd}{10.0}] 
				table [x=ref, y=sigma_error, col sep=comma] 
				{data/rigid_body_motion_jm_1.csv};
				\label{plot:le-jm-hdiv-alpha1}
		\addplot+[discard if not={Bnd}{1000.0}] 
				table [x=ref, y=sigma_error, col sep=comma] 
				{data/rigid_body_motion_jm_1.csv};
				\label{plot:le-jm-hdiv-alpha2}
		\addplot+[discard if not={Bnd}{100000.0}] 
				table [x=ref, y=sigma_error, col sep=comma] 
				{data/rigid_body_motion_jm_1.csv};
				\label{plot:le-jm-hdiv-alpha3}
		\nextgroupplot
		\addplot+[discard if not={Bnd}{10.0}] 
				table [x=ref, y=displacement_error, col sep=comma] 
				{data/rigid_body_motion_jm_1.csv};
		\addplot+[discard if not={Bnd}{1000.0}] 
				table [x=ref, y=displacement_error, col sep=comma] 
				{data/rigid_body_motion_jm_1.csv};
		\addplot+[discard if not={Bnd}{100000.0}] 
				table [x=ref, y=displacement_error, col sep=comma] 
				{data/rigid_body_motion_jm_1.csv};
		\nextgroupplot[group/empty plot]
		\nextgroupplot[ylabel={PEERS}]
		\addplot+[discard if not={Bnd}{10.0}] 
				table [x=ref, y=sigma_error, col sep=comma] 
				{data/rigid_body_motion_peers_1.csv};
		\addplot+[discard if not={Bnd}{1000.0}] 
				table [x=ref, y=sigma_error, col sep=comma] 
				{data/rigid_body_motion_peers_1.csv};
		\addplot+[discard if not={Bnd}{100000.0}] 
				table [x=ref, y=sigma_error, col sep=comma] 
				{data/rigid_body_motion_peers_1.csv};
		\nextgroupplot
		\addplot+[discard if not={Bnd}{10.0}] 
				table [x=ref, y=displacement_error, col sep=comma] 
				{data/rigid_body_motion_peers_1.csv};
		\addplot+[discard if not={Bnd}{1000.0}] 
				table [x=ref, y=displacement_error, col sep=comma] 
				{data/rigid_body_motion_peers_1.csv};
		\addplot+[discard if not={Bnd}{100000.0}] 
				table [x=ref, y=displacement_error, col sep=comma] 
				{data/rigid_body_motion_peers_1.csv};
		\nextgroupplot
		\addplot+[discard if not={Bnd}{10.0}] 
				table [x=ref, y=omega_err, col sep=comma] 
				{data/rigid_body_motion_peers_1.csv};
		\addplot+[discard if not={Bnd}{1000.0}] 
				table [x=ref, y=omega_err, col sep=comma] 
				{data/rigid_body_motion_peers_1.csv};
		\addplot+[discard if not={Bnd}{100000.0}] 
				table [x=ref, y=omega_err, col sep=comma] 
				{data/rigid_body_motion_peers_1.csv};
		\nextgroupplot[ylabel=$\mathrm{AFW}_1$]
		\addplot+[discard if not={Bnd}{10.0}] 
				table [x=ref, y=sigma_error, col sep=comma] 
				{data/rigid_body_motion_afw_1.csv};
		\addplot+[discard if not={Bnd}{1000.0}] 
				table [x=ref, y=sigma_error, col sep=comma] 
				{data/rigid_body_motion_afw_1.csv};
		\addplot+[discard if not={Bnd}{100000.0}] 
				table [x=ref, y=sigma_error, col sep=comma] 
				{data/rigid_body_motion_afw_1.csv};
		\nextgroupplot
		\addplot+[discard if not={Bnd}{10.0}] 
				table [x=ref, y=displacement_error, col sep=comma] 
				{data/rigid_body_motion_afw_1.csv};
		\addplot+[discard if not={Bnd}{1000.0}] 
				table [x=ref, y=displacement_error, col sep=comma] 
				{data/rigid_body_motion_afw_1.csv};
		\addplot+[discard if not={Bnd}{100000.0}] 
				table [x=ref, y=displacement_error, col sep=comma] 
				{data/rigid_body_motion_afw_1.csv};
		\nextgroupplot
		\addplot+[discard if not={Bnd}{10.0}] 
				table [x=ref, y=omega_err, col sep=comma] 
				{data/rigid_body_motion_afw_1.csv};
		\addplot+[discard if not={Bnd}{1000.0}] 
				table [x=ref, y=omega_err, col sep=comma] 
				{data/rigid_body_motion_afw_1.csv};
		\addplot+[discard if not={Bnd}{100000.0}] 
				table [x=ref, y=omega_err, col sep=comma] 
				{data/rigid_body_motion_afw_1.csv};
		\end{groupplot}
		\node at ($(group c1r3.south) + (0,-2)$) 
			{(A) $\|\mat{\sigma}-\mat{\sigma}_h\|_{\div}$};
		\node at ($(group c2r3.south) + (0,-2)$) 
			{(B) $\|\vec{u}-\vec{u}_h\|$};
		\node at ($(group c3r3.south) + (0,-2)$) 
			{(C) $\|\mat{\omega} - \mat{\omega}_h\|$};
	\end{tikzpicture}
	\caption{Numerical results for the 2D linear isotropic solid example
		in \cref{sec:lin_el} for
		$\delta = 10$ (\ref*{plot:le-jm-hdiv-alpha1}),
		$\delta = 10^3$ (\ref*{plot:le-jm-hdiv-alpha2}), and
		$\delta = 10^5$ (\ref*{plot:le-jm-hdiv-alpha3}) with the 
		JMK scheme (first row), the PEERS scheme (second row), and 
		$\mathrm{AFW}_1$ scheme (third row).
		(A) the stress errors in 
			$H(\div; \Omega, \mathbb{R}^{2\times 2}_{\sym})$  
		(B) the displacement errors in 
		$L^2(\Omega; \mathbb{R}^2)$, and 
		(C) $L^2(\Omega; \mathbb{R}_{\anti}^{2 \times 2})$ 
		errors in the Lagrange multiplier $\mat{\omega}$
		(only for the weakly symmetric methods). 
		Observe that three methods produce nearly the same errors.}
	\label{fig:le2d-errors}
\end{figure}

 \subsection{Example 2: Transversely isotropic solids in 2D}
\label{sec:trans_iso}

We now introduce a transverse isotropy by adding a directional contribution 
aligned with a prescribed vector field, 
i.e.~taking $\mat{\tilde{F}} = \delta (\vec{\nu} \otimes \vec{\nu})$
in \cref{eq:linear_elasticity_constituitive_typical},
where the director $\vec{\nu} : \Omega \to \mathbb{R}^2$ specifies the axis of 
transverse isotropy and $\delta > 0$ is a parameter. 
Such material laws arise in colloidal 
suspensions \cite{ericksen} and 
fiber-reinforced materials \cite{MerodioOgden05}.
Note that $\mat{F}$ is then given by 
\cref{eq:strain-func-stress-anisotropy-from-stress-func-strain}.
 
Suppose that we choose exact displacement $\vec{u}$, the director $\vec{\nu}$,
and the boundary data $\vec{g}$ as follows:
\begin{equation}
	\vec{u} = -\frac{\delta}{2\mu} \begin{bmatrix}
		\frac{1}{3}x^3 - \frac{1}{3}y^3\\
		x^2 y + xy^2 + \frac{1}{3} y^3 + \frac{2}{3} x^3
	\end{bmatrix},
	\quad 
	\vec{\nu}(x,y) =  \begin{bmatrix}
		x \\ x+y
	\end{bmatrix},
	\quad \text{and} \quad
	\vec{g} = \vec{u}|_{\partial \Omega},
\end{equation} 
where $\delta \in \{10, 10^3, 10^5\}$ is again a scaling parameter
that we vary. The rotation tensor is given by
\begin{align}
	\mat{\omega} = \frac{\delta(x^2 + xy + y^2)}{2\mu} \begin{bmatrix}
		0 & 1 \\
		-1 & 0
	\end{bmatrix}.
\end{align}
A direct calculation shows that for any choice of $\mu$ and $\delta$,
$\mat{\sigma} = \lambda (\tr \symgrad(\vec{u}))\mat{\mathbb{I}}$,
and so $\mat{\sigma} \equiv \mat{0}$ if and only if $\lambda = 0$. 
Thus, we prescribe $\lambda = 0$ and again set $\texttt{maxh}=1/8$.

The numerical results are displayed in 
\cref{fig:trans2d-errors}. First note that the JMK scheme gives results
that have the same behavior as the previous example. Again, we will see
below in \cref{sec:structure-preservation-examples} 
that for the JMK scheme in exact arithmetic, 
$\mat{\sigma}_h \equiv \mat{0}$. In contrast to the previous example,
the PEERS and $\mathrm{AFW}_1$ schemes produce large
errors in the stress $\mat{\sigma}$ and the rotation tensor $\mat{\omega}$
that also scale like $\delta$. These errors
cannot be explained by the accumulation of roundoff errors or stopping criteria
of the iterative solver. To investigate this phenomenon further, we also
consider the $\mathrm{AFW}_k$ scheme with larger $k$. The results
for $k \in \{1,2,3\}$ are displayed in \cref{fig:trans2d-afw-errors},
which show that at exactly $k=3$, the scheme recovers the behavior in
the previous example. Namely, the stress and rotation tensor errors are
zero up to roundoff
error and solver tolerances. Note that the exact rotation tensor
satisfies $\mat{\omega} \in \Xi_h$ for the $\mathrm{AFW}_k$ scheme for all 
$k \geq 3$, while in the previous example we have $\mat{\omega} \in \Xi_h$
for all the schemes. Thus, the inclusion $\mat{\omega} \in \Xi_h$ seems to
be crucial for the robustness of stress errors with respect
to $\delta$ for the weakly symmetric schemes. This finding will be
confirmed theoretically in \cref{sec:theory}.

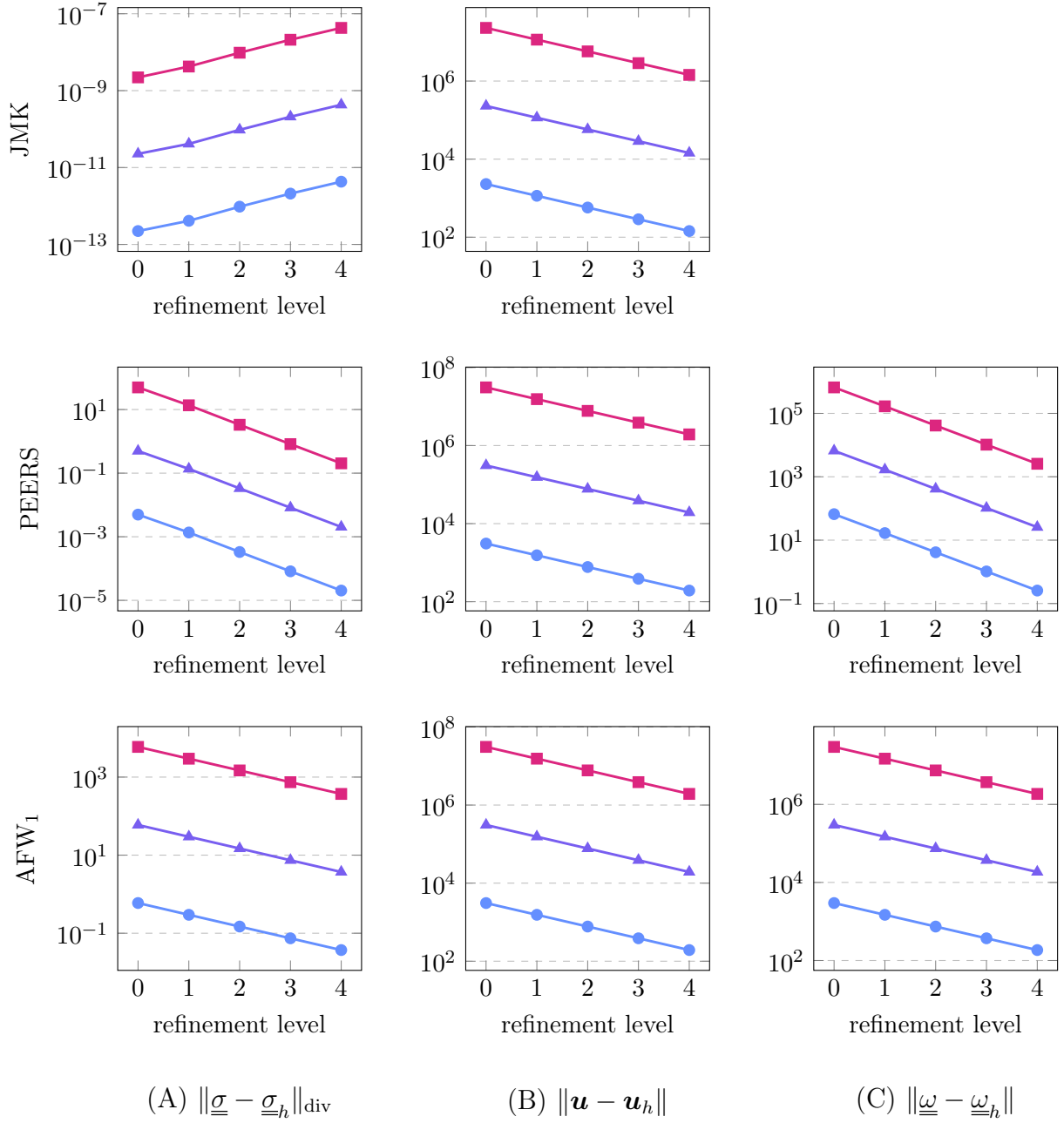
\begin{figure}[htbp]
	\centering
	\begin{tikzpicture}
		\centering
		\begin{groupplot}[%
			group style={group size=3 by 3, 
				         horizontal sep=45pt, 
						 vertical sep=50pt},
			width=0.32\linewidth,
			height=0.32\linewidth,
			domain=2:5, 
			xtick={0,1,2,3,4}, 
			ymode=log, 
			xlabel={refinement level},
			ymajorgrids=true,
			grid style=dashed,
			every axis plot/.append style={line width=1.1pt},
			cycle list name=myplots,
		]
		\nextgroupplot[ylabel={JMK}]
		\addplot+[discard if not={Bnd}{10.0}] 
				table [x=ref, y=sigma_error, col sep=comma] 
				{data/transverse_isotropic_jm_1.csv};
				\label{plot:trans-jm-hdiv-alpha1}
		\addplot+[discard if not={Bnd}{1000.0}] 
				table [x=ref, y=sigma_error, col sep=comma] 
				{data/transverse_isotropic_jm_1.csv};
				\label{plot:trans-jm-hdiv-alpha2}
		\addplot+[discard if not={Bnd}{100000.0}] 
				table [x=ref, y=sigma_error, col sep=comma] 
				{data/transverse_isotropic_jm_1.csv};
				\label{plot:trans-jm-hdiv-alpha3}
		\nextgroupplot
		\addplot+[discard if not={Bnd}{10.0}] 
				table [x=ref, y=displacement_error, col sep=comma] 
				{data/transverse_isotropic_jm_1.csv};
		\addplot+[discard if not={Bnd}{1000.0}] 
				table [x=ref, y=displacement_error, col sep=comma] 
				{data/transverse_isotropic_jm_1.csv};
		\addplot+[discard if not={Bnd}{100000.0}] 
				table [x=ref, y=displacement_error, col sep=comma] 
				{data/transverse_isotropic_jm_1.csv};
		\nextgroupplot[group/empty plot]
		\nextgroupplot[ylabel={PEERS}]
		\addplot+[discard if not={Bnd}{10.0}] 
				table [x=ref, y=sigma_error, col sep=comma] 
				{data/transverse_isotropic_peers_1.csv};
		\addplot+[discard if not={Bnd}{1000.0}] 
				table [x=ref, y=sigma_error, col sep=comma] 
				{data/transverse_isotropic_peers_1.csv};
		\addplot+[discard if not={Bnd}{100000.0}] 
				table [x=ref, y=sigma_error, col sep=comma] 
				{data/transverse_isotropic_peers_1.csv};
		\nextgroupplot
		\addplot+[discard if not={Bnd}{10.0}] 
				table [x=ref, y=displacement_error, col sep=comma] 
				{data/transverse_isotropic_peers_1.csv};
		\addplot+[discard if not={Bnd}{1000.0}] 
				table [x=ref, y=displacement_error, col sep=comma] 
				{data/transverse_isotropic_peers_1.csv};
		\addplot+[discard if not={Bnd}{100000.0}] 
				table [x=ref, y=displacement_error, col sep=comma] 
				{data/transverse_isotropic_peers_1.csv};
		\nextgroupplot
		\addplot+[discard if not={Bnd}{10.0}] 
				table [x=ref, y=omega_err, col sep=comma] 
				{data/transverse_isotropic_peers_1.csv};
		\addplot+[discard if not={Bnd}{1000.0}] 
				table [x=ref, y=omega_err, col sep=comma] 
				{data/transverse_isotropic_peers_1.csv};
		\addplot+[discard if not={Bnd}{100000.0}] 
				table [x=ref, y=omega_err, col sep=comma] 
				{data/transverse_isotropic_peers_1.csv};
		\nextgroupplot[ylabel={$\mathrm{AFW}_1$}]
		\addplot+[discard if not={Bnd}{10.0}] 
				table [x=ref, y=sigma_error, col sep=comma] 
				{data/transverse_isotropic_afw_1.csv};
		\addplot+[discard if not={Bnd}{1000.0}] 
				table [x=ref, y=sigma_error, col sep=comma] 
				{data/transverse_isotropic_afw_1.csv};
		\addplot+[discard if not={Bnd}{100000.0}] 
				table [x=ref, y=sigma_error, col sep=comma] 
				{data/transverse_isotropic_afw_1.csv};
		\nextgroupplot
		\addplot+[discard if not={Bnd}{10.0}] 
				table [x=ref, y=displacement_error, col sep=comma] 
				{data/transverse_isotropic_afw_1.csv};
		\addplot+[discard if not={Bnd}{1000.0}] 
				table [x=ref, y=displacement_error, col sep=comma] 
				{data/transverse_isotropic_afw_1.csv};
		\addplot+[discard if not={Bnd}{100000.0}] 
				table [x=ref, y=displacement_error, col sep=comma] 
				{data/transverse_isotropic_afw_1.csv};
		\nextgroupplot
		\addplot+[discard if not={Bnd}{10.0}] 
				table [x=ref, y=omega_err, col sep=comma] 
				{data/transverse_isotropic_afw_1.csv};
		\addplot+[discard if not={Bnd}{1000.0}] 
				table [x=ref, y=omega_err, col sep=comma] 
				{data/transverse_isotropic_afw_1.csv};
		\addplot+[discard if not={Bnd}{100000.0}] 
				table [x=ref, y=omega_err, col sep=comma] 
				{data/transverse_isotropic_afw_1.csv};
		\end{groupplot}
		\node at ($(group c1r3.south) + (0,-2)$) 
			{(A) $\|\mat{\sigma}-\mat{\sigma}_h\|_{\div}$};
		\node at ($(group c2r3.south) + (0,-2)$) {(B) $\|\vec{u}-\vec{u}_h\|$};
		\node at ($(group c3r3.south) + (0,-2)$) 
			{(C) $\|\mat{\omega} - \mat{\omega}_h\|$};
	\end{tikzpicture}
	\caption{Numerical results for the 2D transversely isotropic example
		in \cref{sec:trans_iso} for
		$\delta = 10$ (\ref*{plot:le-jm-hdiv-alpha1}),
		$\delta = 10^3$ (\ref*{plot:le-jm-hdiv-alpha2}), and
		$\delta = 10^5$ (\ref*{plot:le-jm-hdiv-alpha3}) with the 
		JMK, PEERS, and $\mathrm{AFW}_1$ schemes.
		The errors are displayed as in \cref{fig:le2d-errors}. 
		Observe that the stress errors for the PEERS and 
		$\mathrm{AFW}_1$ schemes are no longer zero up to roundoff errors,
      in contrast to
		\cref{fig:le2d-errors}, even though both examples satisfy
		$\mat{\sigma} \equiv \mat{0}$.
      Here, the exact rotation tensor $\mat{\omega}$
      is not included in the discrete space $\Xi_h$ for the weakly symmetric schemes.
   }
	\label{fig:trans2d-errors}
\end{figure}

\begin{figure}[htbp]
	\centering
	\begin{tikzpicture}
		\centering
		\begin{groupplot}[%
			group style={group size=3 by 3, 
				         horizontal sep=45pt, 
						 vertical sep=50pt},
			width=0.32\linewidth,
			height=0.32\linewidth,
			domain=2:5, 
			xtick={0,1,2,3,4}, 
			ymode=log, 
			xlabel={refinement level},
			ymajorgrids=true,
			grid style=dashed,
			every axis plot/.append style={line width=1.1pt},
			cycle list name=myplots,
		]
		\nextgroupplot[ylabel={$\mathrm{AFW}_1$}]
		\addplot+[discard if not={Bnd}{10.0}] 
				table [x=ref, y=sigma_error, col sep=comma] 
				{data/transverse_isotropic_afw_1.csv};
				\label{plot:trans-at-hdiv-alpha1}
		\addplot+[discard if not={Bnd}{1000.0}] 
				table [x=ref, y=sigma_error, col sep=comma] 
				{data/transverse_isotropic_afw_1.csv};
				\label{plot:trans-at-hdiv-alpha2}
		\addplot+[discard if not={Bnd}{100000.0}] 
				table [x=ref, y=sigma_error, col sep=comma] 
				{data/transverse_isotropic_afw_1.csv};
				\label{plot:trans-at-hdiv-alpha3}
		\nextgroupplot
		\addplot+[discard if not={Bnd}{10.0}] 
				table [x=ref, y=displacement_error, col sep=comma] 
				{data/transverse_isotropic_afw_1.csv};
		\addplot+[discard if not={Bnd}{1000.0}] 
				table [x=ref, y=displacement_error, col sep=comma] 
				{data/transverse_isotropic_afw_1.csv};
		\addplot+[discard if not={Bnd}{100000.0}] 
				table [x=ref, y=displacement_error, col sep=comma] 
				{data/transverse_isotropic_afw_1.csv};
		\nextgroupplot
		\addplot+[discard if not={Bnd}{10.0}] 
				table [x=ref, y=omega_err, col sep=comma] 
				{data/transverse_isotropic_afw_1.csv};
		\addplot+[discard if not={Bnd}{1000.0}] 
				table [x=ref, y=omega_err, col sep=comma] 
				{data/transverse_isotropic_afw_1.csv};
		\addplot+[discard if not={Bnd}{100000.0}] 
				table [x=ref, y=omega_err, col sep=comma] 
				{data/transverse_isotropic_afw_1.csv};
		\nextgroupplot[ylabel={$\mathrm{AFW}_2$}]
		\addplot+[discard if not={Bnd}{10.0}] 
				table [x=ref, y=sigma_error, col sep=comma] 
				{data/transverse_isotropic_afw_2.csv};
		\addplot+[discard if not={Bnd}{1000.0}] 
				table [x=ref, y=sigma_error, col sep=comma] 
				{data/transverse_isotropic_afw_2.csv};
		\addplot+[discard if not={Bnd}{100000.0}] 
				table [x=ref, y=sigma_error, col sep=comma] 
				{data/transverse_isotropic_afw_2.csv};
		\nextgroupplot
		\addplot+[discard if not={Bnd}{10.0}] 
				table [x=ref, y=displacement_error, col sep=comma] 
				{data/transverse_isotropic_afw_2.csv};
		\addplot+[discard if not={Bnd}{1000.0}] 
				table [x=ref, y=displacement_error, col sep=comma] 
				{data/transverse_isotropic_afw_2.csv};
		\addplot+[discard if not={Bnd}{100000.0}] 
				table [x=ref, y=displacement_error, col sep=comma] 
				{data/transverse_isotropic_afw_2.csv};
		\nextgroupplot
		\addplot+[discard if not={Bnd}{10.0}] 
				table [x=ref, y=omega_err, col sep=comma] 
				{data/transverse_isotropic_afw_2.csv};
		\addplot+[discard if not={Bnd}{1000.0}] 
				table [x=ref, y=omega_err, col sep=comma] 
				{data/transverse_isotropic_afw_2.csv};
		\addplot+[discard if not={Bnd}{100000.0}] 
				table [x=ref, y=omega_err, col sep=comma] 
				{data/transverse_isotropic_afw_2.csv};
		\nextgroupplot[ylabel={$\mathrm{AFW}_3$}]
		\addplot+[discard if not={Bnd}{10.0}] 
				table [x=ref, y=sigma_error, col sep=comma] 
				{data/transverse_isotropic_afw_3.csv};
		\addplot+[discard if not={Bnd}{1000.0}] 
				table [x=ref, y=sigma_error, col sep=comma] 
				{data/transverse_isotropic_afw_3.csv};
		\addplot+[discard if not={Bnd}{100000.0}] 
				table [x=ref, y=sigma_error, col sep=comma] 
				{data/transverse_isotropic_afw_3.csv};
		\nextgroupplot
		\addplot+[discard if not={Bnd}{10.0}] 
				table [x=ref, y=displacement_error, col sep=comma] 
				{data/transverse_isotropic_afw_3.csv};
		\addplot+[discard if not={Bnd}{1000.0}] 
				table [x=ref, y=displacement_error, col sep=comma] 
				{data/transverse_isotropic_afw_3.csv};
		\addplot+[discard if not={Bnd}{100000.0}] 
				table [x=ref, y=displacement_error, col sep=comma] 
				{data/transverse_isotropic_afw_3.csv};
		\nextgroupplot
		\addplot+[discard if not={Bnd}{10.0}] 
				table [x=ref, y=omega_err, col sep=comma] 
				{data/transverse_isotropic_afw_3.csv};
		\addplot+[discard if not={Bnd}{1000.0}] 
				table [x=ref, y=omega_err, col sep=comma] 
				{data/transverse_isotropic_afw_3.csv};
		\addplot+[discard if not={Bnd}{100000.0}] 
				table [x=ref, y=omega_err, col sep=comma] 
				{data/transverse_isotropic_afw_3.csv};
		\end{groupplot}
		\node at ($(group c1r3.south) + (0,-2)$) 
			{(A) $\|\mat{\sigma}-\mat{\sigma}_h\|_{\div}$};
		\node at ($(group c2r3.south) + (0,-2)$) {(B) $\|\vec{u}-\vec{u}_h\|$}; 
		\node at ($(group c3r3.south) + (0,-2)$) 
			{(C) $\|\mat{\omega} - \mat{\omega}_h\|$};;
	\end{tikzpicture}
	\caption{Numerical results for the 2D transversely isotropic example
		in \cref{sec:trans_iso} for
		$\delta = 10$ (\ref*{plot:trans-at-hdiv-alpha1}),
		$\delta = 10^3$ (\ref*{plot:trans-at-hdiv-alpha2}), and
		$\delta = 10^5$ (\ref*{plot:trans-at-hdiv-alpha3}) with the 
		$\mathrm{AFW}_k$, $k \in \{1,2,3\}$, schemes.
      The errors are displayed as in \cref{fig:le2d-errors}. 	
		Observe that the stress errors for the $\mathrm{AFW}_1$  and 
		$\mathrm{AFW}_2$ schemes are not zero up to roundoff errors,
      even though $\mat{\sigma} \equiv \mat{0}$.
      Meanwhile, $\mathrm{AFW}_3$ produces a stress-free solution.
      Note that the exact rotation tensor $\mat{\omega}$
      is only included in the discrete space $\Xi_h$ for the $\mathrm{AFW}_k$
      scheme with $k\geq 3$.
   }
	\label{fig:trans2d-afw-errors}
\end{figure}

\subsection{Example 3: Polar Fluids}
\label{sec:polar}

In the previous two examples, we have seen that the schemes for the 
strongly symmetric formulation \cref{eq:linear_laws_weak_strong_sym_fem}
and the weakly symmetric formulation \cref{eq:linear_laws_weak_weak_sym_fem}
behaved similarly with respect to the scaling parameter $\delta$ provided 
that the exact rotation tensor satisfied $\mat{\omega} \in \Xi_h$. In the 
example in \cref{sec:trans_iso}, the condition $\mat{\omega} \in \Xi_h$ required 
sufficiently high-order elements. We now turn to an example where the inclusion 
$\mat{\omega} \in \Xi_h$ is not possible for any piecewise polynomial space 
$\Xi_h$.

Consider polar fluids --- fluids 
composed by constituents endowed with a molecular orientation that while 
retaining fluid flow properties exhibit anisotropic elastic response.
Following Leslie's and Ericksen's works \cite{ericksen,leslie,leslie2} if we 
introduce a director field $\vec{\nu}:\Omega\to \mathbb{R}^d$ describing 
the average orientation of the molecular constituents, we can assume the stress 
response of the polar fluid subject to the divergence constraint 
$\nabla\cdot \vec{u} = g_{\mathrm{div}}$ is given by
\begin{equation}
	\label{eq:constitutive_law_polar}
    \mat{\sigma} = 2\mu \symgrad(\vec{u}) - {p\mat{I}} 
		+ K_{F} \nabla\vec{\nu}^{\top}\nabla\vec{\nu},
\end{equation}
where $\mu$ is the fluid viscosity, $p:\Omega\to \mathbb{R}$ is a Lagrange
multiplier enforcing the divergence constraint, 
and $K_F : \Omega \to \mathbb{R}$ is known as the Frank constant.
In general, $K_F$ and $\vec{\nu}$ depend nonlinearly on $\vec{u}$ and $p$
(see e.g. \cite[eq. (6.17)]{farrell} for compressible, 
inviscid liquid crystals).
Here, we assume that $K_F$ and $\vec{\nu}$ are given so that 
\cref{eq:constitutive_law_polar} is of the form 
\cref{eq:constitutive_law_fluid}.
To complete the specification of the variational form 
\cref{eq:linear_laws_weak_strong_sym} or \cref{eq:linear_laws_weak_weak_sym},
we also need to specify $\int_{\Omega} \tr\mat{\sigma} \d{x}$, which
we take to be zero since $\mat{\sigma} \equiv \mat{0}$ in the examples below.

\subsubsection{Two dimensions}
\label{sec:polar-2d}

Suppose we choose the velocity and pressure to be
\begin{align*}
	\vec{u}(x,y) = -\frac{\delta}{\mu}
	\begin{bmatrix}
		-\cos(x)\cosh(y)\\
		\sin(x)\sinh(y)
	\end{bmatrix}
	\quad \text{and} \quad
	p(x,y) = -\delta \sin(x)\cosh(y)
\end{align*}
and $\vec{\nu}$ and $K_F$ to be
\begin{align*}
	\vec{\nu}(x,y) =  \begin{bmatrix}
		x \\ y
	\end{bmatrix}
	\quad \text{and} \quad
	 K_F(x,y) = \delta \sin(x)\cosh(y).
\end{align*}
Then, a direct calculation shows that for any value of $\mu$ and $\delta$, 
we have
\begin{align*}
	\mat{\sigma} \equiv \mat{0} 
	\quad \text{and} \quad 
	\mat{\omega} = -\frac{\delta}{\mu} \begin{bmatrix}
		0 & -\cos(x) \sinh(y) \\
		\cos(x) \sinh(y) & 0
	\end{bmatrix}.
\end{align*}
Note that the inclusion $\mat{\omega} \in \Xi_h$ is not possible for any
piecewise polynomial space $\Xi_h$.

The numerical results for the JMK scheme with $\texttt{maxh} = 1/32$,
the $\mathrm{HZ}_3$ scheme with $\texttt{maxh} = 1/8$, and the 
$\mathrm{AFW}_3$ scheme with $\texttt{maxh} = 1/8$ are displayed
in \cref{fig:polar2d-errors}. We observe that for the strongly symmetric
schemes, the stress errors behave similarly as in the previous two 
examples, while for the weakly symmetric scheme $\mathrm{AFW}_3$, the stress
errors are well above solver tolerances and scale like $\delta$.

\begin{figure}[htbp]
	\centering
	\begin{tikzpicture}
		\centering
		\begin{groupplot}[%
			group style={group size=3 by 3, 
				         horizontal sep=45pt, 
						 vertical sep=50pt},
			width=0.32\linewidth,
			height=0.32\linewidth,
			domain=2:5, 
			xtick={0,1,2,3,4}, 
			ymode=log, 
			xlabel={refinement level},
			ymajorgrids=true,
			grid style=dashed,
			every axis plot/.append style={line width=1.1pt},
			cycle list name=myplots,
		]
		\nextgroupplot[ylabel={JMK}]
		\addplot+[discard if not={Bnd}{10.0}] 
				table [x=ref, y=sigma_error, col sep=comma] 
				{data/polar_jm_1.csv};
				\label{plot:polar-jm-hdiv-alpha1}
		\addplot+[discard if not={Bnd}{1000.0}] 
				table [x=ref, y=sigma_error, col sep=comma] 
				{data/polar_jm_1.csv};
				\label{plot:polar-jm-hdiv-alpha2}
		\addplot+[discard if not={Bnd}{100000.0}] 
				table [x=ref, y=sigma_error, col sep=comma] 
				{data/polar_jm_1.csv};
				\label{plot:polar-jm-hdiv-alpha3}
		\nextgroupplot
		\addplot+[discard if not={Bnd}{10.0}] 
				table [x=ref, y=displacement_error, col sep=comma] 
				{data/polar_jm_1.csv};
		\addplot+[discard if not={Bnd}{1000.0}] 
				table [x=ref, y=displacement_error, col sep=comma] 
				{data/polar_jm_1.csv};
		\addplot+[discard if not={Bnd}{100000.0}] 
				table [x=ref, y=displacement_error, col sep=comma] 
				{data/polar_jm_1.csv};
		\nextgroupplot[group/empty plot]
		\nextgroupplot[ylabel={$\mathrm{HZ}_3$}]
		\addplot+[discard if not={Bnd}{10.0}] 
				table [x=ref, y=sigma_error, col sep=comma] 
				{data/polar_hz_3.csv};
		\addplot+[discard if not={Bnd}{1000.0}] 
				table [x=ref, y=sigma_error, col sep=comma] 
				{data/polar_hz_3.csv};
		\addplot+[discard if not={Bnd}{100000.0}] 
				table [x=ref, y=sigma_error, col sep=comma] 
				{data/polar_hz_3.csv};
		\nextgroupplot
		\addplot+[discard if not={Bnd}{10.0}] 
				table [x=ref, y=displacement_error, col sep=comma] 
				{data/polar_hz_3.csv};
		\addplot+[discard if not={Bnd}{1000.0}] 
				table [x=ref, y=displacement_error, col sep=comma] 
				{data/polar_hz_3.csv};
		\addplot+[discard if not={Bnd}{100000.0}] 
				table [x=ref, y=displacement_error, col sep=comma] 
				{data/polar_hz_3.csv};
		\nextgroupplot[group/empty plot]
		\nextgroupplot[ylabel={$\mathrm{AFW}_3$}]
		\addplot+[discard if not={Bnd}{10.0}] 
				table [x=ref, y=sigma_error, col sep=comma] 
				{data/polar_afw_3.csv};
		\addplot+[discard if not={Bnd}{1000.0}] 
				table [x=ref, y=sigma_error, col sep=comma] 
				{data/polar_afw_3.csv};
		\addplot+[discard if not={Bnd}{100000.0}] 
				table [x=ref, y=sigma_error, col sep=comma] 
				{data/polar_afw_3.csv};
		\nextgroupplot
		\addplot+[discard if not={Bnd}{10.0}] 
				table [x=ref, y=displacement_error, col sep=comma] 
				{data/polar_afw_3.csv};
		\addplot+[discard if not={Bnd}{1000.0}] 
				table [x=ref, y=displacement_error, col sep=comma] 
				{data/polar_afw_3.csv};
		\addplot+[discard if not={Bnd}{100000.0}] 
				table [x=ref, y=displacement_error, col sep=comma] 
				{data/polar_afw_3.csv};
		\nextgroupplot
		\addplot+[discard if not={Bnd}{10.0}] 
				table [x=ref, y=omega_err, col sep=comma] 
				{data/polar_afw_3.csv};
		\addplot+[discard if not={Bnd}{1000.0}] 
				table [x=ref, y=omega_err, col sep=comma] 
				{data/polar_afw_3.csv};
		\addplot+[discard if not={Bnd}{100000.0}] 
				table [x=ref, y=omega_err, col sep=comma] 
				{data/polar_afw_3.csv};
		\end{groupplot}
		\node at ($(group c1r3.south) + (0,-2)$) 
			{(A) $\|\mat{\sigma}-\mat{\sigma}_h\|_{\div}$};
		\node at ($(group c2r3.south) + (0,-2)$) {(B) $\|\vec{u}-\vec{u}_h\|$};
		\node at ($(group c3r3.south) + (0,-2)$) 
			{(C) $\|\mat{\omega} - \mat{\omega}_h\|$};
	\end{tikzpicture}
	\caption{Numerical results for the 2D polar fluid example
		in \cref{sec:polar-2d} for
		$\delta = 10$ (\ref*{plot:polar-jm-hdiv-alpha1}),
		$\delta = 10^3$ (\ref*{plot:polar-jm-hdiv-alpha2}), and
		$\delta = 10^5$ (\ref*{plot:polar-jm-hdiv-alpha3}) with the 
		JMK scheme, the $\mathrm{HZ}_3$ scheme, and 
		the $\mathrm{AFW}_3$ scheme.
		The errors are displayed as in \cref{fig:le2d-errors}. 
		Observe that the strongly symmetric schemes JMK and $\mathrm{HZ}_3$
		have similar behavior to the strongly symmetric schemes in the previous 
		two examples, while
		the stress errors for $\mathrm{AFW}_3$ are well above solver tolerances
		and scale like $\delta$.}
	\label{fig:polar2d-errors}
\end{figure}
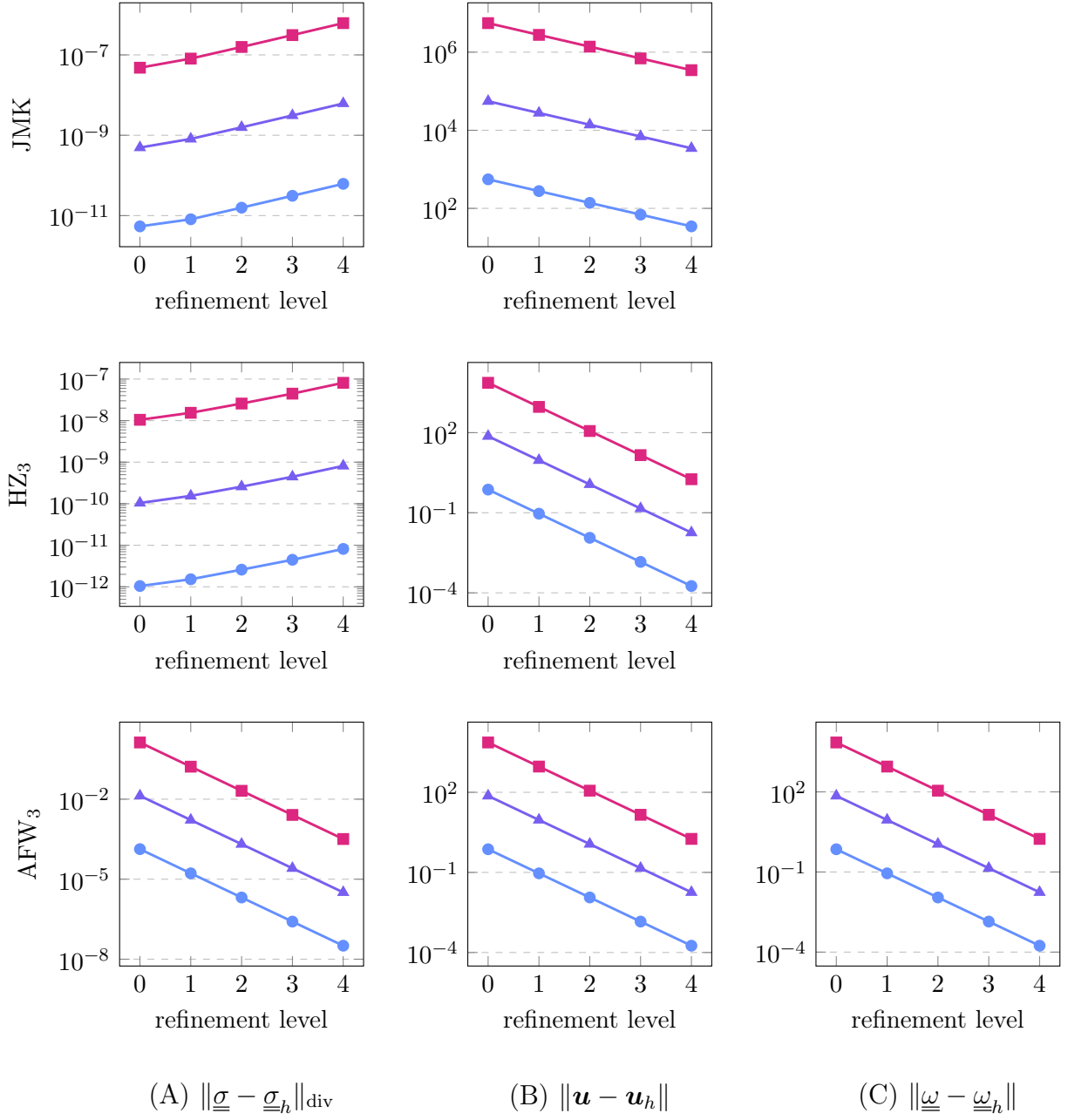

\subsubsection{Three dimensions}
\label{sec:polar-3d}

Similar phenomena occur in three dimensions. 
Consider the following velocity and pressure fields:
\begin{align*}
	\vec{u}(x, y, z) = -\frac{\delta}{2\mu} \begin{bmatrix}
		x + 2\sin(y) \\
		\frac{3}{2} y + \frac{1}{4} \sin(2y) \\
		z 
	\end{bmatrix}
	\quad \text{and} \quad 
	p(x, y, z) \equiv 0
\end{align*}
and $\vec{\nu}$ and $K_F$ to be 
\begin{align*}
	\vec{\nu}(x, y, z) = \begin{bmatrix}
		x + \sin(y) \\
		y \\ 
		z
	\end{bmatrix}
	\quad \text{and} \quad
	K_F(x, y, z) = \delta.
\end{align*}
The numerical results for the JMK and $\mathrm{AFW}_1$ schemes with 
$\texttt{maxh} = 1/4$ are displayed in \cref{fig:polar3d-errors}. Note that 
the behavior of the methods is analogous to the behavior in two dimensions.
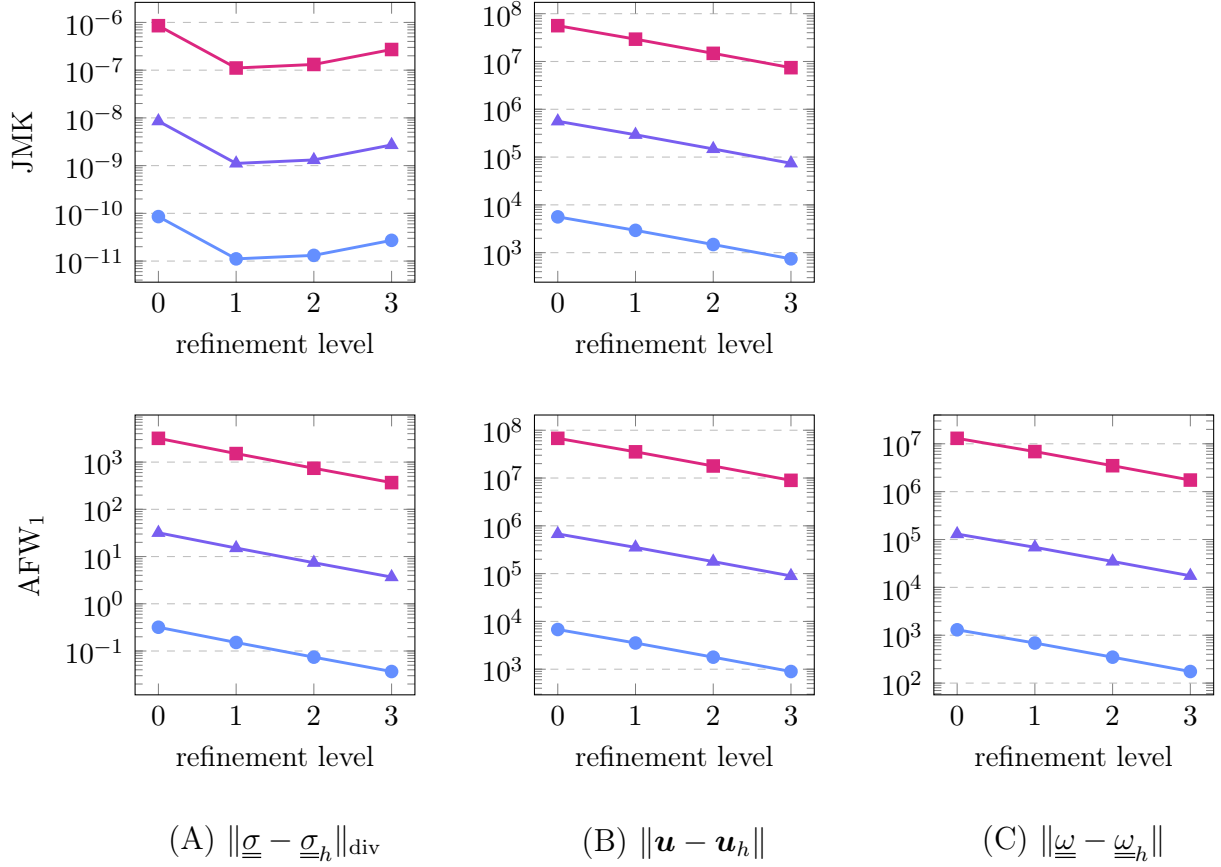
\begin{figure}[htbp]
	\centering
	\begin{tikzpicture}
		\centering
		\begin{groupplot}[%
			group style={group size=3 by 2, 
				         horizontal sep=45pt, 
						 vertical sep=50pt},
			width=0.32\linewidth,
			height=0.32\linewidth,
			domain=2:5, 
			xtick={0,1,2,3,4}, 
			ymode=log, 
			xlabel={refinement level},
			ymajorgrids=true,
			grid style=dashed,
			every axis plot/.append style={line width=1.1pt},
			cycle list name=myplots,
		]
		\nextgroupplot[ylabel={JMK}]
		\addplot+[discard if not={Bnd}{10.0}] 
				table [x=ref, y=sigma_error, col sep=comma] 
				{data/polar_3d_jm_1.csv};
				\label{plot:polar3d-jm-hdiv-alpha1}
		\addplot+[discard if not={Bnd}{1000.0}] 
				table [x=ref, y=sigma_error, col sep=comma] 
				{data/polar_3d_jm_1.csv};
				\label{plot:polar3d-jm-hdiv-alpha2}
		\addplot+[discard if not={Bnd}{100000.0}] 
				table [x=ref, y=sigma_error, col sep=comma] 
				{data/polar_3d_jm_1.csv};
				\label{plot:polar3d-jm-hdiv-alpha3}
		\nextgroupplot
		\addplot+[discard if not={Bnd}{10.0}] 
				table [x=ref, y=displacement_error, col sep=comma] 
				{data/polar_3d_jm_1.csv};
		\addplot+[discard if not={Bnd}{1000.0}] 
				table [x=ref, y=displacement_error, col sep=comma] 
				{data/polar_3d_jm_1.csv};
		\addplot+[discard if not={Bnd}{100000.0}] 
				table [x=ref, y=displacement_error, col sep=comma] 
				{data/polar_3d_jm_1.csv};
		\nextgroupplot[group/empty plot]
		\nextgroupplot[ylabel={$\mathrm{AFW}_1$}]
		\addplot+[discard if not={Bnd}{10.0}] 
				table [x=ref, y=sigma_error, col sep=comma] 
				{data/polar_3d_afw_1.csv};
		\addplot+[discard if not={Bnd}{1000.0}] 
				table [x=ref, y=sigma_error, col sep=comma] 
				{data/polar_3d_afw_1.csv};
		\addplot+[discard if not={Bnd}{100000.0}] 
				table [x=ref, y=sigma_error, col sep=comma] 
				{data/polar_3d_afw_1.csv};
		\nextgroupplot
		\addplot+[discard if not={Bnd}{10.0}] 
				table [x=ref, y=displacement_error, col sep=comma] 
				{data/polar_3d_afw_1.csv};
		\addplot+[discard if not={Bnd}{1000.0}] 
				table [x=ref, y=displacement_error, col sep=comma] 
				{data/polar_3d_afw_1.csv};
		\addplot+[discard if not={Bnd}{100000.0}] 
				table [x=ref, y=displacement_error, col sep=comma] 
				{data/polar_3d_afw_1.csv};
		\nextgroupplot[]
		\addplot+[discard if not={Bnd}{10.0}] 
				table [x=ref, y=omega_err, col sep=comma] 
				{data/polar_3d_afw_1.csv};
		\addplot+[discard if not={Bnd}{1000.0}] 
				table [x=ref, y=omega_err, col sep=comma] 
				{data/polar_3d_afw_1.csv};
		\addplot+[discard if not={Bnd}{100000.0}] 
				table [x=ref, y=omega_err, col sep=comma] 
				{data/polar_3d_afw_1.csv};
		\end{groupplot}
		\node at ($(group c1r3.south) + (0,3.5)$) 
			{(A) $\|\mat{\sigma}-\mat{\sigma}_h\|_{\div}$};
		\node at ($(group c2r3.south) + (0,3.5)$) 
			{(B) $\|\vec{u}-\vec{u}_h\|$};
		\node at ($(group c3r3.south) + (0,3.5)$) 
			{(C) $\|\mat{\omega}-\mat{\omega}_h\|$};
	\end{tikzpicture}
	\caption{Numerical results for the 3D polar fluid example
		in \cref{sec:polar-3d} for
		$\delta = 10$ (\ref*{plot:polar3d-jm-hdiv-alpha1}),
		$\delta = 10^3$ (\ref*{plot:polar3d-jm-hdiv-alpha2}), and
		$\delta = 10^5$ (\ref*{plot:polar3d-jm-hdiv-alpha3}) with the 
		JMK and $\mathrm{AFW}_1$ schemes.
		The errors are displayed as in \cref{fig:le2d-errors}. 
		Observe that the strongly symmetric schemes JMK has similar behavior to 
		the strongly symmetric schemes in the previous 
		two examples, while the stress errors for $\mathrm{AFW}_1$ are well 
		above solver tolerances and scale like $\delta$.}
	\label{fig:polar3d-errors}
\end{figure}

\subsection{Summarizing example and material robustness}
\label{sec:ex_summary}

We have seen that any scheme that preserves angular momentum,
i.e.~is strongly symmetric, always produces a zero stress state, up to 
roundoff errors, whenever the true stress vanishes. However, the schemes
that violate the angular momentum conservation, 
i.e.~the weakly symmetric ones, can produce arbitrarily large stress states
unless the exact rotation tensor belongs to the discrete space $\Xi_h$. 
In general, the solution to \cref{eq:linear_laws_strong} will be a linear 
combination of stress-free and stressed states. To illustrate the
behavior of the schemes in this case, we add a smooth component to the exact 
velocity in the 2D polar fluid example in \cref{sec:polar-2d}:
\begin{align*}
	\vec{u}(x,y) = -\frac{\delta}{\mu}
	\begin{bmatrix}
		-\cos(x)\cosh(y)\\
		\sin(x)\sinh(y)
	\end{bmatrix} + \begin{bmatrix}
		\cos(x)y\\
		\sin(y)
	\end{bmatrix}
	\quad \text{and} \quad
	p(x,y) = -\delta \sin(x)\cosh(y)
\end{align*}
and define $\mat{\sigma}$ so that \cref{eq:linear_constitutive_law} holds 
with $\lambda = \infty$ for the same choice of $\mat{F}$ as in 
\cref{sec:polar-2d}. In particular, the exact stress is independent of 
$\delta$. In \cref{fig:polarextra-errors}, we observe that as 
$\delta \to \infty$ the stress error curves overlap up to solver tolerances 
for strongly symmetric schemes, while the stress errors scale like $\delta$ for 
weakly symmetric schemes.

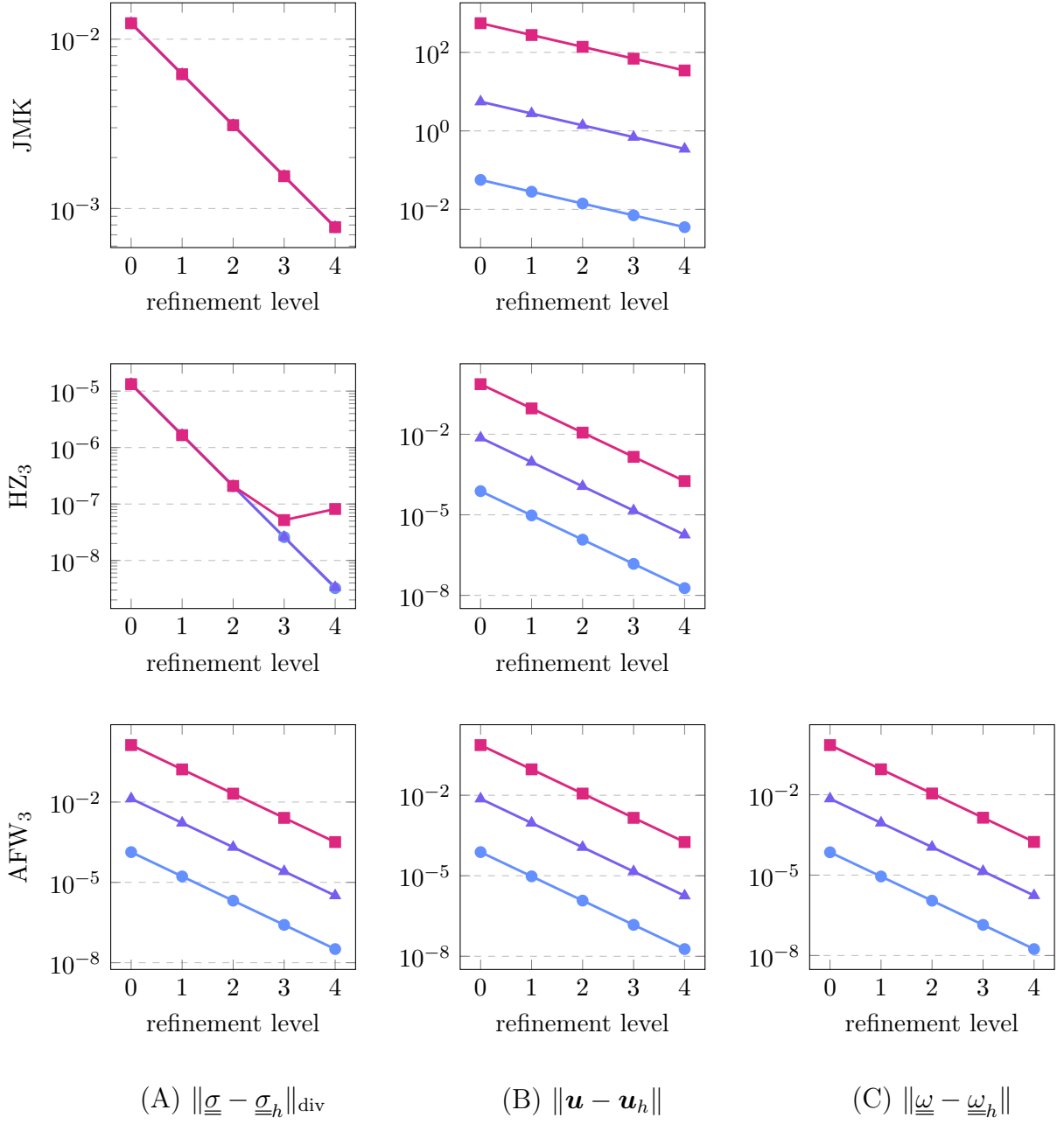
\begin{figure}[htbp]
	\centering
	\begin{tikzpicture}
		\centering
		\begin{groupplot}[%
			group style={group size=3 by 3, 
				         horizontal sep=45pt, 
						 vertical sep=50pt},
			width=0.32\linewidth,
			height=0.32\linewidth,
			domain=2:5, 
			xtick={0,1,2,3,4}, 
			ymode=log, 
			xlabel={refinement level},
			ymajorgrids=true,
			grid style=dashed,
			every axis plot/.append style={line width=1.1pt},
			cycle list name=myplots,
		]
		\nextgroupplot[ylabel={JMK}]
		\addplot+[discard if not={Bnd}{10.0}] 
				table [x=ref, y=sigma_error, col sep=comma] 
				{data/polar_extra_jm_1.csv};
				\label{plot:polar-extra-jm-hdiv-alpha1}
		\addplot+[discard if not={Bnd}{1000.0}] 
				table [x=ref, y=sigma_error, col sep=comma] 
				{data/polar_extra_jm_1.csv};
				\label{plot:polar-extra-jm-hdiv-alpha2}
		\addplot+[discard if not={Bnd}{100000.0}] 
				table [x=ref, y=sigma_error, col sep=comma] 
				{data/polar_extra_jm_1.csv};
				\label{plot:polar-extra-jm-hdiv-alpha3}
		\nextgroupplot
		\addplot+[discard if not={Bnd}{10.0}] 
				table [x=ref, y=displacement_error, col sep=comma] 
				{data/polar_extra_jm_1.csv};
		\addplot+[discard if not={Bnd}{1000.0}] 
				table [x=ref, y=displacement_error, col sep=comma] 
				{data/polar_extra_jm_1.csv};
		\addplot+[discard if not={Bnd}{100000.0}] 
				table [x=ref, y=displacement_error, col sep=comma] 
				{data/polar_extra_jm_1.csv};
		\nextgroupplot[group/empty plot]
		\nextgroupplot[ylabel={$\mathrm{HZ}_3$}]
		\addplot+[discard if not={Bnd}{10.0}] 
				table [x=ref, y=sigma_error, col sep=comma] 
				{data/polar_extra_hz_3.csv};
		\addplot+[discard if not={Bnd}{1000.0}] 
				table [x=ref, y=sigma_error, col sep=comma] 
				{data/polar_extra_hz_3.csv};
		\addplot+[discard if not={Bnd}{100000.0}] 
				table [x=ref, y=sigma_error, col sep=comma] 
				{data/polar_extra_hz_3.csv};
		\nextgroupplot
		\addplot+[discard if not={Bnd}{10.0}] 
				table [x=ref, y=displacement_error, col sep=comma] 
				{data/polar_extra_hz_3.csv};
		\addplot+[discard if not={Bnd}{1000.0}] 
				table [x=ref, y=displacement_error, col sep=comma] 
				{data/polar_extra_hz_3.csv};
		\addplot+[discard if not={Bnd}{100000.0}] 
				table [x=ref, y=displacement_error, col sep=comma] 
				{data/polar_extra_hz_3.csv};
		\nextgroupplot[group/empty plot]
		\nextgroupplot[ylabel={$\mathrm{AFW}_3$}]
		\addplot+[discard if not={Bnd}{10.0}] 
				table [x=ref, y=sigma_error, col sep=comma] 
				{data/polar_extra_afw_3.csv};
		\addplot+[discard if not={Bnd}{1000.0}] 
				table [x=ref, y=sigma_error, col sep=comma] 
				{data/polar_extra_afw_3.csv};
		\addplot+[discard if not={Bnd}{100000.0}] 
				table [x=ref, y=sigma_error, col sep=comma] 
				{data/polar_extra_afw_3.csv};
		\nextgroupplot
		\addplot+[discard if not={Bnd}{10.0}] 
				table [x=ref, y=displacement_error, col sep=comma] 
				{data/polar_extra_afw_3.csv};
		\addplot+[discard if not={Bnd}{1000.0}] 
				table [x=ref, y=displacement_error, col sep=comma] 
				{data/polar_extra_afw_3.csv};
		\addplot+[discard if not={Bnd}{100000.0}] 
				table [x=ref, y=displacement_error, col sep=comma] 
				{data/polar_extra_afw_3.csv};
		\nextgroupplot
		\addplot+[discard if not={Bnd}{10.0}] 
				table [x=ref, y=omega_err, col sep=comma] 
				{data/polar_extra_afw_3.csv};
		\addplot+[discard if not={Bnd}{1000.0}] 
				table [x=ref, y=omega_err, col sep=comma] 
				{data/polar_extra_afw_3.csv};
		\addplot+[discard if not={Bnd}{100000.0}] 
				table [x=ref, y=omega_err, col sep=comma] 
				{data/polar_extra_afw_3.csv};
		\end{groupplot}
		\node at ($(group c1r3.south) + (0,-2)$) 
			{(A) $\|\mat{\sigma}-\mat{\sigma}_h\|_{\div}$};
		\node at ($(group c2r3.south) + (0,-2)$) {(B) $\|\vec{u}-\vec{u}_h\|$};
		\node at ($(group c3r3.south) + (0,-2)$) 
			{(C) $\|\mat{\omega} - \mat{\omega}_h\|$};
	\end{tikzpicture}
	\caption{Numerical results for the 2D polar fluid example with a 
		stressed configuration
		as presented in \cref{sec:ex_summary} for
		$\delta = 10$ (\ref*{plot:polar-extra-jm-hdiv-alpha1}),
		$\delta = 10^3$ (\ref*{plot:polar-extra-jm-hdiv-alpha2}), and
		$\delta = 10^5$ (\ref*{plot:polar-extra-jm-hdiv-alpha3}) with the 
		JMK scheme, the $\mathrm{HZ}_3$ scheme, and 
		the $\mathrm{AFW}_3$ scheme.
		The errors are displayed as in \cref{fig:le2d-errors}. 
		Observe that the strongly symmetric schemes JMK and $\mathrm{HZ}_3$
		produce errors that are independent of $\delta$, while
		the stress errors for $\mathrm{AFW}_3$ are well above solver tolerances
		and scale like $\delta$.}
	\label{fig:polarextra-errors}
\end{figure}

In summary, the stress errors for approximations to 
\cref{eq:linear_laws_strong} may be sensitive to shifts in the exact solution 
by a stress-free state. The discrete stresses from schemes that exactly
preserve angular momentum are unaffected by these shifts, while the stresses
from schemes that violate the preservation of angular momentum produce 
arbitrarily large errors. Based on these observations, 
we say a scheme is \textit{material robust} for the
linearized system \cref{eq:linear_laws_strong}
if the approximation to any stress-free state is stress-free pointwise.
As a consequence, the stress errors of a material robust scheme 
for \cref{eq:linear_laws_strong} remain bounded 
whenever the exact solution to \cref{eq:linear_laws_strong} is shifted by any 
zero-stress state. Thus, the results in this section suggest that 
the schemes preserving angular momentum are material robust, while 
the schemes that do not are not material robust. This finding will be confirmed
rigorously in the next section.

\section{Supporting theory}
\label{sec:theory}

We now seek a general theory to explain the numerical results above
and rigorously justify the material robustness of the strongly symmetric 
schemes. We first present a series of results in an abstract Hilbert space 
setting and then apply them to the particular examples above.
Let $\mathcal{V}$ and $\mathcal{Q}$ be Hilbert spaces 
with inner products $(\cdot,\cdot)_{\mathcal{V}}$ and 
$(\cdot,\cdot)_{\mathcal{Q}}$ and induced norms 
$\|\cdot\|_{\circ} := \sqrt{(\cdot,\cdot)_{\circ}}$ for 
$\circ \in \{ \mathcal{V}, \mathcal{Q} \}$.
Let $A(\cdot,\cdot) : \mathcal{V} \times \mathcal{V} \to \mathbb{R}$ and 
$B(\cdot,\cdot) : \mathcal{V} \times \mathcal{Q} \to \mathbb{R}$ be bounded 
bilinear forms.
For $F \in \mathcal{V}'$ and $G \in \mathcal{Q}'$, consider a 
generic saddle point problem: Find $u \in \mathcal{V}$ and $p \in \mathcal{Q}$ 
such that
\begin{subequations}
	\label{eq:saddle}
	\begin{alignat}{2}
		A(u, v) + B(v, p) &= F(v) \qquad & &\forall v \in \mathcal{V}, \\
		B(u, q) &= G(q) \qquad & &\forall q \in \mathcal{Q}.
	\end{alignat}
\end{subequations}
We assume that the bilinear forms $A(\cdot,\cdot)$ and $B(\cdot,\cdot)$
satisfy the usual Babu\v{s}ka-Brezzi conditions 
(see e.g. \cite[Theorem 4.2.3]{BoffiBrezziFortin13}) so
that solutions to \cref{eq:saddle} exist, are unique,
and depend continuously on $F$ and $G$. For instance, we assume that the
following inf-sup condition holds:
\begin{align}
	\label{eq:inf-sup-b-bilinear-continuous}
		\beta &:= \inf_{q \in \mathcal{Q}} \sup_{v \in \mathcal{V}} 
			\frac{B(v, q)}{\|v\|_{\mathcal{V}} \|q\|_{\mathcal{Q}}} > 0.
\end{align}
\noindent The particular examples we have seen in \cref{sec:general-setup}
are
\begin{itemize}
	\item Problem \cref{eq:linear_laws_weak_strong_sym}: 
		$\mathcal{V} = \Sigma^{\sym}$ and 
		$\mathcal{Q} = V$ defined in \cref{eq:sigma-sym-v-def}, 
		$A(\mat{\sigma}, \mat{\tau}) = a(\mat{\sigma}, \mat{\tau})$ defined in
		\cref{eq:a-bilinear}, and 
		$B(\mat{\tau}, \vec{v}) = b(\mat{\tau}, \vec{v})$ defined in
		\cref{eq:b-bilinear}.

	\item Problem \cref{eq:linear_laws_weak_weak_sym}: 
		$\mathcal{V} = \Sigma$ and, 
		$\mathcal{Q} =  V \times \Xi$ defined in 
		\cref{eq:sigma-sym-v-def,eq:sigma-xi-def}, 
		$A(\mat{\sigma}, \mat{\tau}) = a(\mat{\sigma}, \mat{\tau})$ defined in
		\cref{eq:a-bilinear}, and 
		$B(\mat{\tau}, (\vec{v}, \mat{\xi})) = b(\mat{\tau}, \vec{v}) 
			+ c(\mat{\tau}, \mat{\xi})$ defined in 
			\cref{eq:b-bilinear,eq:sigma-xi-def}.
\end{itemize}

Let $\mathcal{V}_h \subset \mathcal{V}$ and $\mathcal{Q}_h \subset \mathcal{Q}$ 
denote finite dimensional subspaces and consider the following conforming 
discretization of \cref{eq:saddle}: Find $u_h \in \mathcal{V}_h$ and 
$q_h \in \mathcal{Q}_h$ such that
\begin{subequations}
	\label{eq:saddle-discrete}
	\begin{alignat}{2}
		A(u_h, v_h) + B(v_h, p_h) &= F(v_h) \qquad 
			& &\forall v_h \in \mathcal{V}_h, \\
		B(u_h, q_h) &= G(q_h) \qquad & &\forall q_h \in \mathcal{Q}_h.
	\end{alignat}
\end{subequations}
Analogous to above, we assume that the discrete Babu\v{s}ka-Brezzi conditions 
hold (see e.g. \cite[Theorem 4.2.3]{BoffiBrezziFortin13}) so that solutions to 
\cref{eq:saddle-discrete} exist, are unique, and depend continuously 
on $F$ and $G$. For instance, we assume that the following discrete
inf-sup condition holds
\begin{align}
	\label{eq:inf-sup-b-bilinear-discrete}
	\beta_h &:= \inf_{q_h \in \mathcal{Q}_h} \sup_{v_h \in \mathcal{V}_h} 
		\frac{B(v_h, q_h)}{\|v_h\|_{\mathcal{V}} \|q_h\|_{\mathcal{Q}}} > 0.
\end{align}

\subsection{Structure-preservation}

The key property of the system \cref{eq:saddle} fundamental to material 
robustness is the following ``invariance" property: If $F(\cdot)$ is shifted
by $B(\cdot, r)$ for some $r \in \mathcal{Q}$, then only the Lagrange multiplier 
$p$ is shifted by $r$. More precisely, we have the following result.
\begin{lemma}
	\label{lem:saddle-invariance}
	Let $F \in \mathcal{V}'$ and $G \in \mathcal{Q}'$ be given, and let 
	$u \in \mathcal{V}$ and $p \in \mathcal{Q}$ satisfy \cref{eq:saddle}. 
	Given $r \in \mathcal{Q}$, let 
	$u_r \in \mathcal{V}$ and $p_r \in \mathcal{Q}$ satisfy
	\begin{subequations}
		\label{eq:saddle-perturbed}
		\begin{alignat}{2}
			A(u_r, v) + B(v, p_r) &= F(v) + B(v, r) \qquad 
				& &\forall v \in \mathcal{V}, \\
			B(u_r, q) &= G(q) \qquad & &\forall q \in \mathcal{Q}.
		\end{alignat}
	\end{subequations}
	Then, $u_r = u$ and $p_r = p + r$ for all $r \in \mathcal{Q}$.
\end{lemma}
\begin{proof}
	Direct verification shows that the $u$ and $p + r$ satisfy 
	\cref{eq:saddle-perturbed}, and so the result follows by the uniqueness
	of solutions.
\end{proof}

In view of \cref{lem:saddle-invariance}, we can define one notion of a 
structure-preserving discretization as follows. For $r \in \mathcal{Q}$, let 
$u_{h, r} \in \mathcal{V}_h$ and $p_{h, r} \in \mathcal{Q}_h$ satisfy	
\begin{subequations}
	\label{eq:saddle-perturbed-discrete}
	\begin{alignat}{2}
		A(u_{h, r}, v) + B(v, p_{h, r}) &= F(v) + B(v, r) \qquad & 
		&\forall v \in \mathcal{V}_h, \\
		B(u_{h, r}, q) &= G(q) \qquad & &\forall q \in \mathcal{Q}_h.
	\end{alignat}
\end{subequations}
Then, we say the discretization $\mathcal{V}_h \times \mathcal{Q}_h$
of \cref{eq:saddle} is \textit{structure-preserving} if for any choice 
of $F \in \mathcal{V}'$ and $G \in \mathcal{Q}'$, 
there holds
\begin{align}
	\label{eq:structure-preserving-def}
	u_{h, r} = u_h \qquad \forall r \in \mathcal{Q}.
\end{align}
That is, the discrete scheme \cref{eq:saddle-discrete} possesses the same
invariance property as the continuous problem \cref{eq:saddle}.
Note that $u_{h, r} = u_h$ for all $r \in \mathcal{Q}_h$ by the same reasoning 
as in \cref{lem:saddle-invariance}, and so posing condition 
\cref{eq:structure-preserving-def} over all $r \in \mathcal{Q}$ 
is nontrivial. In \cite[eq. (1.4)]{JohnLinkeMerdonNeilanRebholz17} in the 
context of incompressible flow, property \cref{eq:structure-preserving-def} 
is called the ``fundamental invariance property''.

Another common notion of a structure-preserving scheme arises in the case
that $G \equiv 0$. Then, the exact solution satisfies
\begin{align}
	\label{eq:b-bilinear-kernel}
	u \in \ker \mathcal{B} 
		&:= \{v \in \mathcal{V} : B(v, q) = 0 \ \forall q \in \mathcal{Q}\},
\end{align}
while the discrete solution satisfies
\begin{align}
	\label{eq:b-bilinear-kernel-discrete}
	u_h \in \ker \mathcal{B}_h 
		&:= \{v_h \in \mathcal{V}_h : B(v_h, q_h) = 0 
			\ \forall q \in \mathcal{Q}_h\}.
\end{align}
Here, one may call a scheme structure-preserving if 
$\ker \mathcal{B}_h \subseteq \ker \mathcal{B}$. For example, in 
incompressible flow problems, the condition 
$\ker \mathcal{B}_h \subseteq \ker \mathcal{B}$ means that the discrete flow
is pointwise divergence-free (and hence incompressible). The following 
result shows that these apparently different
notions of a structure-preserving scheme actually coincide.
\begin{lemma}
	\label{lem:saddle-invariance-discrete}
	For $F \in \mathcal{V}'$, $G \in \mathcal{Q}'$, 
	and $r \in \mathcal{Q}$, let $u_{h} \in \mathcal{V}_h$ and 
	$p_h \in \mathcal{Q}_h$ satisfy \cref{eq:saddle-discrete} 
	and $u_{h, r} \in \mathcal{V}_h$ and $p_{h, r} \in \mathcal{Q}_h$
	satisfy \cref{eq:saddle-perturbed-discrete}. 
	Then, 
	\begin{align}
		\label{eq:saddle-perturbed-discrete-invariance}
		u_{h, r} = u_h \quad \forall r \in \mathcal{Q} 
			\iff \ker \mathcal{B}_h \subseteq \ker \mathcal{B}.
	\end{align}
\end{lemma}
\begin{proof}
	First note that the following holds by definition:
	\begin{align}
		\label{eq:b-discrete-kernel-zero-implies-kernel-inclusion}
		B(v_h, r) = 0 \quad \forall v_h \in \ker \mathcal{B}_h, 
			\ \forall r \in \mathcal{Q} 
		\iff \ker \mathcal{B}_h \subseteq \ker \mathcal{B}.
	\end{align}
	The differences $e_{h, r} : = u_{h} - u_{h, r}$ and 
	$\epsilon_{h, r} := p_{h} - p_{h, r}$ satisfy
	\begin{alignat*}{2}
		A(e_{h, r}, v_h) + B(v_h, \epsilon_{h, r}) &= -B(v_h, r) \qquad 
			& &\forall v_h \in \mathcal{V}_h, \\
		B(e_{h, r}, q_h) &= 0 \qquad & &\forall q_h \in \mathcal{Q}_h.
	\end{alignat*}
	Taking $v_h \in \ker \mathcal{B}_h$ then gives
	\begin{align*}
		A(e_{h, r}, v_h) = -B(v_h, r) \qquad \forall v_h \in \ker \mathcal{B}_h.
	\end{align*}
	By the well-posedness of \cref{eq:saddle-discrete}, the above problem 
	is well-posed (see e.g. \cite[Theorem 4.2.2]{BoffiBrezziFortin13}), and so
	$e_{h, r} \equiv 0$ if and only if 
	$B(v_h, r) = 0$ for all $v_h \in \ker \mathcal{B}_h$.
	Thanks to \cref{eq:b-discrete-kernel-zero-implies-kernel-inclusion}, 
	$e_{h, r} \equiv 0$ for all $r \in \mathcal{Q}$ if and only if 
	$\ker \mathcal{B}_h \subseteq \ker \mathcal{B}$.
\end{proof}

As in \cref{lem:saddle-invariance}, we can also obtain an expression for
the discrete pressure as follows.
\begin{lemma}
	\label{lem:saddle-perturbed-discrete-pressure}
	Let $p_h$ and $p_{h, r}$ be as in \cref{lem:saddle-invariance-discrete}
	and suppose that $\ker \mathcal{B}_h \subseteq \ker \mathcal{B}$. 
	Then, there exists a unique linear projection operator 
	$\Phi_h : \mathcal{Q} \to \mathcal{Q}_h$ such that
	\begin{align}
		\label{eq:pressure-fortin}
		B(v_h, q - \Phi_h q) = 0 \qquad \forall v_h \in \mathcal{V}_h, 
			\ \forall q \in \mathcal{Q},
	\end{align}
	and $p_{h, r} = p_h + \Phi_h r$.
\end{lemma}
\begin{proof}
	The existence of $\Phi_h$ is an immediate consequence of 
	\cite[Proposition 5.1.2]{BoffiBrezziFortin13}. For uniqueness, let
	$\Psi_h : \mathcal{Q} \to \mathcal{Q}_h$ be another linear operator
	satisfying \cref{eq:pressure-fortin}. Then, for any
	$q \in \mathcal{Q}$, we have
	\begin{align*}
		B(v_h, (\Phi_h - \Psi_h)q) = 0 \qquad \forall v_h \in \mathcal{V}_h.
	\end{align*}
	Thanks to the discrete inf-sup condition 
	\cref{eq:inf-sup-b-bilinear-discrete}, there holds
	\begin{align*}
		\beta_h \| (\Phi_h - \Psi_h) q\|_{\mathcal{Q}} 
			\leq \sup_{v_h \in \mathcal{V}_h} 
				\frac{B(v_h, (\Phi_h - \Psi_h) q)}{\|v_h\|_{\mathcal{V}_h}} 
			= 0,
	\end{align*} 
	and so $\Phi_h q = \Psi_h q$ for all $q \in \mathcal{Q}$.

	From the proof of \cref{lem:saddle-invariance-discrete}, we have
	\begin{align*}
		B(v_h, p_h - p_{h, r}) = -B(v_h, r) = -B(v_h, \Phi_h r) 
			\qquad \forall v \in \mathcal{V}_h,
	\end{align*}
	where we used \cref{eq:pressure-fortin}. Arguing as above, 
	we obtain $p_{h, r} = p_h + \Phi_h r$.
\end{proof}

\begin{remark}
	\label{rem:b-special-form}
	For some problems, including those in \cref{sec:examples},
	$B(\cdot, \cdot)$ is of the form
	\begin{align*}
		B(v, q) := (D v, q)_{\mathcal{Q}},
	\end{align*}
	where $D : \mathcal{V} \to \mathcal{Q}$ is a bounded linear operator.
	In this setting, $\Phi_h$ satisfies
	\begin{align*}
		(D v_h, q - \Phi_h q)_{\mathcal{Q}} = 0 \qquad 
			\forall v_h \in \mathcal{V}_h, \ \forall q \in \mathcal{Q}.
	\end{align*}
	Thus, $\Phi_h$ is the  $(\cdot,\cdot)_{\mathcal{Q}}$-orthogonal
	projection onto $D \mathcal{V}_h$. A common way to ensure that 
	$\ker \mathcal{B}_h \subseteq \ker \mathcal{B}$
	is to choose $\mathcal{Q}_h = D \mathcal{V}_h$,
	in which case $\Phi_h$ is the  $(\cdot,\cdot)_{\mathcal{Q}}$-orthogonal
	projection onto $\mathcal{Q}_h$.
\end{remark}

\subsection{Expressing the examples in the form \cref{eq:saddle-perturbed}}

All the examples in \cref{sec:examples} 
were constructed to be precisely of the form \cref{eq:saddle-perturbed} with 
$F(\cdot) \equiv 0$. First consider the strongly symmetric problem 
\cref{eq:linear_laws_weak_strong_sym}.
We choose $r = \vec{w}$ for some $\vec{w} \in \mathcal{Q} = V$. 
We cannot choose any $\vec{w}$, as the 
constitutive law \cref{eq:linear_laws_weak_strong_sym_1} requires 
that the RHS has a specific form. In particular, we require that
\begin{align*}
	(\nabla \cdot \mat{\tau}, \vec{w}) 
		= \langle \mat{\tau} \vec{n}, \vec{g} \rangle_{\partial \Omega}
				+ (\mat{F}, \mat{\tau}) 
		\qquad \forall \mat{\tau} \in \Sigma^{\sym}
\end{align*}
for some suitable $\vec{g}$ and $\mat{F}$.
Supposing that $\vec{w}$ is smooth enough, we integrate the LHS by parts
to obtain 
\begin{align*}
	\langle \mat{\tau} \vec{n}, \vec{w} \rangle_{\partial \Omega}
	- (\symgrad(\vec{w}), \mat{\tau}) 
		= \langle \mat{\tau} \vec{n}, \vec{g} \rangle_{\partial \Omega}
				+ (\mat{F}, \mat{\tau}) 
		\qquad \forall \mat{\tau} \in \Sigma^{\sym},
\end{align*}
and so $\vec{g} = \vec{w}|_{\partial \Omega}$ and 
$\symgrad(\vec{w}) = -\mat{F}$. Since all domains in 
\cref{sec:examples} are contractible, such a choice of $\vec{w}$
is possible only if Saint-Venant's compatibility condition holds 
(see e.g. \cite[Theorem 3.2]{Amrouche06}):
$\nabla^{\top} \times (\nabla \times \mat{F}) = \mat{0}$, where the first
curl is taken row-wise and the $\nabla^{\top} \times$ denotes the columnwise curl. 
Provided that $\mat{F}$ satisfies this condition, the equation 
$\symgrad(\vec{w}) = -\mat{F}$ can be integrated to find $\vec{w}$.
Of course, this is equivalent to saying that $\vec{w}$ gives rise 
to a no-stress configuration; i.e.~$\mat{\sigma} \equiv \mat{0}$ and 
$\vec{u} = \vec{w}$ satisfy \cref{eq:linear_constitutive_law}. 

\begin{itemize}
	\item For linear isotropic materials in \cref{sec:lin_el}, 
	$\mat{F} \equiv \mat{0}$, and so the only choices for $\vec{w}$ are
	rigid body motions, for which \cref{eq:lin_el_exact} is one option.

	\item For the transversely isotropic solid in \cref{sec:trans_iso},
	$\mat{F} = \delta \vec{\nu} \otimes \vec{\nu}$, and so we choose  
	$\vec{\nu}$ to satisfy the Saint-Venant's compatibility condition.

	\item For the polar fluids in \cref{sec:polar,sec:ex_summary}, we have
	$\mat{F} = (K_F \nabla \vec{\nu}^{\top} \nabla \vec{\nu})^D/(2\mu)
	- (g_{\div}/d) \mat{\mathbb{I}}$. In \cref{sec:polar-2d}, we first choose 
	$\vec{\nu} = [x, \ y]^{\top}$ so that 
	$(\nabla \vec{\nu})^{\top} \nabla \vec{\nu} = \mat{\mathbb{I}}$ and then find 
	$\vec{w}$ such that $\symgrad(\vec{w}) = h(x, y) \mat{\mathbb{I}}$.
	Then, $K_F = h(x, y)$ and $g_{\div} = \nabla \cdot \vec{w}$.
	In \cref{sec:polar-3d}, we find $\vec{\nu}$ such that
	$(\nabla \vec{\nu})^{\top} \nabla \vec{\nu}$ satisfies Saint-Venant's
	compatibility condition.	
\end{itemize}
Thus, the RHS of \cref{eq:linear_laws_weak_strong_sym_1} for all the examples
in \cref{sec:examples} is of the form $b(\mat{\tau}, \vec{w})$ for suitably
chosen $\vec{w}$.

For the weakly symmetric problem \cref{eq:linear_laws_weak_weak_sym}, 
we have $r = (\vec{w}, \mat{\eta}) \in \mathcal{Q} = V \times \Xi$. 
As above, the choice of $\vec{w}$
and $\mat{\eta}$ must match the constitutive law 
\cref{eq:linear_laws_weak_weak_sym_1}:
\begin{align*}
	(\nabla \cdot \mat{\tau}, \vec{w}) + (\mat{\tau}, \mat{\eta}) = 
		\langle \mat{\tau} \vec{n}, \vec{g} \rangle_{\partial \Omega}
				+ (\mat{F}, \mat{\tau}) 
		\qquad \forall \mat{\tau} \in \Sigma.
\end{align*}
Assuming $\vec{w}$ is sufficiently smooth and integrating by parts gives
\begin{align*} 
		\langle \mat{\tau} \vec{n}, \vec{w} \rangle_{\partial \Omega} 
		-(\mat{\tau}, \nabla \vec{w})
		+ (\mat{\tau}, \mat{\eta}) = 
		\langle \mat{\tau} \vec{n}, \vec{g} \rangle_{\partial \Omega}
				+ (\mat{F}, \mat{\tau}) 
		\qquad \forall \mat{\tau} \in \Sigma,
\end{align*} 
Thus, we must have that $\vec{g} = \vec{w}$ and 
$-\nabla \vec{w} + \mat{\eta} = \mat{F}$.
Since $\mat{F}$ is symmetric, we have $\symgrad(\vec{w}) = -\mat{F}$
and $\mat{\eta}  = \anti(\nabla \vec{w})$.  Thus, $\vec{w}$ is precisely
as above with $\mat{\eta}  = \anti(\nabla \vec{w})$, and so
the RHS of \cref{eq:linear_laws_weak_weak_sym_1} for all the examples
in \cref{sec:examples} is of the form 
$b(\mat{\tau}, \vec{w}) + c(\mat{\tau}, \mat{\eta})$ for suitably
chosen $\vec{w}$ and $\mat{\eta}$.

\subsection{Material robustness in the examples}
\label{sec:structure-preservation-examples}

Since all the numerical examples are of the form 
\cref{eq:saddle-perturbed-discrete}, we compare the numerical results from 
\cref{sec:examples} to \cref{lem:saddle-invariance-discrete}. For the 
strongly symmetric discrete formulation 
\cref{eq:linear_laws_weak_strong_sym_fem}, 
the JMK and $\mathrm{HZ}_3$ schemes satisfy
$\nabla \cdot \Sigma^{\sym}_h = V_h$, and so 
\begin{align*}
	\ker \mathcal{B}_h = \{ \mat{\tau}_h \in \Sigma^{\sym}_h : 
		(\nabla \cdot \mat{\tau}_h, \vec{v}_h)_{L^2(\Omega)} = 0 
			\ \forall \vec{v}_h \in V_h \}
	\subset \ker \mathcal{B} = \{ \mat{\tau} \in \Sigma^{\sym} :
	 	\nabla \cdot \mat{\tau} \equiv \vec{0} \}.	
\end{align*}
\Cref{lem:saddle-invariance-discrete} then shows that $\mat{\sigma}_h \equiv 0$
in exact arithmetic for all the examples in 
\cref{sec:lin_el,sec:trans_iso,sec:polar} with the 
JMK or $\mathrm{HZ}_3$ schemes. Thus, any scheme of the form 
\cref{eq:linear_laws_weak_strong_sym_fem} with 
$\nabla \cdot \Sigma^{\sym}_h = V_h$ is material robust
and the numerical results are consistent with 
the theory (up to roundoff errors and solver tolerances).

For the weakly symmetric discrete formulation 
\cref{eq:linear_laws_weak_strong_sym_fem}, the situation is more complicated.
Now, the discrete kernel $\ker \mathcal{B}_h$ has the form
\begin{align*}
	\ker \mathcal{B}_h = \{ \mat{\tau}_h \in \Sigma^{\sym}_h : 
		(\nabla \cdot \mat{\tau}_h, \vec{v}_h)_{L^2(\Omega)} 
			+ (\mat{\tau}_h, \mat{\xi}_h )_{L^2(\Omega)} = 0 
		\ \forall \vec{v}_h \in V_h, \ \forall \mat{\xi}_h \in \Xi_h \}.
\end{align*}
For all the schemes used in \cref{sec:examples}, we still have 
$\nabla \cdot \Sigma_h^{\sym} = V_h$ which gives
\begin{align*}
	\ker \mathcal{B}_h = \{  \mat{\tau}_h \in \Sigma^{\sym}_h : 
		\nabla \cdot \mat{\tau}_h \equiv \vec{0} \text{ and } 
		(\mat{\tau}_h, \mat{\xi}_h )_{L^2(\Omega)} = 0 
		\ \forall \mat{\xi}_h \in \Xi_h \}.
\end{align*}
However, none of the above schemes satisfy $\anti(\Sigma_h) \subseteq \Xi_h$, 
and so $\ker \mathcal{B}_h \nsubseteq \ker \mathcal{B}$. That is,
there are discrete tensors $\mat{\tau}_h \in \ker \mathcal{B}_h$ that are not
pointwise symmetric. Consequently, \cref{lem:saddle-invariance} does not apply
and we cannot expect $\mat{\sigma}_h$ to be zero. In the transversely isotropic 
example in \cref{sec:trans_iso}, PEERS and $\mathrm{AFW}_k$, $k \leq 2$,
both produced nonzero $\mat{\sigma}_h$. For the polar fluid example in 
\cref{sec:polar}, all the weakly symmetric schemes produced nonzero 
$\mat{\sigma}_h$, and so the weakly symmetric 
schemes considered in \cref{sec:general-setup-weak-sym} are not 
material robust. 

In summary, we have the following result:
\begin{theorem}
	\label{thm:material-robustness}
	If a strongly symmetric discretization 
	\cref{eq:linear_laws_weak_strong_sym_fem} satisfies 
	\begin{align}
		\label{eq:strong_sym_kernel_inclusion}
		\{ \mat{\tau}_h \in \Sigma^{\sym}_h : 
			(\nabla \cdot \mat{\tau}_h, \vec{v}_h)_{L^2(\Omega)} = 0 
			\ \forall \vec{v}_h \in V_h \}
		\subset \{ \mat{\tau} \in \Sigma^{\sym} :
	 		\nabla \cdot \mat{\tau} \equiv \vec{0} \},
	\end{align}
	then it is material robust.
	Similarly, if a weakly symmetric discretization 
	\cref{eq:linear_laws_weak_weak_sym_fem} satisfies
	\begin{multline}
		\label{eq:weak_sym_kernel_inclusion}
		\{ \mat{\tau}_h \in \Sigma_h : 
		(\nabla \cdot \mat{\tau}_h, \vec{v}_h)_{L^2(\Omega)} 
			+ (\mat{\tau}_h, \mat{\xi}_h )_{L^2(\Omega)} = 0 
			\ \forall \vec{v}_h \in V_h, \ \forall \mat{\xi}_h \in \Xi_h \} \\
		\subset \{ \mat{\tau} \in \Sigma^{\sym} :
	 		\nabla \cdot \mat{\tau} \equiv \vec{0} \},
	\end{multline}
	then it is material robust.
\end{theorem}
\noindent Condition \cref{eq:strong_sym_kernel_inclusion} may be interpreted as 
ensuring that in the absence of body forces $\vec{f} = \vec{0}$, the scheme 
\cref{eq:linear_laws_weak_strong_sym_fem} produces a stress with zero 
linear momentum, i.e. $\div \mat{\sigma} = \vec{0}$. Condition 
\cref{eq:weak_sym_kernel_inclusion} additionally ensures that the angular 
momentum is conserved, i.e. $\mat{\sigma}^{\top} = \mat{\sigma}$.

We have yet to explain the following behavior.
For the linear isotropic solid 
in \cref{sec:lin_el}, all the weakly symmetric schemes gave 
$\mat{\sigma}_h \equiv \mat{0}$, while for the transversely isotropic solid,
$\mathrm{AFW}_3$ gave $\mat{\sigma}_h \equiv \mat{0}$. Recall that the 
key observation was that the rotation tensor $\mat{\omega}$
satisfied $\mat{\omega} \in \Xi_h$ in these cases. To rigorously explain the 
behavior in this case, we turn a priori error estimates.

\subsection{A priori error estimates}

As noted in the previous section, \cref{lem:saddle-invariance-discrete}
is not sufficient to explain the behavior observed in \cref{sec:examples}
in all cases, particularly the implication that $\mat{\sigma} \equiv \mat{0}$
and $\mat{\omega} \in \Xi_h$ guarantees $\mat{\sigma}_h \equiv \mat{0}$. 
To explain this case, we again return to the abstract setting
and establish a priori error estimates,
which requires some additional notation. We denote the norms of the 
bilinear forms $A(\cdot,\cdot)$ and $B(\cdot,\cdot)$ by
\begin{align}
	\label{eq:bilinear-forms-bounded}
	\|A\| := \sup_{\substack{v, w  \in \mathcal{V} \\ 
			\|v\|_{\mathcal{V}} = \|w\|_{\mathcal{V}} = 1}} |A(u, v)|
	\quad \text{and} \quad 
	\|B\| := \sup_{\substack{v \in \mathcal{V}, q \in \mathcal{Q} \\ 
			\|v\|_{\mathcal{V}} = \|q\|_{\mathcal{Q}}= 1 } } |B(v, q)|.
\end{align}
Let $\mathcal{B} : \mathcal{V} \to \mathcal{Q}'$ and 
$\mathcal{B}^t : \mathcal{Q} \to \mathcal{V}'$
denote the linear operators associated with the bilinear form 
$B(\cdot,\cdot)$; i.e.,
\begin{align*}
	\mathcal{B}(v)(q) := B(v, q) 
	\quad \text{and} \quad  
	\mathcal{B}^t(q)(v) := B(v, q)
	\qquad \forall v \in \mathcal{V}, 
		\ \forall q \in \mathcal{Q}.
\end{align*}
Let $\mathcal{B}_h : \mathcal{V}_h \to \mathcal{Q}_h'$ 
and $\mathcal{B}_h^t : \mathcal{Q}_h \to \mathcal{V}_h'$
be defined analogously. Note that the definition of $\ker \mathcal{B}$ 
\cref{eq:b-bilinear-kernel} and $\ker \mathcal{B}_h$ 
\cref{eq:b-bilinear-kernel-discrete} coincide with the usual notion of the 
kernel of a linear operator. We similarly use $\ker \mathcal{B}^t$ and 
$\ker \mathcal{B}_h^t$ to denote the corresponding kernels. 
Given $F \in \mathcal{V}'$ and $G \in \mathcal{Q}'$, we define
the following affine spaces:
\begin{align}
	\label{eq:zstarf-def}
	\mathcal{Z}_h^*(F) &:= \{ q_h \in \mathcal{Q}_h : 
		B(v_h, q_h) = F(v_h) \ \forall v_h \in \mathcal{V}_h \}, \\
	\label{eq:zg-def}
	\mathcal{Z}_h(G) &:= \{ v_h \in \mathcal{V}_h : 
		B(v_h, q_h) = G(q_h) \ \forall q_h \in \mathcal{Q}_h \}.
\end{align} 
Finally, we assume that $A(\cdot,\cdot)$ is symmetric and satisfies
\begin{subequations}
	\label{eq:a-nonneg-spd-ker}
	\begin{alignat}{2}
		A(v, v) &\geq 0 \qquad & &\forall v \in \mathcal{V}, \\
		A(z, z) & \geq \alpha \|z\|_{\mathcal{V}}^2 \qquad 
			& &\forall z \in \ker \mathcal{B}, \\
		A(z_h, z_h) & \geq \alpha_h \|z_h\|_{\mathcal{V}}^2 \qquad 
			& &\forall z_h \in \ker \mathcal{B}_h, 
	\end{alignat}
\end{subequations}
for some $\alpha, \alpha_h > 0$. Note that all the examples above 
satisfy \cref{eq:a-nonneg-spd-ker}. Strictly speaking, 
conditions \cref{eq:a-nonneg-spd-ker} can be weakened 
(see e.g. \cite[Theorem 4.2.2]{BoffiBrezziFortin13}) at the expense of harsher 
dependence on the constants in all the forthcoming estimates.

Without any additional conditions on $\mathcal{V}_h$ and 
$\mathcal{Q}_h$, we recall the standard error estimate from 
\cite[Theorem 5.2.2]{BoffiBrezziFortin13}:
\begin{subequations}
\begin{multline}
	\|u - u_h\|_{\mathcal{V}} \leq 
		2\left( \frac{\|A\|}{ \alpha_h } 
				+ \frac{ \|B\|}{\beta_h} \sqrt{\frac{\|A\|}{\alpha_h}} \right) 
			\inf_{v_h \in \mathcal{V}_h} \|u - v_h\|_{\mathcal{V}} 
		+ \frac{\|b\|}{\alpha_h} 
			\inf_{q_h \in \mathcal{Q}_h} \|p - q_h\|_{\mathcal{Q}}
\end{multline}
and
\begin{multline}
\|p - p_h\|_{\mathcal{Q}} \leq 
		\left( \frac{2}{ \beta_h } \sqrt{\frac{\|A\|^3}{\alpha_h}} 
				+ \frac{\|A\| \|B\|}{\beta_h^2} \right) 
			\inf_{v_h \in {\mathcal{V}}_h} \|u - v_h\|_{\mathcal{V}} \\
		+ \frac{3 \|B\|}{\beta_h } \sqrt{\frac{\|A\|}{\alpha_h}}
		\inf_{q_h \in \mathcal{Q}_h} \|p - q_h\|_{\mathcal{Q}}.
\end{multline}
\end{subequations}
We note that these estimates are insufficient to explain the 
behavior of the numerical examples. In particular, these
estimates only guarantee that $\mat{\sigma}_h \equiv \mat{0}$ for the
strongly symmetric schemes if $\mat{\sigma} \in \Sigma_h$ and $\vec{u} \in V_h$, 
while the weakly symmetric schemes additionally require that 
$\mat{\omega} \in \Xi_h$.
However, our discrete spaces have additional structure which can be
used to obtain better error estimates. We first
consider the structure-preserving case 
$\ker \mathcal{B}_h \subseteq \ker \mathcal{B}$, which are satisfied
by the strongly symmetric schemes considered here. 
Then, we will consider the case
that $\mathcal{Q}_h$ is a product of two spaces and the discretizations are 
only structure-preserving on one of the spaces, as is the case for
the weakly symmetric schemes considered here.

\subsubsection{Structuring-preserving discretizations}

We first assume that the discretization is structure preserving so that
$\ker \mathcal{B}_h \subseteq \ker \mathcal{B}$. The first error estimate
follows from standard arguments 
(see e.g. \cite[Chapter 5.2.2]{BoffiBrezziFortin13}).
\begin{lemma}
	\label{lem:best-approx-kernel-inclusion}
	Suppose that $\ker \mathcal{B}_h \subseteq \ker \mathcal{B}$. 
	Let $u \in \mathcal{V}$ and
	$p \in \mathcal{Q}$ satisfy \cref{eq:saddle} and
	$u_h \in \mathcal{V}_h$ and $p_h \in \mathcal{Q}_h$ satisfy 
	\cref{eq:saddle-discrete}. 
	Then, there holds
	\begin{subequations}
		\label{eq:best-approx-kernel-inclusion}
		\begin{align}
			\label{eq:best-approx-kernel-inclusion-u}
			\|u - u_h\|_{\mathcal{V}} &\leq 
				2\sqrt{\frac{\|A\|}{\alpha_h}}
				\inf_{w_h \in \mathcal{Z}_h( \mathcal{B} u)} 
				\|u - w_h\|_{\mathcal{V}}, \\
			\label{eq:best-approx-kernel-inclusion-p}
			\|p - p_h\|_{\mathcal{Q}} &\leq 
			\frac{2}{\beta_h }\sqrt{\frac{ \|A\|^3}{\alpha_h}} 
				\inf_{w_h \in \mathcal{Z}_h(\mathcal{B} u)} 
					\|u - w_h\|_{\mathcal{V}}
			+ \inf_{r_h \in \mathcal{Z}_h^*(\mathcal{B}^t p)} 
				\|p - r_h\|_{\mathcal{Q}}.
		\end{align}
	\end{subequations}
\end{lemma}
\begin{proof}
	Let $w_h \in \mathcal{Z}_h(G)$. Then, $u_h - w_h \in \ker \mathcal{B}_h$, 
	and there holds
	\begin{align*}
		A(u_h - w_h, v) = F(v) - A(w_h, v) = A(u - w_h, v) 
			\qquad \forall v \in \ker \mathcal{B}_h.
	\end{align*}
	Consequently, \cite[Lemma 4.3.1]{BoffiBrezziFortin13} gives
	\begin{align*}
		A(u_h - w_h, u_h - w_h) &= A(u - w_h, u_h - w_h) \\
		&\leq \sqrt{A(u - w_h, u - w_h)} \sqrt{A(u_h - w_h, u_h - w_h)},
	\end{align*}
	and so
	\begin{align*}
		\alpha_h \|u_h - w_h\|_{\mathcal{V}}^2 \leq A(u_h - w_h, u_h - w_h) 
			\leq A(u - w_h, u - w_h) 
			\leq \|A\| \|u - w_h\|_{\mathcal{V}}^2.
	\end{align*}
	Noting that $\mathcal{Z}_h(G) = \mathcal{Z}_h(\mathcal{B} u)$
	and $\|A\| \geq \alpha_h$, we apply the 
	triangle inequality to obtain \cref{eq:best-approx-kernel-inclusion-u}.	
	Now, for any $r_h \in \mathcal{Z}_h^*(\mathcal{B}^t p)$, we have
	\begin{align*}
		B(v_h, p_h - r_h) = A(u - u_h, v_h) \qquad 
			\forall v_h \in \mathcal{V}_h.
	\end{align*}
	The inf-sup condition and \cref{eq:best-approx-kernel-inclusion-u} then 
	gives 
	\begin{align*}
		\|p_h - r_h\|_Q \leq \frac{2}{\beta_h }\sqrt{\frac{ \|A\|^3}{\alpha_h}} 
			 \inf_{w_h \in \mathcal{Z}_h(\mathcal{B} u)} 
			 	\|u - w_h\|_{\mathcal{V}}.
	\end{align*}
	Taking the infimum over all such $r_h$ and applying the triangle inequality
	completes the proof of \cref{eq:best-approx-kernel-inclusion-p}.
\end{proof}

Note that the infimums appearing on the RHS of 
\cref{eq:best-approx-kernel-inclusion-u} and
\cref{eq:best-approx-kernel-inclusion-p} are restricted to the affine 
spaces $\mathcal{Z}_h(\mathcal{B} u)$ and $\mathcal{Z}_h^*(\mathcal{B}^t p)$.
In some instances, these infimums can be estimated directly to obtain
convergence rates 
(see e.g. \cite[Sections 5 \& 6]{BabuskaSuri92} for the infimum in
\cref{eq:best-approx-kernel-inclusion-u} corresponding to the approximation
properties of divergence-free, continuous piecewise polynomial vector fields).
In other instances, the approximation properties of these affine spaces is less 
clear. 

A classical technique in the literature to obtain
an upper bound for the terms in \cref{eq:best-approx-kernel-inclusion-u} and
\cref{eq:best-approx-kernel-inclusion-p} is to use a Fortin operator 
(see e.g. \cite[Chapter 5.4.3]{BoffiBrezziFortin13}). 
Let $\pi_F : \mathcal{V} \to \mathcal{V}_h$
be a linear operator satisfying
\begin{subequations}
	\label{eq:v-fortin}
	\begin{alignat}{2}
		\label{eq:v-fortin-b-projection}
		b(v - \pi_F v, q_h) &= 0 \qquad 
			& & \forall q_h \in \mathcal{Q}_h, \ \forall v \in \mathcal{V}, \\
		\label{eq:v-fortin-stability}
		\|\pi_F v\|_{\mathcal{V}} &\leq C_F \|v\|_{\mathcal{V}} \qquad 
			& &\forall v \in \mathcal{V},
	\end{alignat}
\end{subequations}
for some $C_F > 0$. Since we are assuming that the discrete inf-sup condition
\cref{eq:inf-sup-b-bilinear-discrete} holds, such an operator always exists
with $C_F \leq \|B\| / \beta_h$ (see e.g. 
\cite[Remark 5.1.10]{BoffiBrezziFortin13}). However,
for some problems, particularly on anisotropic domains, $C_F$ may be smaller
(see \cite[Remark 4.1]{JohnLinkeMerdonNeilanRebholz17}
and references therein). Moreover, let $C_{\Phi}$ be the stability constant
of the operator $\Phi_h$ in \cref{lem:saddle-perturbed-discrete-pressure}:
\begin{align}
	\label{eq:q-fortin-stability-constant}
	C_{\Phi} := \sup_{\substack{q \in \mathcal{Q} \\ \|q\|_{\mathcal{Q}} = 1}}
		\|\Phi_h q\|_{\mathcal{Q}}.
\end{align}
Note that the inf-sup constant gives $C_{\Phi} \leq \|B\| / \beta_h$, but
$C_{\Phi}$ may be smaller; e.g. $C_{\Phi} = 1$ in the example in
\cref{rem:b-special-form}.
With these operators in hand, we have the following result.
\begin{lemma}
	\label{lem:best-approx-kernels}
	For all $u \in \mathcal{V}$, there holds
	\begin{align}
		\label{eq:best-approx-kernel-v}
		\inf_{w_h \in \mathcal{Z}_h(\mathcal{B }u)} \|u - w_h\|_{\mathcal{V}} 
			\leq (1 + C_F) \inf_{v_h \in \mathcal{V}_h} 
				\|u - v_h\|_{\mathcal{V}}.
	\end{align}
	Moreover, if  $\ker \mathcal{B}_h \subseteq \ker \mathcal{B}$, then for 
	all $p \in \mathcal{Q}$, there holds
	\begin{align}
		\label{eq:best-approx-kernel-q}
		\inf_{r_h \in \mathcal{Z}_h^*(\mathcal{B}^t p)} 
			\|p - r_h\|_{\mathcal{Q}} 
		\leq  (1 + C_{\Phi}) \inf_{q_h \in \mathcal{Q}_h} 
			\|p - q_h\|_{\mathcal{Q}}.
	\end{align}
\end{lemma}
\begin{proof}
	Inequality \cref{eq:best-approx-kernel-v} can be found in 
	e.g. \cite[Lemma 50.3]{ErnGuermondII21}. The proof of 
	\cref{eq:best-approx-kernel-q} is identical. We include it here for 
	completeness. Let $q_h \in \mathcal{Q}_h$ and define 
	$s_h = \Phi_h (p - q_h)$.
	Then, $r_h := s_h + q_h \in \mathcal{Z}_h^*(\mathcal{B}^t p)$ since
	$B(v_h, r_h) = B(v_h, p) = (\mathcal{B}^t p)(v_h)$ for all 
	$v_h \in \mathcal{V}_h$. Moreover, we have
	\begin{align*}
		\|p - r_h\|_{\mathcal{Q}} \leq \|p - q_h\|_{\mathcal{Q}} 
				+ \|r_h\|_{\mathcal{Q}}
			\leq (1 + C_{\Phi}) \|p - q_h\|_{\mathcal{Q}}.
	\end{align*}
	The result now follows by first bounding the LHS below
	by the infimum over all elements in $\mathcal{Z}_h^*(\mathcal{B}^t p)$
	and then taking the infimum over all $q_h \in \mathcal{Q}_h$.
\end{proof}

Combining \cref{lem:best-approx-kernel-inclusion} and
\cref{lem:best-approx-kernels} gives the following result.
\begin{corollary}
	\label{cor:best-approx-kernel-inclusion-full-inf}
	Suppose that $\ker \mathcal{B}_h \subseteq \ker \mathcal{B}$. 
	Let $u \in \mathcal{V}$ and
	$p \in \mathcal{Q}$ satisfy \cref{eq:saddle} and
	$u_h \in \mathcal{V}_h$ and $p_h \in \mathcal{Q}_h$ satisfy 
	\cref{eq:saddle-discrete}. 
	Then, there holds
	\begin{subequations}
		\label{eq:best-approx-kernel-inclusion-full-inf}
		\begin{align}
			\label{eq:best-approx-kernel-inclusion-full-inf-u}
			\|u - u_h\|_{\mathcal{V}} &\leq 
				2(1 + C_F)\sqrt{\frac{\|A\|}{\alpha_h}}
				\inf_{v_h \in \mathcal{V}_h} 
					\|u - v_h\|_{\mathcal{V}}, \\
			\label{eq:best-approx-kernel-inclusion-full-inf-p}
			\|p - p_h\|_{\mathcal{Q}} &\leq 
			\frac{2(1 + C_F)}{\beta_h }\sqrt{\frac{ \|A\|^3}{\alpha_h}} 
				\inf_{v_h \in \mathcal{V}_h} 
					\|u - v_h\|_{\mathcal{V}}
			+ (1 + C_{\Phi}) \inf_{r_h \in \mathcal{Q}_h} 
				\|p - r_h\|_{\mathcal{Q}},
		\end{align}
	\end{subequations}
	where $C_F$ is defined in \cref{eq:v-fortin-stability}
	and $C_{\Phi}$ in \cref{eq:q-fortin-stability-constant}
\end{corollary}
\noindent A similar result to \cref{cor:best-approx-kernel-inclusion-full-inf} 
appears in \cite[Theorem 5.2.4]{BoffiBrezziFortin13} 
and \cite[Theorem 2]{FalkOsborn80} but 
with harsher dependence on the constants since $A(\cdot,\cdot)$ is not assumed 
to be symmetric and nonnegative.

\subsubsection{Partially structure-preserving discretizations}

We now suppose that $\mathcal{Q}$ is of the form 
$\mathcal{Q} = \mathcal{Q}_1 \times \mathcal{Q}_2$, where $\mathcal{Q}_i$
is a Hilbert spaces with inner product 
$(\cdot,\cdot)_{\mathcal{Q}_i}$ and induced norm $\|\cdot\|_{\mathcal{Q}_i}$,
$i \in \{1,2\}$. We similarly suppose that the discrete space
$\mathcal{Q}_h \subset \mathcal{Q}$ is of the form
$\mathcal{Q}_h = \mathcal{Q}_{1, h} \times \mathcal{Q}_{2, h}$,
where $\mathcal{Q}_{i, h} \subset \mathcal{Q}_{i}$, $i \in \{1,2\}$.

We consider a bilinear form $B(\cdot,\cdot)$ and linear 
functional $G$ of the form
\begin{alignat*}{2}
	B(v, p) &= B_1(v, p_1) + B_2(v, p_2) \qquad 
		& &\forall p = (p_1, p_2) \in \mathcal{Q}_1 \times \mathcal{Q}_2, \
			\forall v \in \mathcal{V}, \\
	G(p) &= G_1(p_1) + G_2(p_2) \qquad 
		& &\forall p = (p_1, p_2) \in \mathcal{Q}_1 \times \mathcal{Q}_2,
\end{alignat*}
where $B_i(\cdot,\cdot) : \mathcal{V} \times \mathcal{Q}_i \to \mathbb{R}$ are 
continuous bilinear forms with norms $\|B_i\|$
and $G_i(\cdot)$ are continuous linear functionals on $\mathcal{Q}_i$, 
$i \in \{1,2\}$. We define the operators $\mathcal{B}_i$, $\mathcal{B}_i^t$, 
$\mathcal{B}_{i, h}$, and $\mathcal{B}_{i, h}^t$ and
affine spaces $\mathcal{Z}_{i, h}(G_i)$ and $\mathcal{Z}_{i, h}^*(F)$ 
analogously to above.

If we assume that the discretization is only structure-preserving with respect
to $\mathcal{B}_1$ in the sense that 
$\ker \mathcal{B}_{1, h} \subseteq \ker \mathcal{B}_1$ but make no additional
assumption on $\ker \mathcal{B}_{2, h}$, then we have the following error 
estimates.

\begin{lemma}
	Suppose that $\ker \mathcal{B}_{1, h} \subseteq \ker \mathcal{B}_1$. 
	Let $u \in \mathcal{V}$ and
	$p = (p_1, p_2) \in \mathcal{Q}_1 \times \mathcal{Q}_2$ satisfy 
	\cref{eq:saddle} and
	$u_h \in \mathcal{V}_h$ and 
	$p_h = (p_{1, h}, p_{2, h}) \in \mathcal{Q}_{1,h} \times \mathcal{Q}_{2, h}$
	satisfy \cref{eq:saddle-discrete}. Then, there holds
	\begin{subequations}
		\label{eq:best-approx-kernel-partial-inclusion}
		\begin{align} 
		\label{eq:best-approx-kernel-partial-inclusion-u}
		\begin{aligned}
		\|u - u_h\|_{\mathcal{V}} &\leq 2\left( \frac{\|A\|}{\alpha_h} + 
			\frac{\|B_2\|}{ \beta_h} \sqrt{\frac{\|A\|}{\alpha_h}} \right)  
			\inf_{w_h \in \mathcal{Z}_{1, h}(G_1)}\|u - w_h\|_{\mathcal{V}} \\
		&\qquad + \frac{\|B_2\|}{\alpha_h} 
			\inf_{r_{2, h} \in \mathcal{Q}_{2, h}} 
			\|p_2 - r_{2, h}\|_{\mathcal{Q}_2},
		\end{aligned}
	\end{align}
	\begin{align}
		\label{eq:best-approx-kernel-partial-inclusion-p1}
		\begin{aligned}
		\|p_1 - p_{1, h}\|_{\mathcal{Q}_1} &\leq 
		\left( \frac{2\|A\|^2}{\alpha_h \beta_h} 
			+ \frac{4\|B_2\|}{\beta_h^2} \sqrt{\frac{\|A\|^3}{\alpha_h}}
			+ \frac{\|A\| \|B_2\|^2}{\beta_h^3} \right)
		 \inf_{w_h \in \mathcal{Z}_{1, h}(G_1)}\|u - w_h\|_{\mathcal{V}} \\
		 & \qquad
		 + \inf_{r_{1, h} \in \mathcal{Z}_{1, h}^*(\mathcal{B}_1^t p_1)}
		 	\|p_1 - r_{1, h}\|_{\mathcal{Q}_1} \\
		& \qquad
		+ \frac{\|B_2\|}{\beta_h}\left( \frac{\|A\|}{\alpha_h} 
			+ \frac{3 \|B_2\|}{\beta_h} \sqrt{\frac{\|A\|^3}{\alpha_h}} 
		\right)
		\inf_{r_{2, h} \in \mathcal{Q}_{2, h}}
			\|p_2 - r_{2,  h}\|_{\mathcal{Q}_2},
		\end{aligned}
	\end{align} 
	and
	\begin{align}
		\label{eq:best-approx-kernel-partial-inclusion-p2}
		\begin{aligned}
		\|p_2 - p_{2, h}\|_{\mathcal{Q}_2} &\leq \left( 
			\frac{2}{\beta_h} \sqrt{\frac{\|A\|^3}{\alpha_h}}
		+ \frac{\|A\| \|B_2\|}{\beta_h^2}  \right) 
			\inf_{w_h \in \mathcal{Z}_{1, h}(G_1)} \|u - w_h\|_{\mathcal{V}} \\
		&\qquad
		+ \frac{3 \|B_2\|}{\beta_h} \sqrt{\frac{\|A\|}{\alpha_h}}  
		  \inf_{r_{2, h} \in \mathcal{Q}_{2, h}}
		  	\|p_2 - r_{2,  h}\|_{\mathcal{Q}_2}.
		\end{aligned}
	\end{align} 
	\end{subequations}
\end{lemma}
\begin{proof}
	Let $w_h \in \mathcal{Z}_{1, h}(G_1)$ and $r_{2, h} \in \mathcal{Q}_{2, h}$.
	Then, there holds
	\begin{align*}
		A(u_h - w_h, v_h) + B_2(v_h, p_{2, h} - r_{2, h}) &= A(u - w_h, v_h) 
		+ B_2(v_h, p_2 - r_{2, h}), \\
		B_{2}(u_h - w_h, q_{2, h}) &= B_2(u - w_h, q_{2, h}),
	\end{align*}
	for all $v_h \in \mathcal{Z}_{1, h}(0)$ and 
	$q_{2, h} \in \mathcal{Q}_{2, h}$. First note that the following 
	inf-sup condition restricted to 
	$\mathcal{Z}_{1, h}(0) \times \mathcal{Q}_{2, h}$ holds for $B_2$:
	\begin{align}
		\label{eq:inf-sup-h-b-bilinear-kernel}
		\inf_{q_{2, h} \in \mathcal{Q}_{2, h} } 
		\sup_{ v_h \in \mathcal{Z}_{1, h}(0) } 
		\frac{ B_2(v_h, q_{2, h}) }{ \|v_h\|_{\mathcal{V}} 
			\|q_{2, h}\|_{\mathcal{Q}_2} } 
		\geq \beta_h.  
	\end{align}
	Indeed, a consequence full inf-sup condition 
	\cref{eq:inf-sup-b-bilinear-discrete} is that for every 
	$q_{2, h} \in \mathcal{Q}_{2, h}$, there exists $v_h \in \mathcal{V}_h$
	such that 
	\begin{align*}
		B_1(v_h, r_{1, h}) + B_2(v_h, r_{2, h}) 
			= (q_{2, h}, r_{2, h})_{\mathcal{Q}_2} 
		\qquad \forall r_{i,h} \in \mathcal{Q}_{i, h}, i \in \{1,2\},
	\end{align*}
	and $\|v_h\|_{\mathcal{V}} \leq \beta_h^{-1} \|q_{2, h}\|_{\mathcal{Q}_2}$
	(see e.g. \cite[Remark 5.1.10]{BoffiBrezziFortin13}). Clearly, 
	$v_h \in \mathcal{Z}_{1, h}(0)$, and so 
	\cref{eq:inf-sup-h-b-bilinear-kernel} follows.

	We now apply \cite[Theorem  5.2.1]{BoffiBrezziFortin13} to obtain
	\begin{align*}
		\|u_h - w_h\|_{\mathcal{V}} &\leq \left( \frac{\|A\|}{\alpha_h} + 
		\frac{2 \|B_2\|}{ \beta_h} \sqrt{\frac{\|A\|}{\alpha_h}} \right)  
			\|u - w_h\|_{\mathcal{V}} +  
		\frac{\|B_2\|}{\alpha_h}  \|p_2 - r_{2, h}\|_{\mathcal{Q}_2} \\
		\|p_{2, h} - r_{2, h}\|_{\mathcal{Q}_2} &\leq \left( 
			\frac{2}{\beta_h} \sqrt{\frac{\|A\|^3}{\alpha_h}}
		+ \frac{\|A\| \|B_2\|}{\beta_h^2}  \right) \|u - w_h\|_{\mathcal{V}} 
		  + \frac{2 \|B_2\|}{\beta_h} \sqrt{\frac{\|A\|}{\alpha_h}}  
		  \|p_2 - r_{2,  h}\|_{\mathcal{Q}_2}. 
	\end{align*}
	Using the triangle inequality and taking the infimum over all $w_h$
	and $r_{2, h}$ then gives 
	\cref{eq:best-approx-kernel-partial-inclusion-u,%
	eq:best-approx-kernel-partial-inclusion-p2}.

	We now turn to $p_1 - p_{1, h}$. For any 
	$r_{1, h} \in \mathcal{Z}_{1, h}^*(\mathcal{B}_1^t p_1)$, there
	holds
	\begin{align*}
		B_1(v_h, p_{1, h} - r_{1, h}) = A(u - u_h, v_h) 
			+ B_2(v_h, p_2 - p_{2, h})
		\qquad \forall v_h \in \mathcal{V}_h.
	\end{align*}
	Applying the inf-sup condition \cref{eq:inf-sup-b-bilinear-discrete},
	we obtain
	\begin{align*}
		\|p_{1, h} - r_{1, h}\|_{\mathcal{Q}_1} 
			\leq \frac{\|A\|}{\beta_h} \|u - u_h\|_{\mathcal{V}}
				+ \frac{\|B_2\|}{\beta_h} \|p_2 - p_{2, h}\|_{\mathcal{Q}_2}.
	\end{align*}
	Inequality \cref{eq:best-approx-kernel-partial-inclusion-p1} now follows.
\end{proof}

\begin{remark}
	\label{rem:best-approx-kernel-inclusion-full-inf}
	As in \cref{cor:best-approx-kernel-inclusion-full-inf}, we could further 
	replace the infimums in \cref{eq:best-approx-kernel-partial-inclusion} 
	by infimums over the full spaces and pick up a factor of 
	$(1 + C_F)$ and $(1 + C_{\Phi_1})$, where 
	$\Phi_1 : \mathcal{Q}_1 \to \mathcal{Q}_{1, h}$ is the operator
	in \cref{lem:saddle-perturbed-discrete-pressure} for $B_1$. 	
\end{remark}

\subsection{Applying error estimates to examples}

We now apply the error estimates above to the examples.
Note that $\lambda \in [0, 1] \cup \{\infty\}$, and so we have
$\|A\| \leq M/\mu$ and $\alpha_h \geq \alpha/\mu$, for constants
$M$ and $\alpha$ independent of $h$, $\mu$, and $\lambda$ 
(see \cref{lem:a-elliptic-elasticity} below). Moreover,
the ratio $\|B\|/\beta_h$ is bounded uniformly in $h$ as we have
chosen stable elements. Consequently, we will only highlight the
dependence of the constants in the forthcoming error estimates on
$\mu$.

For the symmetric elements, again assuming that 
$\nabla \cdot \Sigma_h^{\sym} = V_h$,
we apply \cref{eq:best-approx-kernel-inclusion} to obtain
\begin{subequations}
	\label{eq:symmetric-a-priori}
	\begin{align}
		\label{eq:symmetric-a-priori-1}
		\| \mat{\sigma} - \mat{\sigma}_h \|_{\mathrm{div}} 
			&\leq C \inf_{\mat{\tau}_h \in \Sigma_h^{\sym}} 
				\| \mat{\sigma} - \mat{\tau}_h \|_{\mathrm{div}} \\
		\label{eq:symmetric-a-priori-2}
		\|\vec{u} - \vec{u}_h\| &\leq C \left(  
				\frac{1}{\mu} \inf_{\mat{\tau}_h \in \Sigma_h^{\sym}} 
					\| \mat{\sigma} - \mat{\tau}_h \|_{\mathrm{div}} 
		+  \inf_{\vec{v}_h \in V_h} \| \vec{u} - \vec{v}_h\| \right).
	\end{align}
\end{subequations}
Whenever $\mat{\sigma} \in \Sigma^{\sym}_h$, we 
obtain $\mat{\sigma}_h = \mat{\sigma}$ from \cref{eq:symmetric-a-priori-1},
further demonstrating the material robustness property as 
$\mat{0} \in \Sigma^{\sym}_h$.
More generally, \cref{eq:symmetric-a-priori-1} shows that the stress 
error is independent of the error in $\vec{u}$, which is consistent with the 
behavior in \cref{sec:ex_summary} where $\mat{\sigma} \neq \mat{0}$ but 
is independent of the scaling parameter $\delta$.

For the weakly symmetric elements, also with $\nabla \cdot \Sigma_h = V_h$, 
we apply \cref{eq:best-approx-kernel-partial-inclusion} and 
\cref{rem:best-approx-kernel-inclusion-full-inf} to obtain 
\begin{subequations}
	\label{eq:weak-symmetric-a-priori}
	\begin{align}
		\label{eq:weak-symmetric-a-priori-1}
		\| \mat{\sigma} - \mat{\sigma}_h \|_{\mathrm{div}} &\leq 
		C \left(  \inf_{\mat{\tau}_h \in \Sigma_h} 
			\| \mat{\sigma} - \mat{\tau}_h \|_{\mathrm{div}} 
			+ \mu \inf_{\mat{\zeta}_h \in \Xi_h } 
			\| \mat{\omega} - \mat{\zeta}_h \| \right) \\
		\label{eq:weak-symmetric-a-priori-2}
		\|\vec{u} - \vec{u}_h\| &\leq C \left( 
			\frac{1}{\mu} \inf_{\mat{\tau}_h \in \Sigma_h} 
			\| \mat{\sigma} - \mat{\tau}_h \|_{\mathrm{div}} 
			+ \inf_{\vec{v}_h \in V_h} \| \vec{u} - \vec{v}_h\| 
			+ \inf_{\mat{\zeta}_h \in \Xi_h } \| \mat{\omega} - \mat{\zeta}_h \| 
			\right)  \\
		\label{eq:weak-symmetric-a-priori-3}
		\|\mat{\omega} - \mat{\omega}_h\| &\leq C \left( 
			\frac{1}{\mu} \inf_{\mat{\tau}_h \in \Sigma_h} 
				\| \mat{\sigma} - \mat{\tau}_h \|_{\mathrm{div}} 
			+ \inf_{\mat{\zeta}_h \in \Xi_h} \| \mat{\omega} - \mat{\zeta}_h \|
			\right).
	\end{align}	
\end{subequations}
Note that inequality \cref{eq:weak-symmetric-a-priori-1} captures
the behavior we observed in \cref{sec:examples} --- if 
$\mat{\sigma} \equiv \mat{0}$ and $\mat{\omega} \in \Xi_h$, 
then $\mat{\sigma}_h \equiv \mat{0}$. If $\mat{\omega} \notin \Xi_h$, then
the error estimates do not guarantee that $\mat{\sigma}_h \equiv \mat{0}$,
and we indeed observed nonzero $\mat{\sigma}_h$ in this case, indicating 
a lack of material robustness.

The construction of the examples in \cref{sec:examples} may also 
be interpreted as follows.
We first find a displacement $\vec{u}^*$ in the kernel of the constitutive 
law (i.e.~\cref{eq:linear_constituitive_law_intro} holds with 
$\vec{u}^*$ and  $\mat{\sigma}^* = \mat{0}$) 
such that $\mat{\omega}^* := \anti (\nabla \vec{u}^*) \notin \Xi_h$. 
We then choose a solution containing a large component of $\mat{\omega}^*$ 
(i.e. $\mat{\omega} = \delta \mat{\omega}^* + \mat{\xi}$ with 
$\|\mat{\sigma}\| + \|\mat{\xi}\| \ll \delta \|\mat{\omega}^*\|$)
to demonstrate that the weakly symmetric schemes produce large errors. 
This construction is in line with the theory, as
the lack of the invariance property \cref{eq:structure-preserving-def}
and the error estimates \cref{eq:weak-symmetric-a-priori}
allow the weakly symmetric schemes to produce a much larger stress error 
(on the order of $\delta$ as $\delta \to \infty$) than the strongly symmetric 
schemes.

\subsection{General comments}
\label{sec:theory-general-comments}

The error estimates in \cref{eq:best-approx-kernel-partial-inclusion}
provide a general recipe for demonstrating ``non-robustness" of a scheme
that is not fully structure-preserving. For example, consider
a stable discretization of the saddle-point system \cref{eq:saddle} with
$\|A\| \leq C \upsilon$, $\alpha_h \geq \alpha \upsilon$ for some 
positive parameter $\upsilon$, and 
$\|B\|/\beta_h$ bounded uniformly in $h$ and $\upsilon$.
Then, \cref{eq:best-approx-kernel-partial-inclusion} reads
\begin{subequations}
	\label{eq:general-param-a-priori}
	\begin{align}
		\label{eq:general-param-a-priori-1}
		\| u - u_h \|_{\mathcal{V}} &\leq 
		C \left(  \inf_{w_h \in \mathcal{V}_h} 
			\| u - w_h \|_{\mathcal{V}} 
			+ \upsilon^{-1} \inf_{ r_{2, h} \in \mathcal{Q}_{2, h} } 
			\| p_{2} - r_{2, h} \|_{\mathcal{Q}_2} \right) \\
		\|p_1 - p_{1, h} \|_{\mathcal{Q}_1} &\leq C \left( 
			\upsilon \inf_{w_h \in \mathcal{V}_h} 
				\| u - w_h \|_{\mathcal{V}} 
			+ \inf_{r_{1, h} \in \mathcal{Q}_{1, h}} 
				\| p_1 - r_{1, h} \|_{\mathcal{Q}_1}
			+ \inf_{r_{2, h} \in \mathcal{Q}_{2, h} } 
				\| p_{2} - r_{2, h} \|_{\mathcal{Q}_2} 
			\right)  \\
		\|p_2 - p_{2,h} \|_{\mathcal{Q}_2} &\leq C \left( 
			\upsilon \inf_{w_h \in \mathcal{V}_h} 
				\| u - w_h \|_{\mathcal{V}} 
			+ \inf_{r_{2, h} \in \mathcal{Q}_{2, h} } 
				\| p_{2} - r_{2, h} \|_{\mathcal{Q}_2} 
			\right),
	\end{align}	
\end{subequations}
where $C$ is independent of $\upsilon$ and $h$. We will consider
$u_h$ the primary variable of interest. There are two ways to 
make the error estimate \cref{eq:general-param-a-priori-1} blow up
in a meaningful way:
\begin{enumerate}
	\item[(i)] Fix $\upsilon$ and find a sequence of solutions
	$u^k, p_1^k, p_2^k$ such that $\|u^k\|_{\mathcal{V}}$ is bounded,
	$p_2^k \notin \mathcal{Q}_{2, h}$,
	and $\|p_2^k\|_{\mathcal{Q}_2} \to \infty$ as $k \to \infty$. 

	\item [(ii)] For a sequence $\upsilon_k \to 0$, find
	solutions $u^k, p_1^k, p_2^k$ such that
	$\|u_k\|_{\mathcal{V}}$ is bounded and 
	$\inf_{ r_{2, h} \in \mathcal{Q}_{2, h} } 
			\| p_{2}^k - r_{2, h} \|_{\mathcal{Q}_2}$
	is bounded away from zero.
\end{enumerate}
In both cases, the best approximation term for $u_k$ in 
\cref{eq:general-param-a-priori-1} will remain bounded, 
while the remaining term in \cref{eq:general-param-a-priori-1} will blow up.
In contrast, for either scenario (i) or (ii), 
the error $\|u^k - u_h^k \|_{\mathcal{V}}$ for a structure preserving scheme 
($\ker \mathcal{B}_h \subset \ker \mathcal{B}$) will not blow up
thanks to \cref{eq:best-approx-kernel-inclusion-full-inf} 
since $\|u^k\|_{\mathcal{V}}$ is bounded. The examples in 
\cref{sec:examples} fit into (i). Whether scenario (i) or (ii)
can occur for realistic problem settings is application dependent.
We note that scenario (i) with $u^k \equiv 0$, 
$\mathcal{Q}_{1} = \mathcal{Q}_{1, h} = \{0\}$, and $\upsilon = 1$ is 
described in \cite[Remark 50.8]{ErnGuermondII21} in the 
general setting of structure-preserving versus standard discretizations.

In the context of incompressible flow, two commonly used terms
are ``pressure-robust'' \cite{JohnLinkeMerdonNeilanRebholz17} 
and ``Reynolds-(semi)-robust'' \cite{Schroder18} to describe various
properties of the discretizations.
Pressure robustness typically means that
the discrete velocity (and possibly other variables of interest) are
invariant if the external forcing is modified by the gradient of a suitable
scalar-valued function. Thus, pressure robustness is precisely the
invariance property in \cref{eq:structure-preserving-def}. Recall that 
\cref{lem:saddle-invariance-discrete} shows that 
the kernel inclusion $\ker \mathcal{B}_h \subset \ker \mathcal{B}$ is 
necessary for the discrete scheme to possess this invariance property
without modifying the discrete problem. Other methods which modify the 
functional $F(\cdot)$ in \cref{eq:saddle-perturbed-discrete-invariance} can 
also achieve the invariance property without requiring
$\ker \mathcal{B}_h \subset \ker \mathcal{B}$; see 
\cite[Section 5.2]{JohnLinkeMerdonNeilanRebholz17} for a discussion.
Reynolds robustness in the context of error estimates typically means
that the velocity error (and possibly the errors of other variables)
does not blow up (or not blow up too fast) as the viscosity tends to zero. 
Thus, schemes that are partially structure preserving 
(see e.g. \cite{Gopalakrishnan20} for a nonconforming scheme) 
are permitted provided that $\upsilon$ in \cref{eq:general-param-a-priori} is the 
inverse of the viscosity.

\section{Transient Example}
\label{sec:conclusions}

We conclude with some comparisons of the JMK and $\mathrm{AFW}_1$ schemes 
applied to a transient polar fluid.
Time-dependent problems are generally outside the framework presented in 
\cref{sec:theory}, as the structure of the variational problem has 
an additional term (see e.g. \cref{eq:time-dependent-strong-2} below).
Nevertheless, we will shortly observe that strongly enforcing symmetry
of the Cauchy stress tensor remains crucial for obtaining accurate 
solutions. We consider the following time-dependent extension 
of the polar fluid in \cref{sec:polar-2d} on $\Omega = (0, 1)^2$:
\begin{subequations}
	\label{eq:transient-polar}
	\begin{alignat}{2}
		\label{eq:transient-polar-1}
		\partial_t \vec{u}  - \div \mat{\sigma} &= 0 
			\qquad & &\text{in } \Omega \times (0, T), \\
		\label{eq:transient-polar-2}
		\frac{1}{2\mu} \mat{\sigma} - \symgrad(\vec{u}) 
			 &= \min\{1, t\} \left[  
				\frac{K_F}{2\mu} 
					\left( \nabla \vec{\nu}^{\top} \nabla\vec{\nu} \right)^D
				- \frac{1}{d} (\div \vec{u}_0) \mat{\mathbb{I}} \right] 
			\qquad & &\text{in } \Omega \times (0, T), \\
		\label{eq:transient-polar-3}
		\vec{u} &= \min\{ 1, t \} \vec{u}_0 \qquad & &
			\text{on } \partial \Omega \times (0, T), 
	\end{alignat}	
\end{subequations}
where 
\begin{equation*}
	\vec{\nu} = \begin{bmatrix}
	x \\ x + y
	\end{bmatrix}, 
	\quad 
	K_F(x,y) = \delta \sin(x)\cosh(y),
	\quad \text{and} \quad
	\vec{u}_0 (x,y) = \delta \begin{bmatrix}
		-\cos(x)\cosh(y) \\
		\sin(x)\sinh(y)
	\end{bmatrix}.
\end{equation*}
We take $\mu = 1$ and $\delta = 10^3$.
Note that we have assumed that we are in a regime where the inertial term 
$\vec{u} \cdot \nabla \vec{u}$ is negligible, due to the small velocity and/or 
high viscosity (after nondimensionalization). The fluid is initially at rest 
and the boundary conditions and scaling of $\mat{F}$ are chosen to 
drive the system to an equilibrium state corresponding to the same stress-free 
state in \cref{sec:polar-2d}.

The corresponding semi-discrete variational formulation with strongly imposed 
symmetry is:
For all $t \in (0, T)$, find $\mat{\sigma}_h(t) \in \Sigma^{\sym}_h$ and 
$\vec{u}_h(t) \in V_h$ such that
\begin{subequations}
	\label{eq:time-dependent-strong}
	\begin{alignat}{2}
		a(\mat{\sigma}_h(t), \mat{\tau}_h) + b(\mat{\tau}_h, \vec{u}_h(t)) 
			&= \langle \mat{\tau}_h \vec{n}, 
					\vec{u}_D(t) \rangle_{\partial \Omega} 
				+ (\mat{F}(t), \mat{\tau}_h)_{L^2(\Omega)} 
			\qquad & &\forall \mat{\tau}_h \in \Sigma_h^{\sym}, \\
		\label{eq:time-dependent-strong-2}
		b(\mat{\sigma}_h(t), \vec{v}_h) 
			- (\partial_t \vec{u}_h(t), \vec{v}_h) &= 0 
			\qquad & &\forall \vec{v}_h \in V_h,
	\end{alignat}	
\end{subequations}
where $\mat{F}(t)$ is the anisotropic forcing term on the RHS of 
\cref{eq:transient-polar-2}, $\vec{u}_D(t)$ is the boundary 
condition of the RHS of \cref{eq:transient-polar-3},  
$\mat{\sigma}_h(0) = \mat{0}$, and $\vec{u}_h(0) = \vec{0}$.
Similarly, the semi-discrete variational formulation with weakly imposed 
symmetry is:
For all $t \in (0, T)$, find $\mat{\sigma}_h(t) \in \Sigma_h$,  
$\vec{u}_h(t) \in V_h$, and $\mat{\omega}_h(t) \in \Xi_h$ such that 
\begin{subequations}
	\label{eq:time-dependent-weak}
	\begin{alignat}{2}
		a(\mat{\sigma}_h(t), \mat{\tau}_h) + b(\mat{\tau}_h, \vec{u}_h(t)) 
				+ c(\mat{\tau}_h, \mat{\omega}_h(t)) 
			&= \langle \mat{\tau}_h\vec{n}, 
					\vec{u}_D(t) \rangle_{\partial \Omega} 
				+ (\mat{F}(t), \mat{\tau}_h)_{L^2(\Omega)} 
			\qquad & &\forall \mat{\tau}_h \in \Sigma_h, \\
		b(\mat{\sigma}_h(t), \vec{v}_h) 
			- (\partial_t \vec{u}_h(t), \vec{v}_h)  &= 0 
			\qquad & &\forall \vec{v}_h \in V_h, \\		
		c(\mat{\sigma}_h(t), \mat{\xi}_h) &= 0 
			\qquad & &\forall \mat{\xi}_h \in \Xi_h,
	\end{alignat}
\end{subequations}
where $\mat{\sigma}_h(0) = \mat{0}$, $\vec{u}_h(0) = \vec{0}$, and 
$\mat{\omega}(0) = \mat{0}$. Since 
\cref{eq:time-dependent-strong,eq:time-dependent-weak} are systems of 
differential-algebraic equations, we discretize in time with a 2-stage 
Gauss-Radau IIA implicit Runge-Kutta method, automated by the 
\texttt{Irksome} package \cite{FarrellKirbyMM21,KirbyMaclachlan25},
with a time step size of $1/100$.
We choose the JMK scheme for \cref{eq:time-dependent-strong} and the 
$\mathrm{AFW}_1$ scheme for \cref{eq:time-dependent-weak}, both on 
an unstructured mesh of the unit square generated with
ngsPETSc with $\texttt{maxh} = 1/40$. The stage-coupled system 
is solved via sparse direct LU factorization using
MUMPS \cite{AmestoyDuff01}.

Snapshots of the magnitudes of the discrete stresses at 
$t \in \{0.5, 1, 1.5\}$ are displayed in \cref{fig:time-dependent}.
The discrete stress solutions for $t > 1.5$, not shown, are nearly identical
to the ones at $t=1.5$.
Until $t=1$, both solutions are visually identical. After $t=1$,
the material robust scheme $\mathrm{JMK}$ relaxes to a zero stress state,
the exact steady-state stress,
while the scheme $\mathrm{AFW}_1$ lacking material robustness relaxes to a 
solution that is far from a zero-stress state. Thus, one should also exercise 
caution when using schemes that are not material robust for transient problems.

\begin{figure}[htbp]
    \centering
	\begin{subfigure}[t][]{0.3\textwidth}
	    \includegraphics[width=0.9\textwidth]{%
			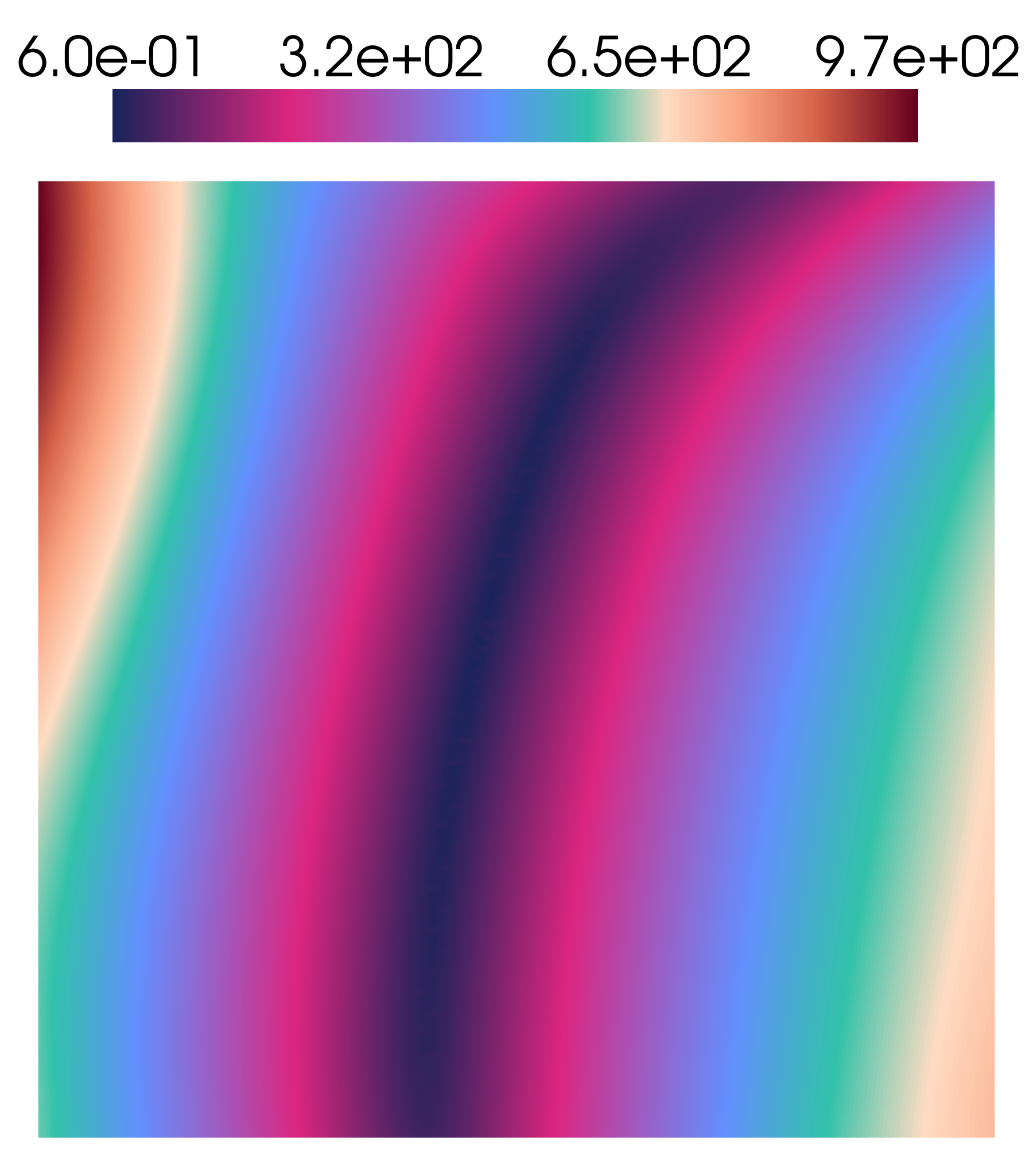}
	    \centering
	    \caption{$t=0.5,\; \mathrm{JMK}$}
    \end{subfigure}
    \begin{subfigure}[t][]{0.3\textwidth}
	    \includegraphics[width=0.9\textwidth]{%
			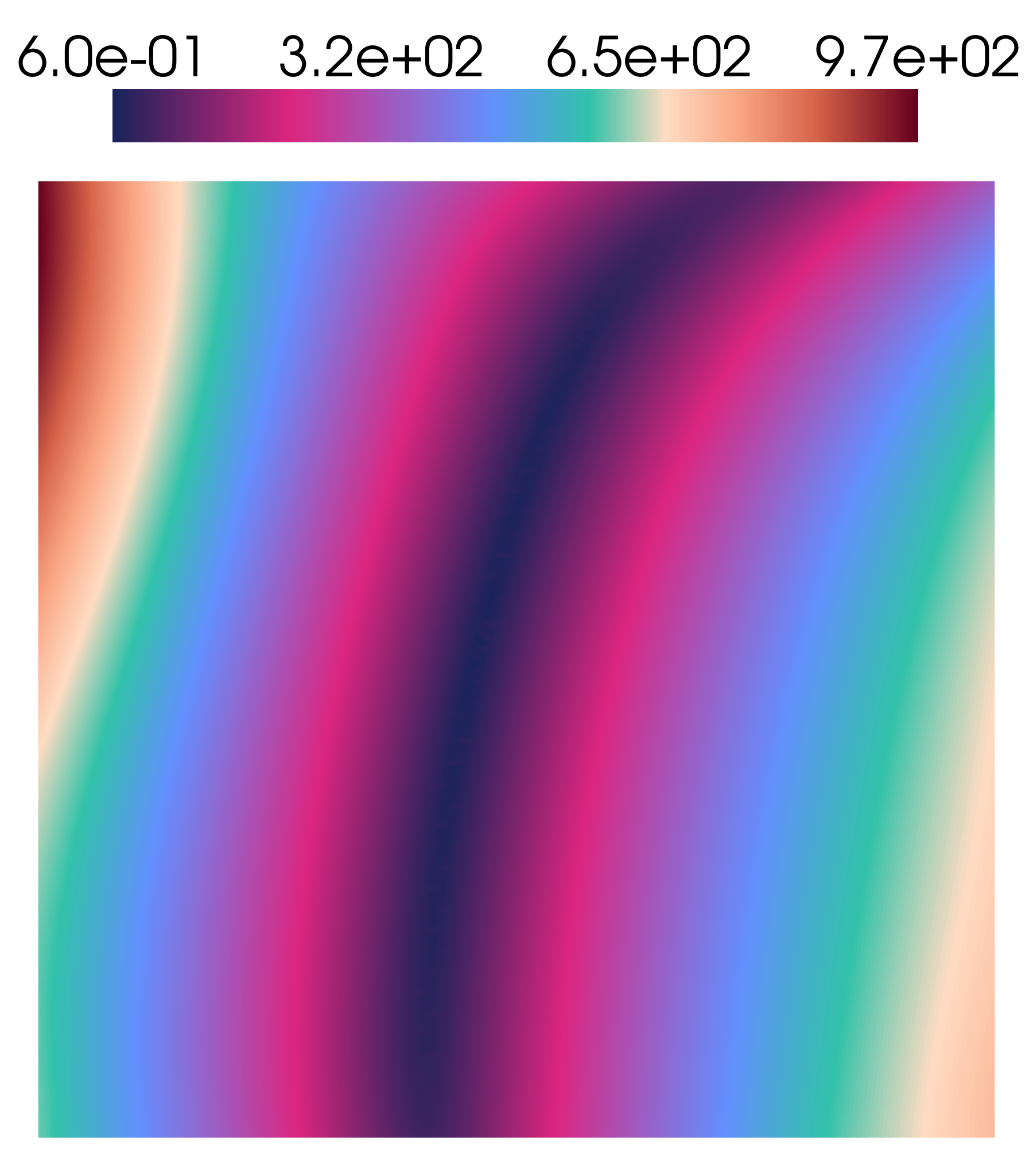}
	    \centering
	    \caption{$t=1.0,\; \mathrm{JMK}$}
    \end{subfigure}
    \begin{subfigure}[t][]{0.3\textwidth}
	    \includegraphics[width=0.9\textwidth]{%
			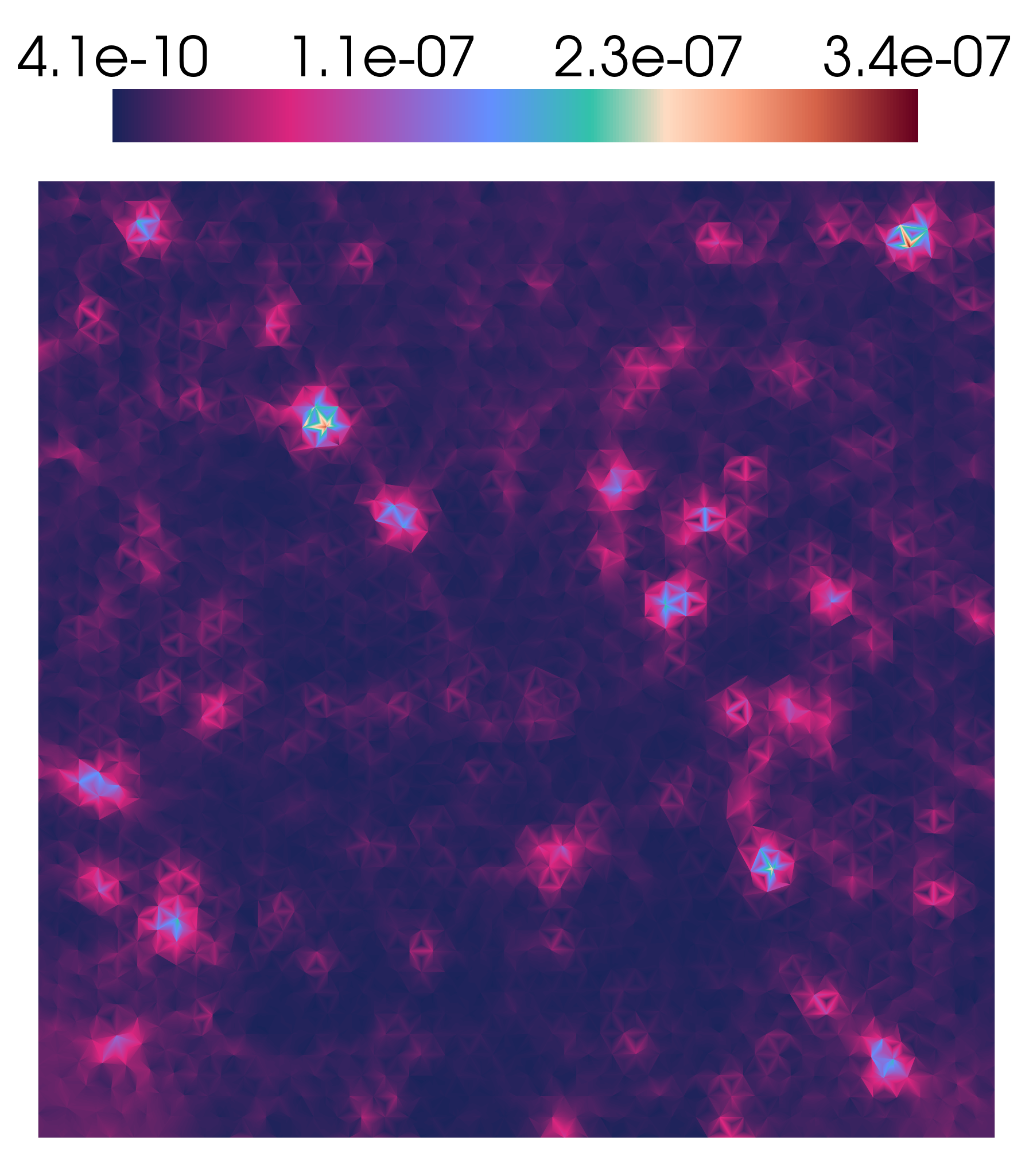}
	    \centering
	    \caption{$t=1.5,\; \mathrm{JMK}$}
    \end{subfigure} \\
    \begin{subfigure}[t][]{0.3\textwidth}
	    \includegraphics[width=0.9\textwidth]{%
			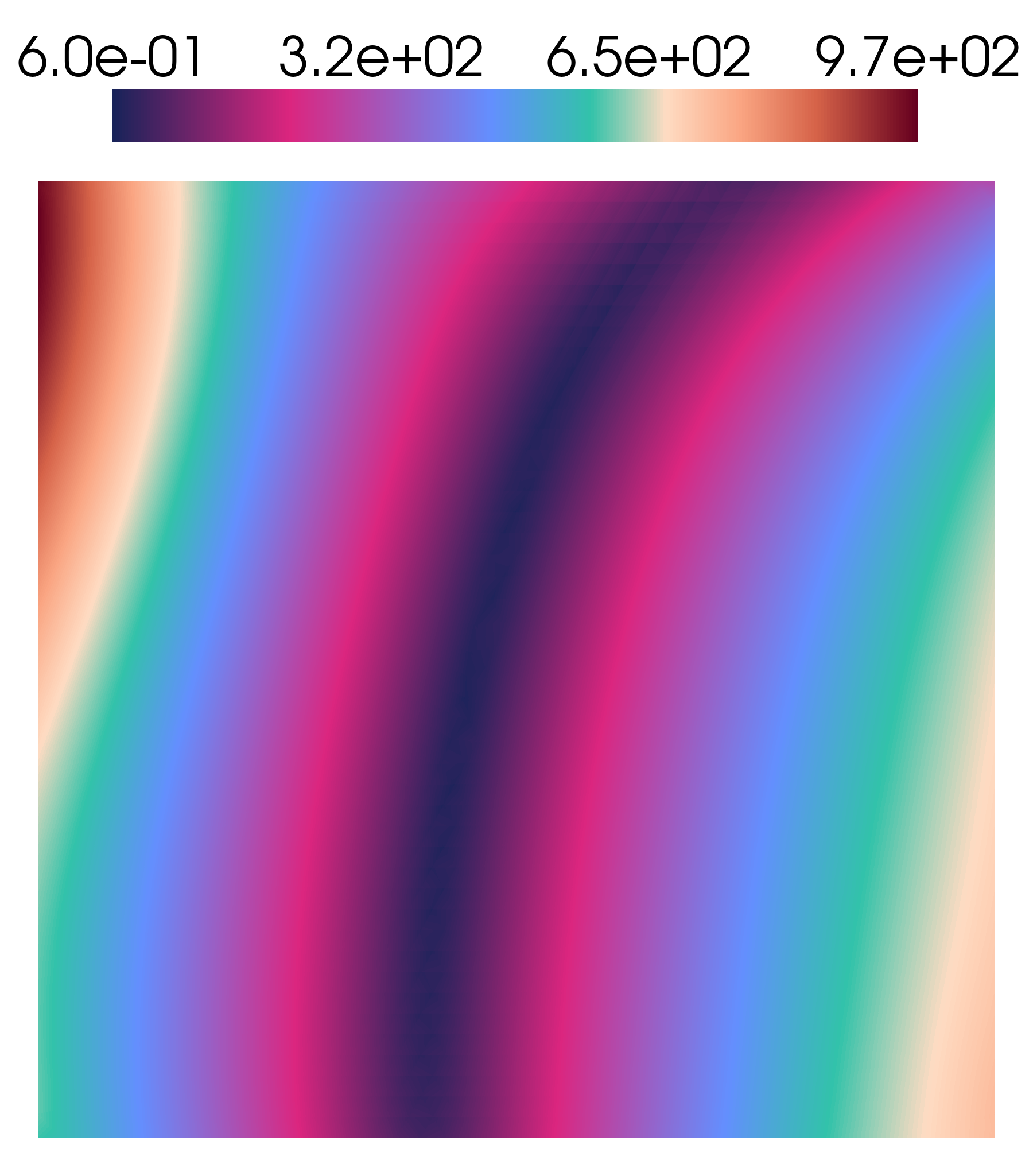}
	    \centering
	    \caption{$t=0.5,\; \mathrm{AFW}_1$}
    \end{subfigure}
	\begin{subfigure}[t][]{0.3\textwidth}
	    \includegraphics[width=0.9\textwidth]{%
			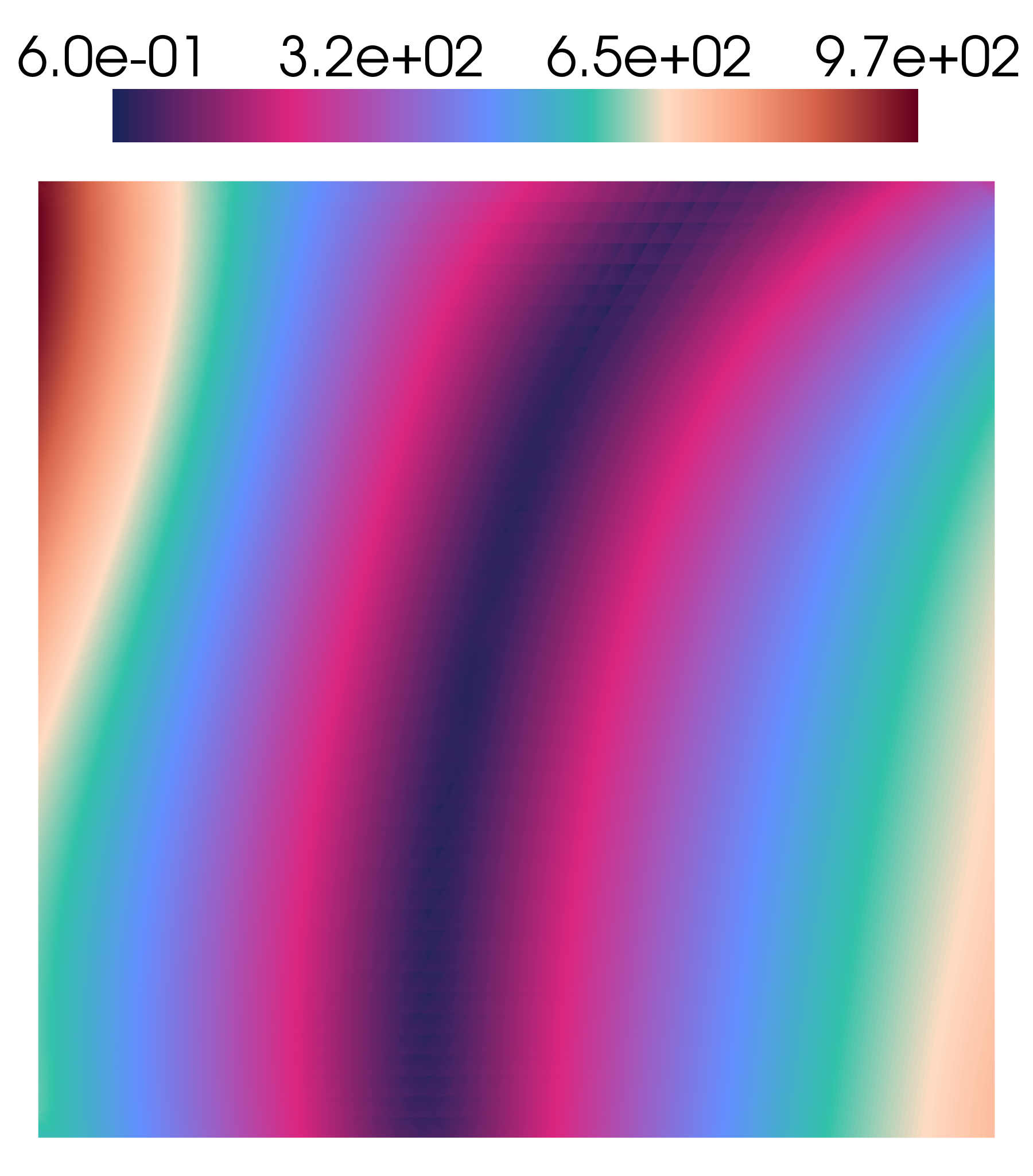}
	    \centering
	    \caption{$t=1.0,\; \mathrm{AFW}_1$}
    \end{subfigure}
	\begin{subfigure}[t][]{0.3\textwidth}
	    \includegraphics[width=0.9\textwidth]{%
			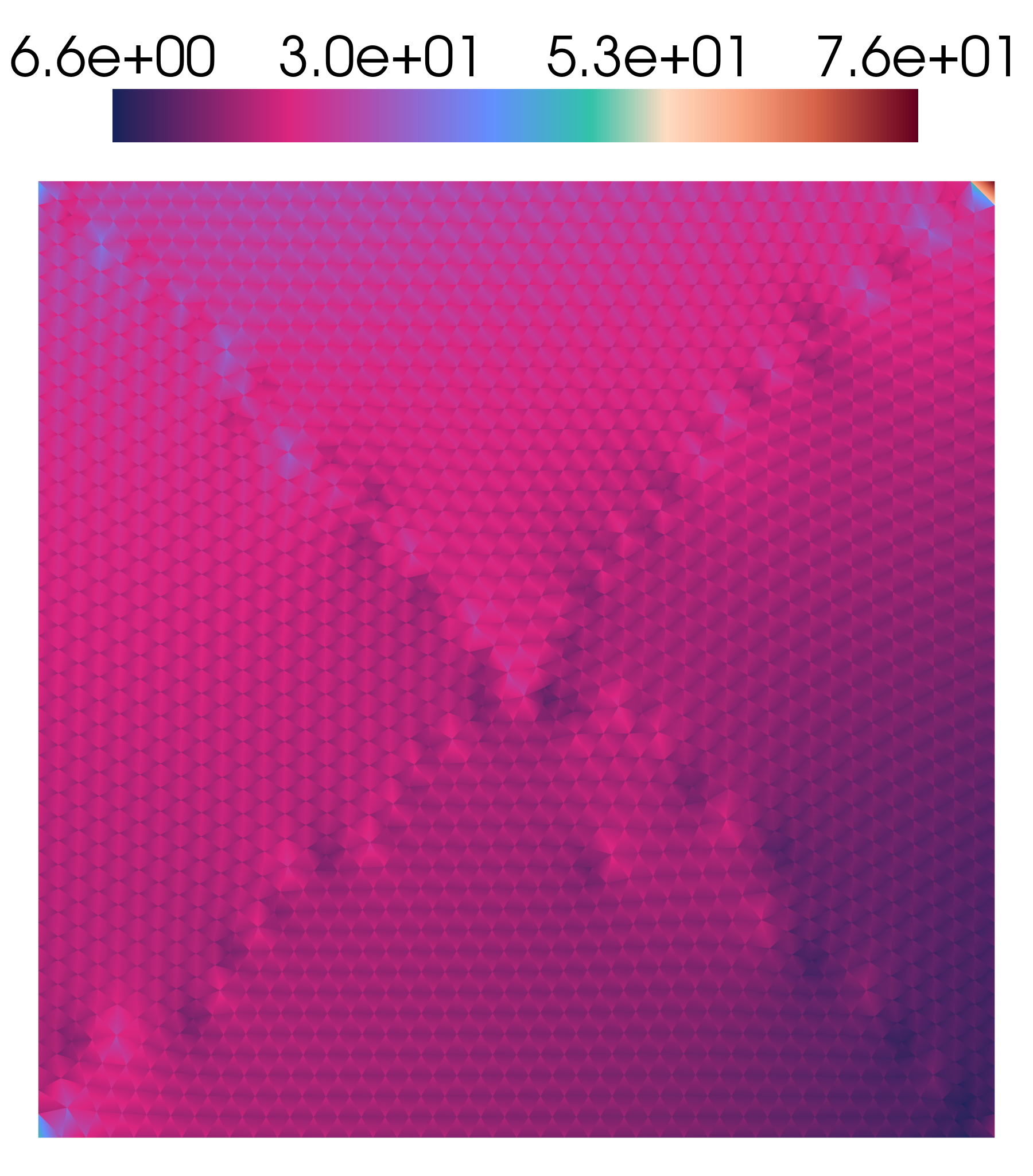}
	    \centering
	    \caption{$t=1.5,\; \mathrm{AFW}_1$}
    \end{subfigure}	
	\caption{Magnitudes of the stresses for 
	discretizations of a transient 2D polar fluid
	\cref{eq:transient-polar}. First row: the strongly 
	symmetric scheme JMK \cref{eq:time-dependent-strong}.
	Second row: the weakly symmetric scheme $\mathrm{AFW}_1$ 
	\cref{eq:time-dependent-weak}. Note that the schemes produce 
	visually identical stress magnitudes until $t=1$ (both rows have
	the same colorbar for these plots). After $t=1$, the 
	strongly symmetric scheme relaxes to a zero stress state,
	the exact steady state stress, while the weakly symmetric scheme 
	relaxes to a state far from zero. Note that the colorbars for 
	$t=1.5$ are different.}
	\label{fig:time-dependent}
\end{figure}

\appendix

\renewcommand{\thesection}{\Alph{section}} 
\makeatletter
\def\@seccntformat#1{\@ifundefined{#1@cntformat}%
	{\csname the#1\endcsname.\hspace{0.5em}}
	{\csname #1@cntformat\endcsname}}
\newcommand\section@cntformat{\appendixname\ \thesection.\hspace{0.5em}}
\makeatother

\section{Well-posedness of \cref{eq:linear_laws_weak_strong_sym} 
and \cref{eq:linear_laws_weak_weak_sym}}
\label{sec:well-posedness-mixed}

We briefly summarize the well-posedness of \cref{eq:linear_laws_weak_strong_sym} 
and \cref{eq:linear_laws_weak_weak_sym}. 
Using the notation of \cref{sec:theory},  note that the kernels of the 
$\mathcal{B}$ operators of \cref{eq:linear_laws_weak_strong_sym} 
and \cref{eq:linear_laws_weak_weak_sym} take the form
\begin{align*}
	\{ \mat{\tau} \in \Sigma^{\sym} : 
		 	(\div \mat{\tau}, \vec{v})_{L^2(\Omega)} = 0 
			\ \forall \vec{v} \in V \} 
		&= \{ \mat{\tau} \in  \Sigma^{\sym} : \div \mat{\tau} \equiv \vec{0} \},
	\\
	\{ \mat{\tau} \in \Sigma : 
			(\div \mat{\tau}, \vec{v})_{L^2(\Omega)} + 
			(\mat{\tau}, \mat{\xi} )_{L^2(\Omega)} = 0 
			\ \forall \vec{v} \in V, \ \forall \mat{\xi} \in \Xi \} 
		&= \{ \mat{\tau} \in  \Sigma^{\sym} : \div \mat{\tau} \equiv \vec{0} \}.
\end{align*}
The following result shows that the bilinear form $a(\cdot,\cdot)$ 
defined in \cref{eq:a-bilinear} is elliptic on the kernel.

\begin{lemma}
	\label{lem:a-elliptic-elasticity}
	For all $\mat{\tau} \in H(\div; \Omega, \mathbb{R}^{d \times d})$,
	there holds
	\begin{align}
		\label{eq:korn-deviatoric}
		\left\| \mat{\tau} \right\|_{L^2(\Omega)} 
		\leq C \left( \| \mat{\tau}^D \|_{L^2(\Omega)}  
			+ \| \div \mat{\tau}\|_{L^2(\Omega)}
			+ \left| \int_{\Omega} \tr \mat{\tau} \d{x} \right| \right).
	\end{align}
	Thus, for all $\mat{\tau} \in \Sigma$ with 
	$\div \mat{\tau} \equiv \vec{0}$, there holds 
	\begin{align}
		\label{eq:a-elliptic-elasticity}
		\frac{1}{2\mu} \| \mat{\tau}\|_{\div}^2 \leq
			C_{\lambda} a(\mat{\tau}, \mat{\tau}) \qquad \text{where} \quad 
			C_{\lambda} := \begin{cases}
				C & \text{if } \lambda = \infty, \\
				1 + \frac{d \lambda}{2\mu} & \text{if } 
				\lambda \in [0, \infty).
			\end{cases},
	\end{align}
	where $C$ is independent of $\mu$ and $\lambda$.
\end{lemma}
\begin{proof}
	Let $\mat{\tau} \in H(\div; \Omega, \mathbb{R}^{d \times d})$. Thanks
	to \cite[Proposition 9.1.1]{BoffiBrezziFortin13}, there holds
	\begin{align*}
		\left\| \mat{\tau} 
			- \int_{\Omega} \tr \mat{\tau} \d{x} \right\|_{L^2(\Omega)} 
		\leq C \left( \| \mat{\tau}^D \|_{L^2(\Omega)} 
			+ \| \div \mat{\tau}\|_{L^2(\Omega)} \right),
	\end{align*}
	and so \cref{eq:korn-deviatoric} follows from the triangle inequality.
	For $\lambda = \infty$, inequality \cref{eq:a-elliptic-elasticity} 
	follows from \cref{eq:korn-deviatoric} on recalling
	that if $\lambda = \infty$ and $\tau \in \Sigma$, then 
	$\int_{\Omega} \tr \mat{\tau} \d{x} = 0$. For $\lambda \in [0, \infty)$,
	we have 
	\begin{align*}
		\frac{1}{2\mu} \tau : \tau 
			= \frac{1}{2\mu} \tau^D : \tau^D 
				+ \frac{1}{2\mu d} (\tr \tau)^2 
			= \frac{1}{2\mu} \tau^D : \tau^D 
				+ \left(1 + \frac{d \lambda}{2\mu} \right) 
					\frac{1}{d(2\mu + d\lambda)} (\tr \tau)^2,
	\end{align*}
	and so \cref{eq:a-elliptic-elasticity} follows on integrating over $\Omega$.
\end{proof}
The remaining ingredients for well-posedness are the inf-sup conditions
\begin{align}
	\beta^{\sym} := \inf_{\vec{v} \in V} 
		\sup_{\mat{\tau} \in \Sigma^{\sym}}
		\frac{b(\mat{\tau}, \vec{v})}{ \| \mat{\tau} \|_{\div} 
			\|\vec{v}\| }
	 > 0 
	 \quad \text{and} \quad 
	 \beta := \inf_{\substack{\vec{v} \in V \\ \mat{\xi} \in \Xi}} 
		\sup_{\mat{\tau} \in \Sigma}
		\frac{b(\mat{\tau}, \vec{v}) + c(\mat{\tau}, \mat{\xi})}{ 
			\| \mat{\tau} \|_{\div} (\|\vec{v}\| + \|\mat{\omega}\|) }
	 > 0.
\end{align}
That $\beta^{\sym} > 0$ follows from the existence of a bounded right
inverse of the divergence operator, which we 
briefly outline here. Let $\vec{v} \in V$ and extend $\vec{v}$
to be zero outside $\Omega$. Let $B \supset \Omega$ be a sufficiently
large ball and let $\vec{u} \in \vec{H}^1_0(B)$ be the weak solution to 
following PDE:
\begin{align*}
	\div \symgrad(\vec{u}) = \vec{v} \quad \text{in } B 
	\quad \text{and} \quad \vec{u}|_{\partial B} = \vec{0}. 
\end{align*}
By elliptic regularity (see e.g. \cite[p. 263, Lemma 3.2]{Necas12}), 
$\vec{u} \in \vec{H}^2(B)$ with 
$\|\vec{u}\|_{H^2(B)} \leq C \|\vec{v}\|_{L^2(\Omega)}$, and so 
$\div \symgrad(\vec{u}) = \vec{v}$. Taking 
$\mat{\tau} = \symgrad(\vec{u})$ then satisfies
$\div \mat{\tau} = \vec{v}$ and 
$\|\mat{\tau}\|_{\div} \leq C\|\vec{v}\|$, and so $\beta^{\sym} > C^{-1}$. 
The proof that $\beta > 0$ is more involved and can be found in 
\cite[Proposition 9.3.2]{BoffiBrezziFortin13}.

In the literature, the assumption that $\vec{g} \equiv \vec{0}$ in 
\cref{eq:linear_u_dirichlet} is often made for linear elasticity problems
where also $\mat{F} \equiv \mat{0}$. In this case, one can test
\cref{eq:linear_laws_weak_strong_sym_1} or 
\cref{eq:linear_laws_weak_weak_sym_1} with $\mat{\tau} = \mat{\mathbb{I}}$
to see that $\int_{\Omega} \tr \mat{\tau} \d{x} = 0$. Many authors
therefore make the choice 
$\Sigma^{\sym} = H_*(\div; \Omega, \mathbb{R}_{\sym}^{d \times d})$
and $\Sigma = H_*(\div; \Omega, \mathbb{R}^{d \times d})$
so that the coercivity estimate in \cref{eq:a-elliptic-elasticity} 
is independent of $\lambda$.
The resulting schemes are then claimed to be ``locking-free,''
as the error estimates can be shown to be well-behaved as 
$\lambda \to \infty$. However, this choice of $\Sigma^{\sym}$
and $\Sigma$ only remains valid more generally if one assumes that
$\int_{\partial \Omega} \vec{g} \cdot \vec{n} \d{s} 
= -\int_{\Omega} \tr \mat{F} \d{x}$
for $\lambda < \infty$ (recall that this is precisely the compatibility
condition for $\lambda = \infty$). If $\vec{g}$ and $\mat{F}$ do
not satisfy this compatibility condition, then one must
take $\Sigma^{\sym}$ and $\Sigma$ as in \cref{eq:sigma-sym-v-def}
and \cref{eq:sigma-xi-def} and the constant 
$C_{\lambda}$ in the coercivity estimate \cref{eq:a-elliptic-elasticity}
blows up as $\lambda \to \infty$. Consequently, one obtains error estimates
that are not uniform in $\lambda$.

\section{Linear solvers used in \cref{sec:examples}}
\label{sec:linear-solvers}

To solve the linear systems arising from 
\cref{eq:linear_laws_weak_strong_sym_fem}
and \cref{eq:linear_laws_weak_weak_sym_fem}, which are symmetric,
we employ MINRES with a standard symmetric block-diagonal
preconditioning strategy. For the strongly symmetric elements
\cref{eq:linear_laws_weak_strong_sym_fem}, consider the bilinear form 
\begin{align}
	\label{eq:al-precon}
	a(\mat{\sigma}_h, \mat{\tau}_h) 
		+ \frac{\mathds{1}_{\lambda > 10^{10}}}{200 d \mu}
			(\tr \mat{\sigma}_h, \tr \mat{\tau}_h)_{L^2(\Omega)}
		+ \frac{\gamma}{2\mu} 
			(\div \mat{\sigma}_h, \div \mat{\tau}_h )_{L^2(\Omega)}
		+ \frac{2\mu}{\gamma} (\vec{u}_h, \vec{v}_h)_{L^2(\Omega)},
\end{align}
where $\mathds{1}$ is the indicator function and $\gamma$ is parameter
that we take to be $10^3$. Note that for $\lambda \leq 10^{10}$, this 
preconditioner corresponds to a standard augmented Lagrangian preconditioner
\cite{BenziOlshanskii06,Hiptmair96,Vassilevski92}. 
The additional term for  $\lambda > 10^{10}$ ensures
that form in invertible in floating point arithmetic for 
$\lambda \in (10^{10}, \infty]$. For the weakly symmetric
elements \cref{eq:linear_laws_weak_weak_sym_fem}, we add 
the term $2\mu(\mat{\omega}_h, \mat{\xi}_h)_{L^2(\Omega)}$ to
\cref{eq:al-precon}. The resulting preconditioner corresponds to a 
mix of augmented Lagrangian preconditioning and operator preconditioning 
\cite{Hiptmair06,Kirby10,MardalWinther11}. 
Instead of using \cref{eq:al-precon} directly, we apply 
one multigrid V-cycle with vertex patch preconditioner analogous to 
\cite{ArnoldFalkWinther00} for the stress terms and invert the 
displacement/velocity mass matrix directly.
We terminate MINRES when the preconditioned residual decreases by a factor of 
$10^{14}$.

For the case $\lambda = \infty$, we do not impose the constraint 
$\int_{\Omega} \tr \mat{\tau}_h \d{x} = 0$ in the implementation.
Consequently, the systems \cref{eq:linear_laws_weak_strong_sym_fem}
and \cref{eq:linear_laws_weak_weak_sym_fem} will have a one dimensional
nullspace and $\mat{\sigma}_h$ is unique up to a scalar multiple of 
$\mat{\mathbb{I}}$. 
Nevertheless, the preconditioner \cref{eq:al-precon} is
invertible, and MINRES still converges. In this case, we also monitor
the inexactness condition in \cite{LiuRoosta23} with tolerance $10^{-14}$.
We post-process the final MINRES iterate to
satisfy 
$\int_{\Omega} \tr \mat{\sigma}_h \d{x} = \int_{\Omega} \tr \mat{\sigma} \d{x}$.

\section{Pressure-robustness}
\label{sec:pressure-robustness}

As mentioned in \cref{sec:theory-general-comments}, the 
structure-preserving property \cref{eq:structure-preserving-def}
is typically called ``pressure robustness'' in the context of 
incompressible flow \cite{JohnLinkeMerdonNeilanRebholz17}. Here,
we show that one obtains similar numerical results as in 
\cref{sec:examples} for Example 1.1 in \cite{JohnLinkeMerdonNeilanRebholz17}.
In particular, consider the following Stokes problem on $\Omega = (0, 1)^2$: 
Find $\vec{u} \in H^1_0(\Omega; \mathbb{R}^2)$ and $p \in L^2_0(\Omega)$ such that 
\begin{subequations}
	\label{eq:stokes}
	\begin{alignat}{2}
		(\symgrad(\vec{u}), \symgrad(\vec{v}))_{L^2(\Omega)} 
			- (\div \vec{v}, p)_{L^2(\Omega)}  
				&= (\vec{f}, \vec{v})_{L^2(\Omega)}
			\qquad & &\forall \vec{v} \in H^1_0(\Omega; \mathbb{R}^2), \\
		-(\div \vec{u}, q)_{L^2(\Omega)} &= 0 
			\qquad & &\forall q \in L^2_0(\Omega),
	\end{alignat}
\end{subequations}
where the forcing term $\vec{f}$ is chosen such that the exact solution is 
given by
\begin{equation*}
      \vec{f} =
      \begin{pmatrix}
        0 \\
        \Ra(1-y+3y^2)
      \end{pmatrix}, \qquad
    \vec{u} = \begin{pmatrix}
        0 \\
        0
    \end{pmatrix}, \qquad
    p = \Ra\left(y^3-\frac{1}{2}y^2 + y - \frac{7}{12}\right)
\end{equation*}
for a positive parameter $\Ra > 0$. Here, $L^2_0(\Omega)$ denotes
mean-free $L^2(\Omega)$ functions. Note that 
\begin{align*}
	(\vec{f}, \vec{v})_{L^2(\Omega)} = -(\div \vec{v}, p)_{L^2(\Omega)},
\end{align*}
which is precisely of the form \cref{eq:saddle-perturbed} with 
$F \equiv 0$ and $G \equiv 0$. 

We mesh $\Omega$ as in \cref{sec:examples} with $\texttt{maxh}=1/8$.
We discretize \cref{eq:stokes} with the lowest-order Hood--Taylor (HT) 
element
$\mathcal{CG}_2(\mathcal{T}_h)^2 \times \mathcal{CG}_1(\mathcal{T}_h)$
and the quadratic Scott--Vogelius (SV) element
$\mathcal{CG}_2(\mathcal{T}_h^{\text{bary}})^2 
\times \mathcal{P}_1(\mathcal{T}_h^{\text{bary}})$,
where $\mathcal{T}_h^{\text{bary}}$ is the barycentric refinement of the mesh. 
Both elements are inf-sup stable, but the HT elements are not 
structure-preserving (pressure-robust), while the SV elements are 
structure-preserving.
We note that numerical results for the HT element also appear in 
\cite{JohnLinkeMerdonNeilanRebholz17}, but a structure-preserving 
method in the sense of \cref{eq:structure-preserving-def} was not considered 
for this example in \cite{JohnLinkeMerdonNeilanRebholz17}.
The direct LU solver MUMPS \cite{AmestoyDuff01} is used to invert the linear 
systems.
The numerical results for $\Ra \in \{ 10, 10^3, 10^5\}$
in \cref{fig:pressure-robustness} are analogous to the results in 
\cref{sec:examples}: The structure-preserving SV elements 
produce small velocity errors that are on the order of machine 
precision for $\Ra = 1$ and scale linearly in $\Ra$, while 
the HT elements produce nonnegligible velocity errors that also scale with $\Ra$,
but converge to zero quadratically as the mesh is refined. Note that 
the theory in \cref{sec:theory} applies here as well, and so the SV element
produces a zero velocity in exact arithmetic. However, the effect of roundoff 
errors and scaling the right-hand side of the linear system by $\Ra$ together 
produce the behavior in \cref{fig:pressure-robustness}.

\begin{figure}[htbp]
	\centering
	\begin{tikzpicture}
		\centering
		\begin{groupplot}[%
			group style={group size=3 by 2, 
				         horizontal sep=45pt, 
						 vertical sep=50pt},
			width=0.32\linewidth,
			height=0.32\linewidth,
			domain=2:5, 
			xtick={0,1,2,3,4,5}, 
			ymode=log, 
			xlabel={refinement level},
			ymajorgrids=true,
			grid style=dashed,
			every axis plot/.append style={line width=1.1pt},
			cycle list={
				{mark=*, alpha1, mark options={fill=alpha1}},
				{mark=triangle*, alpha2, mark options={fill=alpha2}},
				{mark=square*, alpha3, mark options={fill=alpha3}},
				{mark=diamond*, alpha4, mark options={fill=alpha4}},
				{mark=pentagon*, alpha5, mark options={fill=alpha5}}
			},
		]
		\nextgroupplot[ylabel={SV}]
		\addplot+[discard if not={Ra}{10.0}] 
				table [x=ref, y=velocity_error, col sep=comma] 
				{data/no_flow_sv.csv};
				\label{plot:no-flow-sv-vel-alpha1}
		\addplot+[discard if not={Ra}{1000.0}] 
				table [x=ref, y=velocity_error, col sep=comma] 
				{data/no_flow_sv.csv};
				\label{plot:no-flow-sv-vel-alpha2}
		\addplot+[discard if not={Ra}{100000.0}] 
				table [x=ref, y=velocity_error, col sep=comma] 
				{data/no_flow_sv.csv};
				\label{plot:no-flow-sv-vel-alpha3}
		\nextgroupplot
		\addplot+[discard if not={Ra}{10.0}] 
				table [x=ref, y=pressure_error, col sep=comma] 
				{data/no_flow_sv.csv};
				\label{plot:no-flow-sv-pre-alpha1}
		\addplot+[discard if not={Ra}{1000.0}] 
				table [x=ref, y=pressure_error, col sep=comma] 
				{data/no_flow_sv.csv};
				\label{plot:no-flow-sv-pre-alpha2}
		\addplot+[discard if not={Ra}{100000.0}] 
				table [x=ref, y=pressure_error, col sep=comma] 
				{data/no_flow_sv.csv};
				\label{plot:no-flow-sv-pre-alpha3}
		\nextgroupplot
		\addplot+[discard if not={Ra}{10.0}] 
				table [x=ref, y=divergence_error, col sep=comma] 
				{data/no_flow_sv.csv};
				\label{plot:no-flow-sv-div-alpha1}
		\addplot+[discard if not={Ra}{1000.0}] 
				table [x=ref, y=divergence_error, col sep=comma] 
				{data/no_flow_sv.csv};
				\label{plot:no-flow-sv-div-alpha2}
		\addplot+[discard if not={Ra}{100000.0}] 
				table [x=ref, y=divergence_error, col sep=comma] 
				{data/no_flow_sv.csv};
				\label{plot:no-flow-sv-div-alpha3}
		\nextgroupplot[ylabel={HT}]
		\addplot+[discard if not={Ra}{10.0}] 
				table [x=ref, y=velocity_error, col sep=comma] 
				{data/no_flow_ht.csv};
				\label{plot:no-flow-ht-vel-alpha1}
		\addplot+[discard if not={Ra}{1000.0}] 
				table [x=ref, y=velocity_error, col sep=comma] 
				{data/no_flow_ht.csv};
				\label{plot:no-flow-ht-vel-alpha2}
		\addplot+[discard if not={Ra}{100000.0}] 
				table [x=ref, y=velocity_error, col sep=comma] 
				{data/no_flow_ht.csv};
				\label{plot:no-flow-ht-vel-alpha3}
		\nextgroupplot
		\addplot+[discard if not={Ra}{10.0}] 
				table [x=ref, y=pressure_error, col sep=comma] 
				{data/no_flow_ht.csv};
				\label{plot:no-flow-ht-pre-alpha1}
		\addplot+[discard if not={Ra}{1000.0}] 
				table [x=ref, y=pressure_error, col sep=comma] 
				{data/no_flow_ht.csv};
				\label{plot:no-flow-ht-pre-alpha2}
		\addplot+[discard if not={Ra}{100000.0}] 
				table [x=ref, y=pressure_error, col sep=comma] 
				{data/no_flow_ht.csv};
				\label{plot:no-flow-ht-pre-alpha3}
		\nextgroupplot
		\addplot+[discard if not={Ra}{10.0}] 
				table [x=ref, y=divergence_error, col sep=comma] 
				{data/no_flow_ht.csv};
				\label{plot:no-flow-ht-div-alpha1}
		\addplot+[discard if not={Ra}{1000.0}] 
				table [x=ref, y=divergence_error, col sep=comma] 
				{data/no_flow_ht.csv};
				\label{plot:no-flow-ht-div-alpha2}
		\addplot+[discard if not={Ra}{100000.0}] 
				table [x=ref, y=divergence_error, col sep=comma] 
				{data/no_flow_ht.csv};
				\label{plot:no-flow-ht-div-alpha3}	
		\end{groupplot}	
		\node at ($(group c1r2.south) + (0,-2)$) 
			{(A) $\|\vec{u}-\vec{u}_h\|_{1}$};
		\node at ($(group c2r2.south) + (0,-2)$) {(B) $\|p-p_h\|$};
		\node at ($(group c3r2.south) + (0,-2)$) 
			{(C) $\|\div\vec{u}_h\|$};
	\end{tikzpicture}
	\caption{Numerical results for the Stokes problem \cref{eq:stokes}
	for $\mathrm{Ra} = 10$ (\ref*{plot:no-flow-ht-vel-alpha1}),
		$\mathrm{Ra} = 10^3$ (\ref*{plot:no-flow-ht-vel-alpha2}),
		and $\mathrm{Ra} = 10^5$ (\ref*{plot:no-flow-ht-vel-alpha3})
	with the SV (first row) and HT (second row) schemes.
	(A) the velocity errors in $H^1(\Omega; \mathbb{R}^2)$,
	(B) the pressure errors in $L^2(\Omega)$, and 
	(C) the divergence errors of the velocity in $L^2(\Omega)$.
	Observe that the non-structure-preserving scheme HT produces significantly 
	larger velocity errors in both (A) and (C) than the structure-preserving 
	SV scheme.}
	\label{fig:pressure-robustness}
\end{figure}

\bibliographystyle{elsarticle-num-names}
\bibliography{references}
\end{document}